\documentclass[a4paper, 11pt, reqno]{amsart}

\usepackage{amsmath, amsthm, amssymb}
\numberwithin{equation}{section}
\usepackage{xspace}
\usepackage{graphicx}
\usepackage{array}
\usepackage{braket}
\usepackage{multicol}
\usepackage{mathtools}
\usepackage{enumerate}
\usepackage{delarray}
\usepackage{mathtools}
\usepackage{comment}
\usepackage{faktor} 
\usepackage{mathrsfs}
\usepackage{tikz-cd}
\usepackage{hyperref}
\usepackage{cleveref}
\usepackage[normalem]{ulem} 
\usepackage[italicdiff]{physics} 
\usepackage{bbm} 
\usepackage{float}
\usepackage{stmaryrd} 
\usepackage{aligned-overset}
\usepackage{xcolor}
\usepackage[margin=15truemm]{geometry}
\usepackage{appendix}

\newtheorem{main}{Theorem}
\newtheorem{mcor}[main]{Corollary}

\theoremstyle{plain}
\newtheorem{thm}{Theorem}[section]
\newtheorem{lemma}[thm]{Lemma}
\newtheorem{corollary}[thm]{Corollary}
\newtheorem{prop}[thm]{Proposition}
\theoremstyle{definition}
\newtheorem{remark}[thm]{Remark}

\AddToHook{env/lemma/begin}{\crefalias{thm}{lemma}}
\AddToHook{env/corollary/begin}{\crefalias{thm}{corollary}}
\AddToHook{env/prop/begin}{\crefalias{thm}{prop}}
\AddToHook{env/remark/begin}{\crefalias{thm}{remark}}
\crefname{thm}{theorem}{theorems}
\crefname{lemma}{lemma}{lemmas}
\crefname{corollary}{corollary}{corollaries}
\crefname{prop}{proposition}{propositions}
\crefname{remark}{remark}{remarks}

\Crefname{thm}{Theorem}{Theorems}
\Crefname{lemma}{Lemma}{Lemmas}
\Crefname{corollary}{Corollary}{Corollaries}
\Crefname{prop}{Proposition}{Propositions}
\Crefname{remark}{Remark}{Remarks}

\newtheorem{claim}[thm]{Claim}
\theoremstyle{remark}

\theoremstyle{definition}

\newtheorem{exmp}[thm]{Example}
\AddToHook{env/exmp/begin}{\crefalias{thm}{example}}
\crefname{exmp}{example}{examples}
\Crefname{exmp}{Example}{Examples}

\newtheorem{defn}[thm]{Definition}
\newtheorem*{notation*}{Notation}

\newenvironment{subproof}[1][\proofname]{%
  \begin{proof}[#1]%
}{%
  \end{proof}%
}

\DeclarePairedDelimiterX{\inp}[2]{\langle}{\rangle}{#1, #2}

\makeatletter
\newcommand*{\bigcdot}{}
\DeclareRobustCommand*{\bigcdot}{%
  \mathbin{\mathpalette\bigcdot@{}}%
}
\newcommand*{\bigcdot@scalefactor}{.5}
\newcommand*{\bigcdot@widthfactor}{1.15}
\newcommand*{\bigcdot@}[2]{%
  \sbox0{$#1\vcenter{}$}
  \sbox2{$#1\cdot\m@th$}%
  \hbox to \bigcdot@widthfactor\wd2{%
    \hfil
    \raise\ht0\hbox{%
      \scalebox{\bigcdot@scalefactor}{%
        \lower\ht0\hbox{$#1\bullet\m@th$}%
      }%
    }%
    \hfil
  }%
}
\makeatother

\newcommand{\C}{\mathbb{C}}
\newcommand{\F}{\mathbb{F}}
\newcommand{\E}{\mathbb{E}}
\newcommand{\N}{\mathbb{N}\xspace}

\newcommand{\R}{\mathbb{R}\xspace}
\newcommand{\Z}{\mathbb{Z}\xspace}
\newcommand{\Q}{\mathbb{Q}\xspace}

\newcommand{\B}{\mathbb{B}\xspace}

\renewcommand{\S}{\mathbb{S}\xspace}
\newcommand{\T}{\mathbb{T}\xspace}
\newcommand{\K}{\mathbb{K}\xspace}
\newcommand{\X}{\mathbb{X}\xspace}
\newcommand{\Y}{\mathbb{Y}\xspace}

\DeclareMathOperator{\Ind}{Ind}

\DeclareMathOperator{\supp}{supp}

\let\Span\relax
\DeclareMathOperator{\Span}{Span}
\DeclareMathOperator{\Ran}{Ran}

\DeclareMathOperator{\Aut}{Aut}

\DeclareMathOperator{\id}{id}
\DeclareMathOperator{\nor}{nor}

\DeclareMathOperator{\Ad}{Ad}
\DeclareMathOperator{\Dom}{Dom}
\DeclareMathOperator{\Graph}{Graph}
\DeclareMathOperator{\ME}{ME}
\DeclareMathOperator{\mcpME}{mcpME}
\DeclareMathOperator{\Real}{Re}
\DeclareMathOperator{\conv}{conv}

\newcommand{\eps}{\varepsilon}

\newcommand{\acts}{\curvearrowright}

\newcommand\restr[2]{\ensuremath{\left.#1\right|_{#2}}}

\numberwithin{equation}{section}

\usepackage{enumitem}
\usepackage{color}
\newlist{steps}{enumerate}{1}
\setlist[steps, 1]{label = Step \arabic*:}

\hypersetup{%
  colorlinks=true,%
  linkcolor=blue,%
  citecolor=blue,%
  filecolor=blue,%
  menucolor=blue,%
  urlcolor=blue,%
  pdfnewwindow=true,%
  pdfstartview=FitBH
}   

\usepackage{etoolbox}
\makeatletter
\patchcmd{\@setaddresses}{\indent}{\noindent}{}{}
\patchcmd{\@setaddresses}{\indent}{\noindent}{}{}
\patchcmd{\@setaddresses}{\indent}{\noindent}{}{}
\patchcmd{\@setaddresses}{\indent}{\noindent}{}{}
\usepackage[hang,flushmargin]{footmisc}
\makeatother

\title{Properly proximal countable measured groupoids}

\author[Adriana Fernández Quero]{Adriana Fernández Quero}
\address{Department of Mathematics, KU Leuven, Celestijnenlaan 200B, 3001 Leuven, Belgium}\email{adriana.fernandeziquero@kuleuven.be}

\author[Kai Toyosawa]{Kai Toyosawa}
\address{Mathematical Institute, Department of Mathematics and Computer Science at the University of M\"unster, Einsteinstrasse 62, 48149 M\"unster, Germany}\email{kai.toyosawa@uni-muenster.de}

{\thanks{A.F.Q.\ was supported by the FWO research project G016325N of the Research Foundation Flanders.}}

{\thanks{K.T.\ was funded by the Deutsche Forschungsgemeinschaft (DFG, German Research Foundation) under Germany's Excellence Strategy EXC 2044/2 –390685587, Mathematics Münster: Dynamics–Geometry–Structure}}

\newcommand{\KK}{{\mathcal K}}

\newcommand{\BB}{{\mathcal B}}
\newcommand{\HH}{{\mathcal H}}
\newcommand{\RR}{{\mathcal R}}
\newcommand{\GG}{{\mathcal G}}
\newcommand{\II}{{\mathcal I}}
\newcommand{\PP}{{\mathcal P}}
\newcommand{\NN}{{\mathcal N}}
\newcommand{\ZZ}{{\mathcal Z}}
\newcommand{\UU}{{\mathcal U}}
\newcommand{\VV}{{\mathcal V}}
\newcommand{\SSS}{{\mathcal S}}
\newcommand{\CC}{{\mathcal C}}
\newcommand{\AAA}{{\mathcal A}}
\begin{document}

\begin{abstract} We introduce the notion of proper proximality for countable measured groupoids, extending the corresponding notion for countable groups. We prove that this groupoid property is equivalent to a strengthened form of relative proper proximality for the associated groupoid von Neumann algebra. We investigate permanence properties and show, in particular, that proper proximality is preserved under finite direct products, formation of transformation groupoids, and measure equivalence. We also identify broad classes of properly proximal groupoids, including transverse measured groupoids and free product groupoids, and prove that inner amenable groupoids are never properly proximal. 
\end{abstract}

\maketitle


\section{Introduction}

\noindent A countable discrete group $\Gamma$ is \textit{properly proximal} if there is no left $\Gamma $-invariant state on the Higson compactification \cite{Hig89}, also called the small-at-infinity boundary,
\begin{equation*}\label{eq:small-at-infinity group}
	\mathbb S(\Gamma)=\{f\in\ell^{\infty}(\Gamma):f-R_tf\in c_0(\Gamma)\text{ for every }t\in \Gamma\},
\end{equation*}
where $(R_t f)(s)=f(st)$. This notion was introduced by Boutonnet, Ioana, and Peterson \cite{BIP21} motivated by the study of the von Neumann algebras associated with the groups $\mathrm{SL}_d(\mathbb Z)$, $d\geq 3$, and their probability measure-preserving actions. Among their applications, they obtained the first $W^*$-rigidity results for compact actions of these groups. 

Proper proximality may be viewed as a weakening of the notion of biexactness introduced by Ozawa \cite{Ozawa-solid} which has been shown to exhibit many operator-algebraic rigidity consequences \cite{Pop08,OzawaPopa07,Hou09, CH10,chifansinclair,popavaes12,Ioa12,Vae13,CKP16,CdSS18,Dri20,Pat23,DT25,CTY26}.
The von Neumann algebras associated with properly proximal groups share several rigidity properties with those arising from biexact groups. If $\Gamma$ is properly proximal, then $L\Gamma$ does not have property (Gamma). More generally, $L\Gamma$ admits no diffuse regular von Neumann subalgebra $A$ for which the action $\mathcal N_{L\Gamma}(A)\curvearrowright A$ is weakly compact \cite{BIP21}. Unlike solidity for biexact groups, however, these rigidity properties are not automatically inherited by arbitrary nonamenable von Neumann subalgebras.

The class of properly proximal groups is broad. Besides the groups $\mathrm{SL}_d(\mathbb Z)$, $d\geq 3$, it contains non-elementary relatively hyperbolic groups, non-elementary convergence groups, lattices in noncompact semisimple Lie groups with finite center, and groups admitting a proper cocycle into a nonamenable unitary representation \cite{BIP21,HHL20}. Further examples include many groups acting on trees and graph products, as well as wreath products $A\wr B$ with $A\neq\{e\}$ and $B$ nonamenable \cite{dingKEexamples}. In fact, for a nontrivial group $A$, the wreath product $A\wr B$ is properly proximal precisely when $B$ is nonamenable. By contrast, inner amenability provides a general obstruction.

In \cite{IPR19}, Ishan, Peterson, and Ruth proved that proper proximality is preserved under von Neumann equivalence, and hence under both measure equivalence and $W^*$-equivalence. In particular, if $\Gamma$ and $\Lambda$ are countable groups such that $L\Gamma\simeq L\Lambda$, then $\Gamma$ is properly proximal if and only if $\Lambda$ is properly proximal. This result suggested that proper proximality should admit an intrinsic von Neumann algebraic formulation. Such a characterization was established by Ding, Kunnawalkam Elayavalli and Peterson in \cite{ding2023properproximality}. Given a finite von Neumann algebra $M$ and a von Neumann subalgebra $B\subset M$, they introduced a small-at-infinity boundary $\mathbb S_B(M)$ relative to $B$, an operator system in $\mathbb B(L^2M)$ containing $M$. They defined $M$ to be \textit{proper proximal relative to $B$} if $\mathbb S_B(M)$ admits no $M$-central state whose restriction to $M$ is normal. When $B=M$, this yields the notion of a properly proximal von Neumann algebra. 
In \cite{ding2023properproximality}, the authors further proved that a countable discrete group $\Gamma$ is properly proximal if and only if its group von Neumann algebra $L\Gamma$ is properly proximal \cite[Theorem 6.4]{ding2023properproximality}. 

The purpose of this paper is to develop an analogous theory for countable p.m.p.\ groupoids. This setting simultaneously encompasses countable discrete groups, viewed as groupoids with one unit, and countable p.m.p.\ equivalence relations, viewed as principal groupoids. We introduce the small-at-infinity boundary $\mathbb S(\mathcal G)$ of a countable p.m.p.\ groupoid $(\mathcal G,\nu)$ and define proper proximality through the absence of $[\mathcal G]$-invariant states whose restriction to the unit-space algebra $L^\infty(\mathcal G^{(0)},\mu)$ is normal. Our main result identifies this groupoid notion with a strengthened form of relative proper proximality for (twisted) groupoid von Neumann algebras. This form of relative proper proximality requires the lack of $L\mathcal G$-central states on $\mathbb S_{L^\infty \GG^{(0)}}(L\GG)\cap \langle L\Gamma,e_{L^\infty \GG^{(0)}}\rangle$ that are normal on $L\GG$, see Definition \ref{def: A-properlyproximal rel A}.

\begin{main}[{\Cref{thm:properproxR}}]\label{main:ppvN} Let $(\GG,\nu)$ be a countable p.m.p.\ groupoid with unit space $(\mathcal G^{(0)},\mu)$, and let $c\in Z^2(\GG,\T)$ be a $2$-cocycle. Then $(\GG,\nu)$ is properly proximal if and only if $L_c\GG$ is $L^\infty \GG^{(0)}$-properly proximal relative to $L^\infty \GG^{(0)}$.
\end{main}

Our definition agrees with the classical notion of proper proximality for groups if we regard groups as groupoids with singleton unit spaces. More generally, the next result shows that proper proximality is preserved and reflected by passing from an essentially free p.m.p.\ action to its transformation groupoid.

\begin{main}[\Cref{prop: transf preserv prop prox}]\label{main:transformationgroupoids} 
Let $\Gamma\curvearrowright (X,\mu)$ be an essentially free p.m.p.\ action. Then $\Gamma$ is properly proximal if and only if the transformation groupoid $\Gamma\ltimes X$ is properly proximal.
\end{main}

The groupoid formulation nevertheless contains more information than the proper proximality of the ambient von Neumann algebra. For a transformation groupoid $\GG=\Gamma\ltimes X$, it remembers the inclusion $L^\infty X\subset L\GG = L^\infty X\rtimes\Gamma$. Accordingly, \Cref{main:ppvN} characterizes proper proximality of $\GG$ in terms of $L^\infty X$-proper proximality of this inclusion relative to $L^\infty X$, which is stronger than proper proximality of the von Neumann algebra $L\GG$ relative to the boundary piece associated with $L^\infty X$.

The difference can already be seen for wreath products. Let $B$ be a nonamenable countable group that is not properly proximal; for example, one may take $B=\mathrm{SL}_3(\F_p[t^{-1}])\ltimes\mathbb F_p[t,t^{-1}]^3$ as in \cite{IPR19}. Let $\Gamma=(\Z/2\Z)\wr B =(\Z/2\Z)^{(B)}\rtimes B$ and $X=\prod_B\Z/2\Z$. The group $\Gamma$ is properly proximal by \cite[Theorem 1.7]{dingKEexamples}. Hence $L\Gamma$ is properly proximal and, in particular, is properly proximal relative to the boundary piece associated with $L((\Z/2\Z)^{(B)})\subset L\Gamma$. On the other hand, under the Fourier transform on the abelian group $(\Z/2\Z)^{(B)}$, we have $L\Gamma\simeq L^\infty X\rtimes B$, where the subalgebra $L((\Z/2\Z)^{(B)})$ is identified with $L^\infty X$. Since the Bernoulli action $B\curvearrowright X$ is essentially free and $B$ is not properly proximal, \Cref{main:transformationgroupoids} implies that the transformation groupoid $\GG=B\ltimes X$ is not properly proximal and therefore, by \Cref{main:ppvN}, $L\mathcal G=L^\infty X\rtimes B\simeq L\Gamma$ is not $L^\infty X$-properly proximal relative to $L^\infty X$. Thus, for the same von Neumann algebra and the same boundary piece, $L\Gamma$ is properly proximal relative to $L^\infty X$, but it is not $L^\infty X$-properly proximal relative to $L^\infty X$. Equivalently, the full operator system $\mathbb S_{L^\infty X}(L\Gamma)$ admits no $L\Gamma$-central state whose restriction to $L\Gamma$ is normal, whereas the smaller operator system $ \S_{L^\infty X}(L\Gamma)\cap \langle L\Gamma,e_{L^\infty X}\rangle$ does admit such a state. 
This 
condition is the one that detects the inclusion $L^\infty X\subset L^\infty X\rtimes B$ and hence the underlying groupoid structure.

Our main result, \Cref{main:ppvN}, also yields a groupoid analogue of a rigidity consequence of proper proximality established by Boutonnet, Ioana, and Peterson \cite[Theorems 1.1, 1.4]{BIP21}. Namely, if $\Gamma$ is properly proximal, then $L\Gamma$ does not admit any weakly compact Cartan subalgebra, and more generally, in a crossed product $L^{\infty}X\rtimes \Gamma$, if such weakly compact subalgebra exists, it must be $L^{\infty}X$ up to unitary conjugacy. This result was formulated for general properly proximal von Neumann algebras in \cite[Theorem 6.11]{ding2023properproximality}. Combining the aforementioned theorem with \Cref{main:ppvN}, we obtain the following groupoid analogue.

\begin{mcor}[\Cref{cor:weakcompact&properprox}]
    If $\GG$ is a properly proximal countable p.m.p.\ groupoid and $P\subset L\GG$ is a weakly compact regular von Neumann subalgebra of $L\GG$, then $P\preceq_{L\GG} L^\infty \GG^{(0)}$.

    In particular, if $(\RR,\nu)$ is a properly proximal countable measured equivalence relation on $(X,\mu)$, then $L\RR$ admits a weakly compact Cartan subalgebra $A$ if and only if $\RR$ is weakly compact and, in this case, $A$ is unitarily conjugate to $L^\infty X$.
\end{mcor}

The class of properly proximal groups is stable under commensurability up to finite kernels, under direct products, and under measure-equivalence \cite{BIP21,IPR19}. Our next main result establishes groupoid analogues of these permanence properties. We refer readers to Appendix \ref{sec: appendix} for the definition and detailed explanation for (measure class preserving) measure equivalence of countable p.m.p.\ groupoids.

\begin{main}\label{main:permanence} Proper proximality for countable p.m.p.\ groupoids is preserved under 
    \begin{enumerate}
        \item inflations and restrictions (\Cref{restriction}),
        \item finite index inclusions (\Cref{prop: finite index subgpoid,prop:f.i.subgroupoid}),
        \item co-amenable subrelations (\Cref{prop:coamenable}),
        \item finite products (\Cref{prop:productsG1G2}),
        \item transformation groupoids (\Cref{prop: transf preserv prop prox}), and
        \item (measure class preserving) measure equivalence (\Cref{thm:MEpp}).
    \end{enumerate}
\end{main}

We also study the behavior of proper proximality under ergodic decomposition of countable p.m.p.\ groupoids in \Cref{thm:ergodicdecomp0}. We show that properly proximal ergodic groupoid fibers $(\mathcal G_z)_z$ give rise to a properly proximal groupoid $\mathcal G$ whenever the unit space $\mathcal G^{(0)}$ is completely atomic, and we prove the reverse direction without needing this additional assumption.

Our final main contribution is to establish proper proximality, and its failure, for several natural classes of measured groupoids. We consider transverse measured groupoids, introduced in \cite{bddcohomologytransversegroupoid} as discrete measured models for p.m.p.\ actions of locally compact second countable groups, nondegenerate free products of measured groupoids, and inner amenable measured groupoids introduced in \cite{Kida-TuckerDrob}. These last two results generalize \cite[Corollary 1.3]{dingKEexamples} and \cite[Proposition 4.11]{BIP21}, respectively, to the groupoid setting. We further produce examples of properly proximal equivalence relations, including ones that cannot be realized as the orbit equivalence relation of an essentially free p.m.p. action.

\begin{main}\label{main:examples} Countable p.m.p.\ groupoids in the following classes are properly proximal:
    \begin{enumerate}
        \item transverse measured groupoids (\Cref{cor:transverse}), 
        \item nondegenerate free products (\Cref{prop:freeproducts}),
        \item nonamenable treeable equivalence relations (\Cref{prop: nonamen treeable equiv rel}), and
        \item irrational compression of higher-rank lattice relations (\Cref{exmp: restr SLn action}).
    \end{enumerate}
    In contrast, inner amenable countable p.m.p.\ groupoids are never properly proximal (\Cref{prop:inneramnotproperprox}).
\end{main}

\vspace{2mm}

\noindent\textbf{Organization of the paper.} Aside from the introduction, this paper contains five sections and an appendix. In \Cref{sec:groupoids}, we recall the basics of countable p.m.p.\ groupoids, including their actions and the associated von Neumann algebras. \Cref{sec:defpp} introduces the small-at-infinity boundary associated to a groupoid, provides the definition of proper proximality, and gives some characterizations. \Cref{sec:permanence} contains the proof of \Cref{main:permanence}, and \Cref{sec:examples} contains the proof of \Cref{main:examples}. Appendix \ref{sec: appendix} provides the background on measure equivalence for groupoids required for \Cref{thm:MEpp} in \Cref{sec:permanence}. Finally, the proof of \Cref{main:ppvN} and further von Neumann algebraic consequences for groupoid von Neumann algebras appear in \Cref{sec: invariant vNa}.

\vspace{2mm}

\noindent\textbf{Acknowledgements.} The authors thank Jesse Peterson and David Jekel for many helpful comments.

This project started during KT's visit to KU Leuven in November 2025. KT acknowledges the Department of Mathematics at KU Leuven for its hospitality during this visit. The collaboration continued during the conference ``Ergodic Group Actions and Unitary Representations III” hosted by the IMPAN in Warsaw and during AFQ's visit to Universität Münster in May 2026. For the conference, the authors acknowledge the partial support of the Simons Foundation grant (award no. SFI-MPS-T-Institutes-00010825) and from the State Treasury funds commissioned by the Minister of Science and Higher Education under the project ``Organization of the Simons Semesters at the Banach Center - New Energies in 20206-2028" (agreement no. MNiSW/2025/DAP/491); and for the visit to Münster, AFQ would like to express her gratitude to David Kerr for his hospitality. 

\vspace{2mm}

\noindent\textbf{AI disclosure.} ChatGPT 5.6 was used for language editing, proofreading, and literature searching. All mathematical ideas, arguments, and proofs presented in the article were developed entirely by the authors. 

\section{Countable measure-preserving measured groupoids}\label{sec:groupoids}

\noindent In this section we recall the basic notions and facts concerning countable measured groupoids that will be used throughout the paper.

A \emph{groupoid} $\GG$ is a small category where all arrows are invertible, i.e., a set equipped with partially defined multiplication together with an inverse operation satisfying the natural properties such as associativity. We denote the \emph{unit space} by $\GG^{(0)}$ or $X$ interchangeably, the latter notation motivated by the countable equivalence relations and transformation groupoids, which are the major examples of groupoids we consider in this paper. The \emph{source} and \emph{range} maps are written $s,r\colon \GG \to \GG^{(0)}$, respectively, and multiplication is defined on the set of composable pairs $\GG^{(2)}\coloneqq \{(\gamma,\eta)\in \GG\times \GG : s(\gamma)=r(\eta)\}$. For subsets $A,B\subset \GG$, we define their product $AB\coloneqq \{\gamma\eta:\gamma\in A,\eta\in B,s(\gamma)=r(\eta)\}$; that is, the multiplication of elements from $A$ and $B$ whenever well-defined.

A groupoid $\GG$ is called a \emph{countable Borel groupoid} if $\GG$ is equipped with a Borel space structure such that the unit space $\GG^{(0)}$ is a Borel subset of $\GG$, the maps $s,r \colon \GG \to \GG^{(0)}$ are Borel and countable-to-one, and the multiplication and inverse maps are Borel. When moreover the unit space $\GG^{(0)}$ admits a probability measure $\mu$, we consider two measures defined on $\GG$ given by
\[\mu_1(A)=\int_{\GG^{(0)}} |r^{-1}(x)\cap A|\, \mathrm{d}\mu(x) \:\:\text{ and }\:\: \mu_2(A)=\int_{\GG^{(0)}} |s^{-1}(x)\cap A|\,\mathrm{d}\mu(x),\]
for $A\subset\GG$. When these two measures coincide, we denote by $\mu\circ \lambda \coloneqq \mu_1 = \mu_2$ and call the pair $(\GG,\mu\circ\lambda)$ a \emph{countable p.m.p.\ (probability measure preserving) groupoid over $(\GG^{(0)},\mu)$}, where $\lambda$ denotes the counting measure on the fibers. We will usually use the notation $\nu=\mu\circ\lambda$. In this paper we only consider countable p.m.p.\ groupoids. Note that by definition, $\restr{\mu\circ\lambda}{\GG^0} = \mu$.

For each subset $A\subset \GG^{(0)}$, let $\GG_A \coloneqq s^{-1}(A)= \{\gamma\in \GG:s(\gamma) \in A\}$ and $\GG^A \coloneqq r^{-1}(A)= \{\gamma\in \GG:r(\gamma) \in A\}$. For each $A,B\subset \GG^{(0)}$, let $\GG^B_A \coloneqq \GG_A\cap \GG^B$, and write $\GG^A_A=\GG_A\cap \GG^A$ when $B=A$. If $A\subset \GG^{(0)}$ is a non-null Borel subset, then $\GG^A_A$ is a countable p.m.p.\ groupoid over $(A,\frac{1}{\mu(A)}\mu|_A)$ called the \emph{restriction} of $\GG$ to $A$. We say $(\GG,\nu)$ is \emph{aperiodic} if $\GG_x\coloneqq\GG_{\{x\}}$ is infinite for $\mu$-almost every $x\in\GG^{(0)}$, and \emph{principal} if $|\GG_x^x| = 1$ for $\mu$-almost every $x\in\GG^{(0)}$.

We associate to $\GG$ an equivalence relation $\RR_\GG$ on $\GG^{(0)}$ defined by $\RR_\GG = \{(r(\gamma),s(\gamma)):\gamma\in \GG\}$. A Borel subset $A \subset \GG^{(0)}$ is called \emph{$\GG$-invariant} if $r(\GG_x) \subset A$ for $\mu$-almost every $x \in A$. The groupoid $(\GG,\nu)$ is \emph{ergodic} if every $\GG$-invariant Borel subset is null or conull, or equivalently, the equivalence relation $\RR_\GG$ is ergodic. A \emph{complete unit section} of $\GG$ is a subset $A$ of $\GG^{(0)}$ satisfying $r(\GG_A)=\GG^{(0)}$ modulo a null set, equivalently, $A\subset\GG^{(0)}$ meets every equivalence class of the equivalence relation $\RR_{\GG}$.

A \emph{(local) bisection} of $\GG$ is a Borel subset $B\subset \GG$ such that both $s$ and $r$ are injective on $B$. The sets $s(B)$ and $r(B)$ are called the \emph{source} and \emph{range} of $B$, respectively. The set of all local bisections is called the \emph{full pseudogroup} of $\GG$, denoted by $[[\GG]]$. We say a local bisection $B$ is a \emph{global bisection} if $s(B)$ and $r(B)$ coincide with $\GG^{(0)}$ up to null sets, and we let $[\GG]$ be the \emph{full group} of $(\GG,\nu)$ consisting of global bisections.

For each local bisection $B\subset \GG$, left multiplication by $B$ determines uniquely a partial Borel isomorphism $\theta_B \colon \GG^{s(B)} \to \GG^{r(B)}$ up to null sets, so we will identify $B$ with $\theta_B$ whenever no confusion arises, and even drop the subscript $B$ in $\theta_B$ if the choice of the local bisection $B$ is clear from the context. For example, we will denote by $s(\theta)\coloneqq \Dom(\theta) = s(B)$ and $r(\theta) \coloneqq \Ran(\theta) = r(B)$. From this perspective, the full pseudogroup $[[\GG]]$ is identified with the set of partial Borel isomorphisms $\theta$ on $\GG$ such that $\Dom(\theta),\Ran(\theta)$ are preimages of subsets of $\GG^{(0)}$ under the range map $r$ and $\theta(\gamma)\gamma^{-1} = \theta(\eta)\eta^{-1}$ whenever $\gamma,\eta\in \Dom(\theta)$ with $r(\gamma) = r(\eta)$. For each $\theta\in [[\GG]]$, we let $s_\theta^{-1}$ and $r_\theta^{-1}$ denote the inverse maps of the restrictions $\restr{s}{B}, \restr{r}{B}$ of the corresponding local bisection $B$.  

The full pseudogroup $[[\GG]]$ admits a left action and a right action on $L^\infty(\GG,\nu)$. For fixed $f\in L^\infty(\GG,\nu)$ and $\theta \in [\GG]$, we define the left action $L_\theta$ of $\theta$ on $f$ by
\[(L_\theta f)(\gamma) = \begin{cases} f([r_\theta^{-1}(r(\gamma))]^{-1}\gamma) & \text{if }r(\gamma)\in r(\theta), \\
0 & \text{otherwise}.\end{cases}\]
Equivalently, if we identify $\theta$ with the corresponding Borel section $B$, then $(L_\theta f)(\gamma) = f(\theta^{-1}\gamma) = f(B^{-1}\gamma)$. We similarly define the right action $R_\theta$ of $\theta$ on $f$ by
\[(R_\theta f)(\gamma) = \begin{cases} f(\gamma [r_\theta^{-1}(s(\gamma))]) & \text{if }s(\gamma)\in r(\theta), \\
0 & \text{otherwise},\end{cases}\]  
which can also be written as $(R_\theta f)(\gamma) = f(\gamma\theta) = f(\gamma B)$.

\subsection{Measured equivalence relations}

Let $(X,\mu)$ be a standard probability space. A \emph{countable Borel equivalence relation} on $X$ is a Borel subset $\RR\subset X\times X$ such that each equivalence class $[x]_\RR \coloneqq \{ y \in X : (x,y) \in \RR\}$ is countable for $\mu$-almost every $x\in X$. The set $\RR$ admits a natural structure of a countable Borel groupoid with unit space $X$ identified with the diagonal subset $\Delta\coloneqq \{(x,x):x\in X\}$, where the range and source maps are given by $r(x,y)=x, s(x,y)=y$ for every $(x,y)\in \RR$, respectively. The multiplication is $(x,y)(y,z)=(x,z)$ and inversion map is $(x,y)^{-1}=(y,x)$.

For a subset $A \subset X$, its \emph{saturation} (or \emph{$\RR$-saturation}) is defined by $\RR(A)\coloneqq \{ y \in X : \text{ there exists } x \in A \text{ with } (y,x) \in \RR \} = r(s^{-1}(A))$. 
The measure $\mu$ on $X$ is said to be \emph{quasi-invariant with respect to $\RR$} if $\nu(\RR(A)) = 0$ whenever $\mu(A) = 0$ with $A\subset X$ and the inversion map preserves the measure class of the induced measure $\nu$ on $\RR$. If in addition $\nu$ is invariant under the inversion map, then $\mu$ is called \emph{invariant with respect to $\RR$}. It follows that $(\RR,\nu)$ is a countable p.m.p.\ groupoid if and only if $\mu$ is invariant, and in this case we call $(\RR,\nu)$ a \emph{(countable) p.m.p.\ equivalence relation over $(X,\mu)$}. A Borel subset $A\subset X$ is called \textit{saturated} (or \textit{invariant}) if $\RR(A)=A$ up to a null set. The relation $(\RR,\nu)$ is called \textit{ergodic} if every invariant Borel subset is either null or conull.


A Borel bisection $\theta \in [[\RR]]$ is identified with a partial Borel isomorphism $\theta \colon A \to B$ between measurable subsets $A,B \subset X$ such that $(\theta(x),x) \in \RR$ for $\mu$-almost every $x\in A$, so the full pseudogroup $[[\RR]]$ coincides with the collection of all partial automorphisms of $(X,\mu)$ whose graphs are contained in $\RR$. The full group $[\RR]$ consists of those $\theta \in [[\RR]]$ which are defined almost everywhere on $X$; that is, measurable automorphisms $\theta \in \Aut(X,\mu)$ satisfying $(\theta(x),x) \in \RR$ for $\mu$-almost every $x\in X$.

\subsection{Actions of groupoids}\label{Sec:actionsofgroupoids}

Let $(\GG,\nu)$ be a countable p.m.p.\ groupoid with unit space $(\GG^{(0)},\mu)$. A \textit{(left) measurable action of $\GG$ on a fibered measurable measure space $(A,\tau)$} consists of a standard space $A$, a measurable anchor map $t:A\to\GG^{(0)}$ (so that $A_x\coloneqq t^{-1}(x)$ is the fiber over $x\in\GG^{(0)}$), and a measurable map from $\GG\times_{\GG^{(0)}}A\coloneqq \{(\gamma,a):s(\gamma)=t(a)\}$ to $A$ given by $(\gamma,a)\mapsto \gamma\cdot a$ and satisfying:
\begin{enumerate}
    \item $t(\gamma\cdot a)=r(\gamma)$, for every $(\gamma,a)\in \GG\times_{\GG^{(0)}}A$,
    \item $x\cdot a=a$ for every $x\in\GG^{(0)}$ and $a\in t^{-1}(x)$,
    \item $(\gamma\rho)\cdot a=\gamma\cdot (\rho\cdot a)$, whenever $(\gamma\rho,a)\in \GG\times_{\GG^{(0)}}A$ and $(\gamma,\rho\cdot a)\in\GG\times_{\GG^{(0)}}A$.
\end{enumerate}
Equivalently, for each $\gamma\in\GG$, the action determines a map $\alpha(\gamma):A_{s(\gamma)}\to A_{r(\gamma)}$ given by $\alpha(\gamma)a=\gamma\cdot a$, and the above conditions state that $\alpha(\gamma\rho)=\alpha(\gamma)\alpha(\rho)$ and $\alpha(x)=\text{id}_{A_x}$. 

Now suppose $A$ is endowed with a measurable family of probability measures $\tau=\{\tau_x\}_{x\in\GG^{(0)}}$ where each $\tau_x$ is a probability measure on the fiber $A_x$. Then $A$ is called a \emph{fibered probability space} over $\GG^{(0)}$. A fibered probability space $A$ over $\GG^{(0)}$ may equivalently be viewed as a measure space equipped with a measurable disintegration over the unit space. More precisely, if $(A,\tau)$ is a standard Borel space and $t\colon A\to \GG^{(0)}$ is a measurable map, then a measurable family of $\sigma$-finite measures $\{\tau_x\}_{x\in\GG^{(0)}}$ on the fibers $A_x=t^{-1}(x)$ determines a measure $\tau$ on $A$, given by the integral decomposition
\begin{equation}\label{eq:disintegrationmeasure}
    \tau=\int_{\GG^{(0)}}\tau_x\, \mathrm{d}\mu(x),
\end{equation}
making the anchor map $t\colon (A,\tau)\to(\GG^{(0)},\mu)$ measure preserving. Conversely, by the measure disintegration theorem, if $t\colon (A,\tau)\to (\GG^{(0)},\mu)$ is a measure-preserving map between standard $\sigma$-finite measure spaces, then there exists an essentially unique measurable family of probability measures $\{\tau_x\}_{x\in\GG^{(0)}}$, with each $\tau_x$ supported on the fiber $A_x=t^{-1}(x)$, for which \eqref{eq:disintegrationmeasure} holds.

The action $\alpha:\GG\curvearrowright A$ is called \emph{measure preserving} if for every $\gamma\in\GG$, the map $\alpha(\gamma)\colon A_{s(\gamma)}\to A_{r(\gamma)}$ is a measure-space isomorphism; i.e. $\alpha(\gamma)_*\tau_{s(\gamma)}=\tau_{r(\gamma)}$. If each fiber $A_x$ is countable, then the action is called a \emph{discrete action}. For a measurable action $\GG\curvearrowright A$, we can associate to it the corresponding transformation groupoid, $\GG\ltimes A$, defined as follows. The set of groupoid elements is $\GG\times_{\GG^{(0)}}A$ with unit space $A$. The source and range maps $\widetilde s, \:\widetilde r\colon\GG\ltimes A\to A$ are given by $\widetilde s(\gamma,a)=a$ and $\widetilde r(\gamma,a)=\gamma\cdot a$, respectively, and the multiplication and inverse operations are given by $(\gamma,\rho\cdot a)(\rho,a)=(\gamma\rho,a)$ and $(\gamma,a)^{-1}=(\gamma^{-1},\gamma\cdot a)$. The associated transformation groupoid $\GG\ltimes A$ carries a natural measure $\nu_{\GG\ltimes A}=\tau\circ\lambda$, where $\lambda$ denotes counting measure on the fibers.

We say that the $\GG$-action is \emph{free} if $\GG\ltimes A$ is a principal groupoid. In this case, a subset $F\subset A$ is called a \emph{strict fundamental domain} for the $\GG$-action if it intersects each class in $\RR_{\GG\ltimes A}$ in precisely one element. Then $F$ is a complete section for the action groupoid $\GG\ltimes A$.

If $\sigma_1,\sigma_2$ are two measures on a measurable space $X$, $\sigma_1$ is said to be absolutely continuous with respect to $\sigma_2$, and written $\sigma_1\ll\sigma_2$ if $\sigma_1(X_0)=0$ for every set $X_0\subset X$ for which $\sigma_2(X_0)=0$. Equivalently, $\sigma_1\ll\sigma_2$ if and only if for every $\eps>0$ there exists $\delta>0$ such that $\sigma_1(X_0)\leq\eps$ whenever $\sigma_2(X_0)\leq\delta$. If $\sigma_1\ll\sigma_2$ and $\sigma_2\ll\sigma_1$, we say the measures $\sigma_1$ and $\sigma_2$ are equivalent (or have the same measure class) and we denote it by $\sigma_1\sim\sigma_2$. In the context of actions of groupoids, we will generally assume $t_*\tau\sim\mu$, although this will be explicitly stated.

A particular example of interest is when the acting groupoid is a countable discrete group $\Gamma$. In this case, the \emph{transformation groupoid} associated to a p.m.p.\ action $\Gamma\curvearrowright(X,\mu)$ of a countable group $\Gamma$ is the groupoid $\Gamma\ltimes(X,\mu)=(\GG,\nu)$ defined as follows: the set of groupoid elements is $\GG=\Gamma\times X$ with unit space $\GG^{(0)}=\{e\}\times X$, which is again identified with $X$. The source and range maps $s,r\colon\GG\to \GG^{(0)}$ are given by $s(g,x)=x$ and $r(g,x)=g\cdot x$, respectively, and the multiplication and inverse operations are given by $(g,h\cdot x)(h,x)=(gh,x)$ and $(g,x)^{-1}=(g^{-1},g\cdot x)$. If the action $\Gamma\curvearrowright(X,\mu)$ is essentially free, that is, the stabilizer of $\mu$-almost every point of $X$ is trivial, then the groupoid $\Gamma\ltimes (X,\mu)$ is naturally isomorphic to the orbit equivalence relation $\RR(\Gamma\curvearrowright X)\coloneqq\{(g\cdot x,x):g\in \Gamma, x\in X\}$ of the action.

\subsection{von Neumann algebras associated to groupoids and twists by \texorpdfstring{$2$}{2}-cocycles}

We first recall the von Neumann algebra of multiplier operators associated with countable p.m.p.\ groupoids. Let $(\GG,\nu)$ be a countable p.m.p.\ groupoid over the unit space $(\GG^{(0)},\mu)$. Let $\B(L^2(\GG,\nu))$ denote the space of bounded operators on the Hilbert space $L^2(\GG,\nu)$. Each essentially bounded function $f\in L^\infty(\GG,\nu)$ can be identified with the multiplier operator $M_f \in \B(L^2(\GG,\nu))$ given by 
\[(M_f \xi)(\gamma)=f(\gamma)\xi(\gamma), \quad \xi \in L^2(\GG,\nu).\]
Under this identification, $L^\infty(\GG,\nu)$ is a maximal abelian subalgebra (MASA) of $\B(L^2(\GG,\nu))$. We let $L^\infty(\Delta)$ denote the space of functions $f\in L^\infty(\GG,\nu)$ that are essentially supported on $\GG^{(0)}$, so then $L^\infty\Delta$ is identified with a corner of $ L^\infty(\GG,\nu) \subset \B(L^2(\GG,\nu))$. This notation is motivated by the fact that in the case when the groupoid is a countable measured equivalence relation, the unit space is identified with the diagonal set.

Next we consider the groupoid von Neumann algebras $L_c(\GG)$ associated with $(\GG,\nu)$ and $2$-cocycles $c$ on $\GG$. To define $2$-cocycles, we extend the range map on $\GG$ and define $r^{(2)}\colon \GG^{(2)}\to \GG^{(0)}$ by $r^{(2)}(\gamma,\eta) = r(\gamma) \in \GG^{(0)}$ for each pair $(\gamma,\eta)\in \GG^{(2)}$. We equip $\GG^{(2)}$ with the $\sigma$-finite measure $\nu^{(2)}$ defined by
\[\nu^{(2)}(C) = \int_{\GG^{(0)}} |(r^{(2)})^{-1}(x)\cap C|\,\mathrm{d}\mu(x), \quad C\subset \GG^{(2)} \text{ Borel subset}.\]
We let $Z^2(\GG,\T)$ denote the space of $2$-cocycles for $(\GG,\nu)$ which consists of Borel maps $c\colon \GG^{(2)}\to \T$ such that
\[c(\gamma,\eta)c(\gamma\eta,\kappa) = c(\gamma,\eta\kappa)c(\eta,\kappa), \quad \nu^{(2)}\text{-almost everywhere.}\]
We say the cocycle $c\in Z(\GG^{(2)},\T)$ is \emph{trivial} if $c=1$ almost everywhere; two cocycles $c_1,c_2$ are \emph{cohomologous} if there exists a Borel map $\omega\colon \GG\to\T$ such that 
\[c_1(\gamma,\eta)\omega(\gamma)\omega(\eta) = c_2(\gamma,\eta)\omega(\gamma\eta),  \quad \nu^{(2)}\text{-almost everywhere;}\]
and a $2$-cocycle $c\in Z^2(\GG,\T)$ is \emph{normalized} if $c(\gamma,\id_{s(\gamma)}) = c(\id_{r(\gamma)},\gamma) = 1$ for almost every $\gamma\in \GG$. It turns out that every $2$-cocycle is cohomologous to a normalized one, so in what follows we always assume the $2$-cocycles are normalized.

Let $\mathcal{I}(\GG)$ denote the space of Borel functions $f \colon \GG \to \mathbb{C}$ such that
\[\left\Vert f \right\Vert_\II \coloneqq 
\max\left\{\left\Vert x \mapsto \sum_{r(\gamma)=x} |f(\gamma)| \right\Vert_\infty,\;
\left\Vert x \mapsto \sum_{s(\gamma)=x} |f(\gamma)| \right\Vert_\infty \right\}<\infty.\]
Then $\mathcal{I}(\GG)$ is a $*$-algebra under twisted convolution product and involution given by
\[(f*_cg)(\gamma)=\sum_{\gamma_1\gamma_2=\gamma}f(\gamma_1)g(\gamma_2)c(\gamma_1,\gamma_2), \qquad f^*(\gamma)=\overline{f(\gamma^{-1})}\overline{c(\gamma,\gamma^{-1})},\qquad f,g\in \II(\GG).\]
Given a $2$-cocycle $c\in Z^2(\GG,\T)$, we define the $c$-twisted left representation $\lambda^c \colon \mathcal{I}(\GG) \to \B(L^2(\GG,\nu))$ by
\[(\lambda^c(f)\xi)(\gamma) =\sum_{\gamma_1\gamma_2=\gamma}f(\gamma_1)\xi(\gamma_2)c(\gamma_1,\gamma_2) =\sum_{s(\eta)=s(\gamma)}f(\gamma\eta^{-1})\xi(\eta)c(\gamma\eta^{-1},\eta), \quad \xi\in L^2(\GG,\nu).\]
It follows that $\lambda^c(f)$ is a bounded operator and that $\lambda^c$ is a projective $*$-representation. The \emph{groupoid von Neumann algebra $L_c(\GG)$ associated to the groupoid $\GG$ and $2$-cocycle $c$} is defined by $L_c(\GG)\coloneqq \lambda^c(\II(\GG))'' \subset \B(L^2(\GG,\nu))$. Note that any two cohomologous $2$-cocycles $c_1,c_2$ give rise to spatially isomorphic twisted group von Neumann algebras $L_{c_1}(\GG)$ and $L_{c_2}(\GG)$. 

Since $\GG$ is a countable p.m.p.\ groupoid, $L_c(\GG)$ is a tracial von Neumann algebra with a canonical faithful normal tracial state $\tau$ given by the inner product $\tau(T) = \langle T\chi_{\GG^{(0)}},\chi_{\GG^{(0)}}\rangle$ for $T\in L_c(\GG)$, where $\chi_{\GG^{(0)}} \in L^2(\GG,\nu)$ is the characteristic function supported on the unit space $\GG^{(0)}$. For $f \in \II(\GG)$, the trace can be written as $\tau(\lambda^c(f))=\int_{\GG^{(0)}} f(x)\,\mathrm{d}\mu(x)$.

Note that $L_c(\GG)$ contains a unital subalgebra canonically identified with $L^\infty(\GG^{(0)},\mu)$ via
\begin{equation} \label{eqn: identify Linfty X}
    (f\xi)(\gamma) = f(r(\gamma))\xi(\gamma), \quad f\in L^\infty(\GG^{(0)},\mu),\:\xi \in L^2(\GG,\nu).
\end{equation}
In what follows, whenever we regard $L^\infty(\GG^{(0)},\mu)$ as a subalgebra of $\B(L^2(\GG,\nu))$, this identification is understood to be given by the above formula. Note that, although this algebra is algebraically isomorphic to $L^\infty(\Delta)\subset L^\infty(\GG,\nu)$, they are two distinct operator-theoretic realizations inside $\B(L^2(\GG,\nu))$.

Let $B\subset\GG$ be the bisection corresponding to $\theta\in[[\GG]]$ and put $u_\theta^c\coloneqq \lambda^c(\chi_B)$.  Then
\begin{equation} \label{eqn: left conv formula}
    \begin{aligned}
        (u^c_\theta \xi)(\gamma)&= \sum_{\gamma_1\gamma_2=\gamma}\chi_B(\gamma_1)\xi(\gamma_2)c(\gamma_1,\gamma_2) = \begin{cases}
        c(r_\theta^{-1}(r(\gamma))\theta^{-1}\gamma)\xi\bigl([r_\theta^{-1}(r(\gamma))]^{-1}\gamma\bigr)
        & \text{if } r(\gamma) \in r(\theta),\\
        0 & \text{if } r(\gamma) \not\in r(\theta)\end{cases}\\
        &=c(r_\theta^{-1}(r(\gamma)),\theta^{-1}\gamma)\xi(\theta^{-1}\gamma).
\end{aligned}
\end{equation}
If $B_1,B_2$ correspond to $\theta_1,\theta_2 \in [[\GG]]$, there are unitary functions $\omega_{\theta_1,\theta_2},\eta_{\theta_1}\in L^\infty\GG^{(0)}$ on the appropriate supports such that
\begin{equation*}
    \begin{aligned}
    u_{\theta_1}^cu_{\theta_2}^c &=\omega_{\theta_1,\theta_2}u_{\theta_1\theta_2}^c\:\:\text{ and }\:\: (u_{\theta_1}^c)^* =\eta_{\theta_1}u_{\theta_1^{-1}}^c.
    \end{aligned}
\end{equation*}
More explicitly, if $\gamma_1\in B_1$ and $\gamma_2\in B_2$ are the unique composable arrows with $r(\gamma_1\gamma_2)=x$, then $\omega_{\theta_1,\theta_2}(x) = c(\gamma_1,\gamma_2)$; if
$\gamma\in B_1^{-1}$ is the unique arrow with $r(\gamma)=x$, then $\eta_{\theta_1}(x)=\overline{c(\gamma,\gamma^{-1})}$.  Consequently,
\[(u_\theta^c)^*u_\theta^c=\chi_{s(\theta)}\:\:\text{ and }\:\: u_\theta^c(u_\theta^c)^*=\chi_{r(\theta)}.\]
Thus, $u_\theta^c$ is a partial isometry for $\theta\in[[\GG]]$ and it is unitary whenever $\theta\in[\GG]$. For $f \in L^\infty \GG$, we have
\begin{equation*}
    \begin{aligned}
        (M_f u_\theta^c \xi)(\gamma) &= f(\gamma)c(r_\theta^{-1}(r(\gamma)),\theta^{-1}\gamma)\xi(\theta^{-1}\gamma) = f(\gamma)c(r_\theta^{-1}(r(\gamma)),\theta^{-1}\gamma)\xi(\theta^{-1}\gamma)\chi_{\GG^{s(\theta)}}(\theta^{-1}\gamma),\\
        (u_\theta^cM_f\xi)(\gamma)&= c(r_\theta^{-1}(r(\gamma)),\theta^{-1}\gamma)(M_f\xi)(\theta^{-1}\gamma) = c(r_\theta^{-1}(r(\gamma)),\theta^{-1}\gamma)f(\theta^{-1}\gamma)\xi(\theta^{-1}\gamma) = (L_\theta(f))(\gamma) (u_\theta^c\xi)(\gamma),
    \end{aligned}
\end{equation*}
so $u^c_\theta M_f = M_{L_\theta f}u^c_\theta$, $u^c_\theta M_f (u^c_\theta)^* = M_{L_\theta f}$, and $[M_f,u^c_\theta] =  M_{f-L_\theta f}u^c_\theta$. In particular, each $u^c_\theta$ is an element in the \emph{quasi-normalizer $\GG\NN_{L_c\GG}(L^\infty \GG^{(0)})$} consisting of partial isometries $v\in L_c\GG$ such that $v(L^\infty \GG^{(0)})v^*\subset L^\infty \GG^{(0)}$ and $v^*(L^\infty \GG^{(0)})v\subset L^\infty \GG^{(0)}$. If $\theta \in [\GG]$ is in the full group, then $u^c_\theta$ is a unitary that lies in the \emph{normalizer group} $\NN_{L_c\GG}(L^\infty \GG^{(0)}) \coloneqq \{u\in L_c\GG: u $ is a unitary such that $u(L^\infty \GG^{(0)})u^*= L^\infty \GG^{(0)}\}$, and the map
\[[\GG] \to \mathcal{U}(L(\GG)), \qquad \theta \mapsto u_\theta\]
is projective with coefficients in the unit-space algebra $L^\infty \GG^{(0)}$.

By \cite[Proposition 3.1]{groupoidfactor}, $\GG$ admits a countable collection of partial bisections $\{B_i\} \subset [[\mathcal{G}]]$ such that the entire groupoid $\GG= \bigsqcup_i B_i$ is their disjoint union and yields a Fourier decomposition for elements of $L_c(\GG)$. Hence, $L_c(\GG)$ is generated by the partial isometries $\{u^c_\theta : \theta \in [[\GG]]\} \subset \GG\NN_{L_c\GG}(L^\infty \GG^{(0)})$, so $L^\infty \GG^{(0)}$ is a \emph{regular subalgebra} of $L_c\GG$ in the sense that the normalizer group $\NN_{L_c\GG}(L^\infty \GG^{(0)})$ generates $L_c\GG$ as a von Neumann algebra by \cite[Lemma 12.1.2]{anantharaman2017introduction}. However, we note that $L^{\infty}\GG^{(0)}$ is not a maximal abelian subalgebra of $L_c\GG$ unless $\GG$ is principal, so $L^\infty \GG^{(0)}\subset L_c\GG$ is not a Cartan subalgebra in general.

The standard Hilbert space $L^2(L_c(\GG))$ of $L_c(\GG)$ is canonically identified with $L^2\GG$ via mapping $L^2(L_c\GG)\ni \widehat{\lambda^c(f)} \mapsto f \in L^2\GG$ for $f\in \II(\GG)$ since $\chi_{\GG^{(0)}}$ is cyclic and separating. There is also a natural twisted right convolution representation $\rho^c$ of $\mathcal{I}(G)$ on $L^2(\GG,\nu)$ given by
\begin{equation} \label{eqn: right conv formula}
    (\rho^c(f)\xi)(\gamma) = \sum_{\gamma_1\gamma_2 = \gamma}\xi(\gamma_1) f(\gamma_2^{-1}) c(\gamma_1,\gamma_2^{-1}), \quad \xi\in L^2(\GG,\nu).
\end{equation}
Then $R_c(\GG)\coloneqq \rho^c(\mathcal{I}(\GG))'' = L_c(\GG)' = JL_c(\GG)J$, where $J = J_c$ is the (twisted) modular conjugation map in $\B(L^2(\GG,\nu))$ given by 
\[(J\xi)(\gamma)=\overline{c(\gamma^{-1},\gamma)}\,\overline{\xi(\gamma^{-1})}.\]
Thus, the commutant $R_c(\GG)$ of $L_c(\GG)$ in $\B(L^2(\GG,\nu))$ is generated by right convolution operators, and the left and right regular representations commute. Similar to the left regular case, each partial Borel isomorphism $\theta \in [[\GG]]$ with the corresponding Borel bisection $B\subset \GG$ gives rise to a partial isometry $v^c_\theta\coloneqq \rho^c(u_\theta^c) =  Ju_\theta^cJ\in R_c(\GG)$ that satisfies 
\[v^c_\theta M_f (v^c_\theta)^* = M_{R_\theta f}\:\:\text{ and }\:\: [M_f,v^c_\theta] =  M_{f\chi_{\GG_{r(\theta)}}-R_\theta f}v^c_\theta,\]
for $f \in L^\infty \GG$. When $c\in Z^2(\GG,\T)$ is the trivial $2$-cocycle, we omit the symbol $c$ in the superscripts and subscripts in above notations. 

\subsection{Direct integral decomposition} \label{sec: Direct int decomp}

For a countable p.m.p.\ discrete measured groupoid $(\GG,\nu)$ over the unit space $(\mathcal G^{(0)},\mu)$, we can decompose $L^2(\GG,\nu)$ as a direct integral of Hilbert spaces
\[L^2(\GG,\nu) = \int_{\GG^{(0)}}^\oplus \ell^2(\GG_x)\,\mathrm{d}\mu(x).\]
Note that the $J(L^\infty \GG^{(0)})J$ is the algebra of diagonalizable operators with respect to the above disintegration of $L^2(\GG,\nu)$, where the right multiplication operators in $J(L^\infty \GG^{(0)})J$ are explicitly given by
\[(JfJ\xi)(\gamma) = f(s(\gamma))\xi(\gamma), \qquad f\in L^\infty \GG^{(0)},\: \xi\in L^2(\GG,\nu).\]
The commutant $(J(L^\infty \GG^{(0)})J)'$, which admits a decomposition
\[(J(L^\infty \GG^{(0)})J)' = \int_{\GG^{(0)}}^\oplus \B(\ell^2(\GG_x))\,\mathrm{d}\mu(x) \subset \B(L^2(\GG,\nu)),\]
is the algebra of decomposable operators with essentially bounded measurable field $x\mapsto T_x\in \B(\ell^2(\GG_x))$ \cite[Corollary IV.8.16]{Takesaki1}. It is clear that the elements in $L_c\GG$ and $L^\infty(\GG,\nu)$ commute with $J(L^\infty \GG^{(0)})J$, so they are decomposable operators. For $f\in \II(\mathcal G)$, we have $\lambda^c(f) = \int_{\GG^{(0)}}^\oplus \lambda^c_x(f)\,\mathrm{d}\mu(x)$, where $\lambda^c_x(f)\colon \ell^2(\GG_x)\to \ell^2(\GG_x)$ is defined as
\[(\lambda^c_x(f)\xi)(\gamma) =\sum_{\gamma_1\gamma_2=\gamma}f(\gamma_1)\xi(\gamma_2)c(\gamma_1,\gamma_2),  \qquad \xi\in \ell^2(\GG_x).\]
On the other hand, under the above identification we have
\[L^\infty(\GG,\nu) = \int_{\GG^{(0)}}^\oplus \ell^\infty(\GG_x)\,\mathrm{d}\mu(x).\]
There is a canonical normal faithful conditional expectation $\E\colon (J(L^\infty \GG^{(0)})J)'\to L^\infty(\GG)$ given by the fiberwise diagonal map
\begin{equation} \label{eqn: cond exp}
    \E(T)(\gamma) = \langle T_{s(\gamma)}\delta_\gamma,\delta_\gamma\rangle_{\ell^2(\GG_{s(\gamma)})}, \qquad T = \int_{\GG^{(0)}}^\oplus T_{x}\,\mathrm{d}\mu(x)\in (J(L^\infty \GG^{(0)})J)',
\end{equation}
which defines an essentially bounded measurable function on $\GG$ because $|\E(T)(\gamma)| \leq \| T\|$ and matrix coefficients of measurable decomposable fields are measurable. Normality of $\E$ follows from monotone convergence on diagonal coefficients fiberwise.

Fix a $2$-cocycle $c \in Z^2(\GG,\T)$. For a partial Borel isomorphism $\theta \in [[\GG]]$ with the corresponding Borel bisection $B\subset \GG$, under the above direct integral decomposition of $L^2(\GG)$, the vector $\widehat{u_\theta^c}$ corresponds to the measurable field
\[x\mapsto \begin{cases}
    \delta_{\theta \cdot 1_x} & \text{if }x\in s(\theta),\\
    0 & \text{if }x\not\in s(\theta),
\end{cases}\]
since a local bisection contains at most one arrow above each source point. Hence, for any $T\in (J(L^\infty \GG^{(0)})J)'$,
\begin{equation} \label{eqn: inner prod}
    \langle T\widehat{u_\theta^c}, \widehat{u_\theta^c}\rangle 
    = \int_{s(\theta)} \langle T_x \delta_{\theta \cdot 1_x}, \delta_{\theta \cdot 1_x}\rangle_{\ell^2(\GG_x)}\,\mathrm{d}\mu(x) 
    = \int_{s(\theta)} \E(T)(\theta \cdot 1_x)\,\mathrm{d}\mu(x) 
    = \int_B \E(T)(\gamma)\,\mathrm{d}\nu(\gamma). 
\end{equation}
Let $\theta\in [[\GG]]$ and $B\subset \GG$ be the corresponding local bisection. In general $v^c_\theta\coloneqq \rho^c(\chi_B) = Ju_\theta^cJ$ is not decomposable since it does not lie inside $(JL^\infty \GG^{(0)} J)'$, but it still admits a fiberwise decomposition as follows. For each $x\in r(\theta) = r(B)$, recall that $r_\theta^{-1}(x)$ is the unique arrow in $B$ with $r(r_\theta^{-1}(x)) = x$. Let $\alpha\colon r(B)\to s(B)$ be a partial Borel isomorphism on $\GG^{(0)}$ induced by $B$ and given by $\alpha(x) = s(r_\theta^{-1}(x))$. For each $x\in r(B)$, right multiplication by $r_\theta^{-1}(x)$ induces a unitary $V_x^c\colon \ell^2(\GG_{\alpha(x)})\to \ell^2(\GG_{x})$ defined by 
\[(V_x^c \zeta)(\gamma) = c(\gamma r_\theta^{-1}(x),r_\theta^{-1}(x))\zeta(\gamma r_\theta^{-1}(x)), \qquad \gamma\in \GG_x.\]
On the other hand, by the right convolution formula (\ref{eqn: right conv formula}), we have for $\xi\in L^2(\GG)$,
\[(v_\theta^c\xi)(\gamma) = (\rho^c(\chi_B)\xi)(\gamma) = \begin{cases}
    c\bigl(\gamma r_\theta^{-1}(x), r_\theta^{-1}(x)\bigr)\xi(\gamma r_\theta^{-1}(x)) & \text{ if }s(\gamma) = x \in r(B),\\
    0 & \text{ if }s(\gamma)\not\in r(B).
\end{cases}\]
Using the direct integral decomposition of $\xi$, this gives a fiberwise formula for $v_\theta^c$ such that
\[(v_\theta^c\xi)_x = \begin{cases}
    V_x^c \xi_{\alpha(x)} & \text{ if }s(\gamma) = x \in r(B),\\
    0 & \text{ if }s(\gamma)\not\in r(B).
\end{cases}\]
Using this fiberwise formula we can show $\E\colon (J(L^\infty \GG^{(0)})J)'\to L^\infty \GG$ is equivariant by the right $[[\GG]]$-action in the sense that $\E\circ \Ad(Ju_\theta^c J)= R_\theta \circ \E$. Indeed, for $T\in (J(L^\infty \GG^{(0)}) J)'$ with decomposition $T = \int_{\GG^{(0)}}^\oplus T_x\,\mathrm{d}\mu(x)$, we know $v_\theta^c T (v_\theta^c)^* \in v_\theta^c(J(L^\infty \GG^{(0)})J)'(v_\theta^c)^* \subset (J(L^\infty \GG^{(0)})J)'$ is also decomposable, and its fiber at $x\in r(\theta)$ is
\[(v_\theta^c T (v_\theta^c)^*)_x = V_x^c T_{\alpha(x)}(V_x^c)^*,\]
and zero when $x\not\in r(\theta)$. Therefore, for $\gamma\in \GG_x$, $x\in r(\theta)$, we have
\begin{equation*}
    \begin{aligned}
        \E(v_\theta^c T(v_\theta^c)^*)(\gamma) &=  \langle (v_\theta^c T(v_\theta^c)^* )_{x} \delta_\gamma,  \delta_\gamma\rangle_{\ell^2(\GG_x)} = \langle V_x^c T_{\alpha(x)}(V_x^c)^* \delta_\gamma,\delta_\gamma\rangle_{\ell^2(\GG_x)} \\
        & = \langle T_{\alpha(x)}(V_x^c)^* \delta_\gamma, (V_x^c)^*\delta_\gamma\rangle_{\ell^2(\GG_{\alpha(x)})}
        = \langle T_{\alpha(x)} \delta_{\gamma r_\theta^{-1}(x)}, \delta_{\gamma r_\theta^{-1}(x)}\rangle_{\ell^2(\GG_{\alpha(x)})}\\
        & = \E(T)(\gamma r_\theta^{-1}(x)) = R_\theta (\E(T))(\gamma).
    \end{aligned}
\end{equation*}
When $x = s(\gamma)\not\in r(\theta)$, then $ \E(v_\theta^c T(v_\theta^c)^*)(\gamma) = 0 = R_\theta (\E(T))(\gamma)$. Hence $\E\circ \Ad(Ju_\theta^c J)= R_\theta \circ \E$ as we claimed.

Also, $\E$ is left equivariant under the conjugate actions of twisted unitaries, i.e., $L_\theta\circ \E = \E\circ \Ad(u_\theta^c)$ for $\theta \in [[\GG]]$. Indeed, fix $\theta\in [[\GG]]$ with corresponding local Borel bisection $B\subset \GG$. Recall that for any $x \in s(\theta) = s(B)$, $s_\theta^{-1}(x)\in B$ is the unique arrow with $s(s_\theta^{-1}(x)) = x$. Left multiplication by $B$ induces fiberwise a unitary $U_x^c\colon \ell^2(\GG_x)\to \ell^2(\GG_x)$, for each $x\in \GG^{(0)}$, given by 
\[U_x^c\delta_\gamma = \begin{cases}
    c(s_\theta^{-1}(r(\gamma)),\gamma)\delta_{s_\theta^{-1}(r(\gamma))\gamma} & \text{ if }r(\gamma)\in s(\theta),\\
    0 & \text{ otherwise.}
\end{cases}\]
On the other hand, from the left convolution formula (\ref{eqn: left conv formula}) we have
\[(u_\theta^c \xi)(\gamma) = \begin{cases}
    c(r_\theta^{-1}(r(\gamma)),r_\theta^{-1}(r(\gamma))^{-1}\gamma)\xi(r_\theta^{-1}(r(\gamma))^{-1}\gamma)  & \text{ if }r(\gamma)\in r(\theta),\\
    0 & \text{ otherwise}.
\end{cases}\]
Since $\theta\in[[\GG]]$, it is clear that $s_\theta^{-1}(r(\gamma)) = r_\theta^{-1}(r(\gamma))^{-1}$, so $u_\theta^c$ admits the direct integral decomposition $u_\theta^c = \int_{\GG^{(0)}}^\oplus U_x^c\,\mathrm{d}\mu(x)$. For any $T \in (J(L^\infty \GG^{(0)})J)'$, it follows that
\[u_\theta^c T (u_\theta^c)^* = \int_{\GG^{(0)}}^\oplus U_x^c T_x (U_x^c)^*\,\mathrm{d}\mu(x) \in (J(L^\infty \GG^{(0)})J)'.\]
Hence for $\gamma\in \GG$ with $r(\gamma)\in B$, we have
\begin{equation*}
    \begin{aligned}
        \E(u_\theta^c T (u_\theta^c)^*)(\gamma) &= \langle U_x^cT_x (U_x^c)^*\delta_\gamma, \delta_\gamma\rangle_{\ell^2(\GG_x)} = \langle T_x (U_x^c)^*\delta_\gamma, (U_x^c)^* \delta_\gamma\rangle_{\ell^2(\GG_x)} \\
        &= \langle T_x \delta_{r_\theta^{-1}(r(\gamma))^{-1}\gamma}, \delta_{r_\theta^{-1}(r(\gamma))^{-1}\gamma}\rangle_{\ell^2(\GG_x)} = (\E(T))(r_\theta^{-1}(r(\gamma))^{-1}\gamma) = L_\theta(\E(T))(\gamma),
    \end{aligned}
\end{equation*}
and this proves the left $[[\GG]]$-equivariance formula $\E\circ \Ad(u_\theta^c )= L_\theta \circ \E$.

Moreover, the restriction of $\E$ on $L_c\GG$ is the (trace-preserving) normal canonical conditional expectation $\E_{L^\infty \GG^{(0)}}\colon L_c\GG \to L^\infty \GG^{(0)}$, where we as usual identify $L^\infty \GG^{(0)}$ as a unital von Neumann algebra of $L_c\GG$ as in \eqref{eqn: identify Linfty X}. Indeed, for $f\in \II(\GG)$ and $\gamma\in \GG$,
\[\E(\lambda^c(f))(\gamma)=\langle\lambda^c_{s(\gamma)}(f)\delta_\gamma,  \delta_\gamma\rangle = f(1_{r(\gamma)})c(1_{r(\gamma)},\gamma)=f(r(\gamma)),\]
where the last equality follows from the normalization of the $2$-cocycle $c$. On the other hand, $\E_{L^\infty \GG^{(0)}}(\lambda^c(f))=f|_{\GG^{(0)}}$. Hence $\E(\lambda^c(f))=\E_{L^\infty \GG^{(0)}}(\lambda^c(f))$. Since $\lambda^c(\II(\GG))$ is ultraweakly dense in $M$ and both sides are
normal in $x\in L_c\GG$, we conclude that
\begin{equation} \label{eqn: restr of cond exp}
    \E|_{L_c\GG} = \E_{L^\infty \GG^{(0)}}|_{L_c\GG}.
\end{equation}

\subsection{C\texorpdfstring{$^*$}{*}-bimodule topologies and bidual spaces associated with von Neumann algebras}
In this section we collect some basic facts about a locally convex topology on C$^{*}$-bimodules relative to von Neumann algebras, introduced by Magajna in \cite{MR1616512} and \cite{MR1750836}, and later adapted and generalized in \cite{ding2023properproximality} and \cite{ding2023biexact}. This topology and the corresponding bidual spaces will be used next section for giving equivalent characterizations of properly proximal groupoids in terms of biduals.

Let $M$ be a von Neumann algebra. A \textit{(concrete) operator $M$-system} $X$ consists of a concrete embedding of an operator system $X\subset \mathbb{B}(\mathcal{H})$ into the space of bounded linear operators on a Hilbert space $\mathcal{H}$ and a faithful non-degenerate $*$-representation $\pi\colon M\to \mathbb{B}(\mathcal{H})$ such that $X$ is a $\pi(M)$-bimodule. We call such an embedding of $X\subset \mathbb{B}(\mathcal{H})$ together with $\pi$ \textit{a concrete realization of $X$ as an $M$-system}. We say the operator $M$-system $X$ is \textit{($M$-)normal} if the concrete realization can be made so that $\pi$ extends to a normal representation of $M$. If $X$ is moreover a unital C$^{*}$-algebra, we will say that it is an \textit{operator $M$-C$^{*}$-algebra} or an \textit{($M$-)normal} $M$-C$^{*}$-algebra if $X$ is so as an operator system. We will often drop $\pi$ in the notation.

Given positive normal linear functionals $\omega,\rho \in M_*$ in the predual of $M$, we consider the seminorm on $X$ as in \cite[Section 3]{ding2023biexact} given by
\[s^{\rho}_{\omega}(x) = \inf\{\rho(a^{*}a)^{\frac{1}{2}}\left\Vert y \right\Vert\omega(b^{*}b)^{\frac{1}{2}} : x= a^{*}yb, a,b\in M, y\in X\}.\]
We call the topology on $X$ induced by the seminorms $\{s^{\rho}_{\omega}:\omega,\rho \in (M_*)_{+}\}$ the \textit{$M$-topology} on $X$.

We denote by $X^{M\sharp M}$, or just by $X^{\sharp}$ if no confusion will arise, the space of linear functionals $\varphi \in X^{*}$ such that for any $x\in X$, the map $M\times M \ni (a,b)\mapsto \varphi(axb)$ is an ultraweakly continuous bilinear form on $M$, i.e., the linear functionals $M\ni a\mapsto \varphi(ax)$ and $M\ni b\mapsto \varphi(xb)$ are normal. We call the $\sigma(X, X^{M\sharp M})$-topology the \textit{weak $M$-topology}. By \cite[Thoerem 3.7]{MR1750836} or \cite[Proposition 3.3]{ding2023biexact}, a functional $\varphi\in X^{*}$ is continuous in the $M$-topology if and only if $\varphi\in X^{M\sharp M}$.

Now, suppose $B$ is a unital $M$-normal $M$-C$^{*}$-algebra. We let $p_{\nor}\in M^{**}\subset B^{**}$ denote the projection corresponding to the support of the identity representation $M$. As explained in \cite[Section 2]{ding2023properproximality} and \cite[Section 3.1]{ding2023biexact}, we may identify $(p_{\nor}B^{**}p_{\nor})_{*} \simeq B^{\sharp}$ by considering the restriction map to $B$, so the dual map naturally gives $B^{\sharp *}$ a von Neumann algebra structure so that 
\[B^{\sharp *}\simeq p_{\nor}B^{**}p_{\nor}\]
as von Neumann algebras. Since $p_{\nor}$ commutes with $A\subset B^{**}$, the above isomorphism preserves the natural $M$-bimodule structures on $B^{\sharp *}$ and $p_{\nor}B^{**}p_{\nor}$, so we can view $M$ as a von Neumann subalgebra of $B^{\sharp *}$. 

If $i_{B}\colon B\to B^{\sharp *}$ denotes the canonical inclusion map of $B$, then $i_{B}$ is an $M$-bimodular complete order isomorphism, but it is not a $*$-homomorphism in general because $p_{\nor}$ might not be central in $B^{**}$.

\section{Properly proximal groupoids}\label{sec:defpp}

\noindent This section introduces the definition of proper proximality for countable measured groupoids and discusses some equivalent characterizations.

\subsection{The \texorpdfstring{$c_0$}{c0} and small-at-infinity boundary spaces} For each measurable subset $B\subset \mathcal G^{(0)}$, recall that $\GG^B=r^{-1}(B)= \{\gamma \in \GG: r(\gamma) \in B\}$ and $\GG_B = s^{-1}(B)= \{\gamma \in \GG: s(\gamma) \in B\}$. If we let $p \coloneqq \chi_B \in L^\infty \mathcal G^{(0)} \subset L\GG$, then $p = M_{\chi_{\GG^B}}$ and $JpJ = M_{\chi_{\GG_B}}$ in $\B(L^2\GG)$.

\begin{defn}
    For a countable p.m.p.\ groupoid $(\GG,\nu)$ with unit space $(\mathcal G^{(0)},\mu)$, define $c_0(\GG)$ to be the collection of all functions $f\in L^{\infty}(\GG,\nu)$ such that for every $\eps>0$ there exist measurable subsets $B,B' \subset \mathcal G^{(0)}$ with $\mu(B),\mu(B') < \eps$ and a finite measure subset $F\subset \GG$ such that
\begin{equation*}
   \| (1-\chi_{\GG^B})(1-\chi_{\GG_{B'}})f(1-\chi_{F})\| _\infty\leq \eps.
\end{equation*}
The small-at-infinity boundary of $\GG$, denoted by $\mathbb S(\mathcal G)$, consists of all functions $f\in L^{\infty}(\mathcal G,\nu)$ satisfying $f\chi_{\GG_{r(\theta)}}-R_{\theta}f\in c_0(\GG)$ for every $\theta \in [[\GG]]$. 
\end{defn}

It is clear that $c_0(\GG)$ is an ideal of $L^\infty(\GG,\nu)$ and $\S(\GG)$ is a C$^*$-subalgebra of $L^\infty(\GG,\nu)$.

\begin{remark}\label{rem:3.2}
\begin{enumerate}
    \item[(i)] Notice that $[[\GG]]$ acts by left translation on $\mathbb{S}(\GG)$. Moreover, $L^{\infty}(\mathcal G^{(0)})\subset\mathbb{S}(\GG)$ since for $f\in L^{\infty}(\GG^{(0)})$ and $\theta\in [[\GG]]$, we have $(f\chi_{\GG_{r(\theta)}}-R_{\theta}f)(\gamma)=f(r(\gamma))-f(r(\gamma))=0$.
    \item[(ii)] We note that $\mathbb{S}(\GG) = \{f\in L^{\infty}(\GG,\nu):f-R_{\theta}f\in c_0(\GG)$ for every $\theta\in[\GG]\}$. Indeed, for a fixed $\theta\in [[\GG]]$, if we let $B\subset\GG$ denote the corresponding local section, then there exists a global section $A\subset \GG$ such that $B\subset A$ by the Lusin-Novikov Theorem \cite[Theorem 18.10]{Kechris}. Let $\theta'\in [\GG]$ denote the Borel isomorphism on $\GG$ given by the left multiplication by $A$, so that $\restr{\theta'}{ s(\theta)} = \theta$. If $f-R_{\theta'}f\in c_0(\GG)$, then we have $f\chi_{\GG_{r(\theta)}}-R_{\theta}f = (f-R_{\theta'}f)\chi_{\GG_{r(\theta)}} \in c_0(\GG)$. 
\end{enumerate}
\end{remark}

We discuss two ways in which the full pseudogroup $[[\mathcal G]]$ can be generated: with respect to the uniform topology metric $d_u$ and through finite compositions of elements from a subset $\mathcal P\subset[[\mathcal G]]$ (including their inverses). First, we can endow $[\mathcal G]$ with the uniform metric $d_u(\phi,\psi)=\mu(\{x\in\mathcal G^{(0)}:\phi^{-1}(x)\neq\psi^{-1}(x)\})$, which turns $([\mathcal G],d_{u})$ into a Polish group (see e.g. \cite[Lemma 9.1]{Kida-TuckerDrob}). If every element in $[\mathcal G]$ can be approximated in the uniform metric by an element in $\mathcal P$, we write $\overline{\mathcal P}^{d_u}=[\mathcal G]$. Second, let $\mathcal P \subset [[\mathcal G]]$ be a set of Borel partial bisections, and denote by $\langle\mathcal P\rangle$ the set of all finite compositions of elements in $\mathcal{P}$ and their inverses, that is, $\langle\mathcal P\rangle = \{ \theta_{n_1} \circ \ldots \circ \theta_{n_k} : \theta_{n_i} \in \mathcal P \cup \mathcal P^{-1}, k \in \mathbb N \}$. Since $\langle \mathcal P \rangle$ is countable, we enumerate its elements as $\langle\mathcal P\rangle = \{\sigma_n\}_{n \in \mathbb N}$. For any $\rho \in [[\mathcal G]]$ and each $n$, define the set $E_n = \{ x \in s(\rho) : \rho(x) = \sigma_n(x)\}$. Since both $\rho$ and $\sigma_{n}$ are measurable, each $E_{n}$ is a measurable set. We write $\langle \mathcal P\rangle=[[\GG]]$ when for every $\rho\in[[\GG]]$ and almost every $x\in s(\rho)$, there exists some $n\in\mathbb N$ for which $x\in E_n$. To make the sets $E_n$ disjoint, we define a sequence of measurable sets by $D_1 = E_1$ and $D_n = E_n \setminus \cup_{k<n}E_k$. By construction, $\{D_n\}_{n \in \mathbb N}$ forms a measurable partition of $s(\rho)$. Consequently, restricted to each $D_{n}$, $\rho$ matches $\sigma _{n}$, yielding $\rho|_{D_n} = \sigma_n|_{D_n}$, which implies $\rho = \sqcup_n \sigma_n|_{D_n}$.

The following lemma establishes that verifying whether a function $f$ belongs to $\mathbb S(\mathcal G)$ only requires checking that $f-R_{\theta}f\in c_0(\mathcal G)$ for all $\theta$ in a suitably chosen countable subset of $[[\mathcal G]]$.

\begin{lemma}\label{claim:generating set} Let $(\mathcal G,\nu)$ be a countable p.m.p.\ groupoid with unit space $(\mathcal G^{(0)},\mu)$. Let $\mathcal P\subset[[\mathcal G]]$ be a countable set and denote by $\langle\mathcal P\rangle=\{\theta_{n_1}\circ\ldots\circ\theta_{n_k}:\theta_{n_i}\in\mathcal P\cup\mathcal P^{-1}, k\in\mathbb N\}$. Then, the following are equivalent:
\begin{enumerate}
    \item[(a)] $f\chi_{\mathcal G_{r(\theta)}}-R_{\theta}f\in c_0(\mathcal G)$ for every $\theta\in\mathcal P$.
    \item[(b)] $f\chi_{\mathcal G_{r(\theta)}}-R_{\theta}f\in c_0(\mathcal G)$ for every $\theta\in \langle \mathcal P\rangle$.
    \item[(c)] $f\chi_{\mathcal G_{r(\theta)}}-R_{\theta}f\in c_0(\mathcal G)$ for every $\theta\in \overline{\langle \mathcal P\rangle}^{d_u}$.
\end{enumerate}
In particular, if $\mathcal P$ generates $[[\mathcal G]]$ in the sense that $\langle \mathcal P\rangle=[[\mathcal G]]$ or $\overline{\mathcal P}^{d_u}=[[\mathcal G]]$, and one of the above holds, then $f\in\mathbb S(\mathcal G)$.
\end{lemma}

\begin{proof} It is clear that (c) implies (a) and that (a) implies (b). The second assertion holds by induction and using that $c_0(\mathcal G)$ is stable under multiplication by characteristic functions of source and range sets, under right translation by elements of $[[\mathcal G]]$, and under finite sums.

Assume (b) holds, and assume, per \Cref{rem:3.2}, that $\mathcal P\subset [\mathcal G]$. Let $\theta_n\in \langle \mathcal P\rangle$ such that $\theta_n\to\theta$ uniformly. For $f\in L^{\infty}(\mathcal G)$ satisfying $f-R_{\theta_n}f\in c_0(\mathcal G)$ for every $n$, letting $\gamma\in\mathcal G$ with $s(\gamma)\not\in D_n=\{x:\theta_n^{-1}(x)\neq \theta^{-1}(x)\}$, one sees that
\begin{equation*}
    (f-R_{\theta}f)(\gamma)=f(\gamma)-f(\gamma\theta)=f(\gamma)-f(\gamma\theta_n)=(f-R_{\theta_n}f)(\gamma).
\end{equation*}
Since $\mu(D_n)=d_u(\theta_n,\theta)\to 0$, this shows that $f-R_{\theta}f\in c_0(\mathcal G)$, proving (c).
\end{proof}

\begin{corollary} \label{lem: right countable family}
There exists a countable family $\Sigma=\{\rho_m:m\geq 1\}\subset [\GG]$ such that for any $f\in L^\infty(\GG,\mu\circ\lambda)$, $f - R_\rho f \in c_0(\GG)$ for every $\rho\in [\GG]$ if and only if $f - R_{\rho_m} f \in c_0(\GG)$ for every $m\geq 1$.
\end{corollary}

\begin{proof}
By Lusin-Novikov, we can decompose the countable Borel groupoid $\GG$ into countably many Borel bisections $\GG=\sqcup_m B_m$.  Each local bisection extends, modulo null sets, to a global bisection $\rho_m\in [\GG]$. From the proof of \Cref{claim:generating set}, the statement holds for the family $\{\rho_m\}_{m\geq 1}$. 
\end{proof}

\vspace{1mm}

We continue with a sequence of three lemmas. These will supply some calculations to establish that proper proximality for groupoids is an invariant of the associated von Neumann algebra in Section \ref{sec: invariant vNa}.

\begin{defn}
    We say a measurable subset $S \subset \GG$ is a \emph{$c_0$-set} if the characteristic function $\chi_S \in c_0(\GG)$. Equivalently, $S$ is a $c_0$-set if for any $0<\eps<1$, there exist measurable subsets $B,B'\subset \GG^{(0)}$ with $\mu(B),\mu(B')<\eps$ and a finite-measure subset $F\subset\GG$ such that $S\subset \GG^B \cup \GG_{B'}\cup F$. 
\end{defn}

It is clear that subsets of $c_0$-sets and finite unions of $c_0$-sets are again $c_0$-sets. Moreover, we note that every finite-measure subset of $\GG$ is a $c_0$-set, but there might exist $c_0$-sets of infinite measure in general. For example, consider the trivial transformation groupoid $\GG =\Gamma\ltimes X = \Gamma\times X$ with $\Gamma$ a countable infinite group acting trivially on a diffuse standard probability space $(X,\mu)$. The source and range maps are given by $s(g,x) = r(g,x) = x$ for any $(g,x)\in \Gamma\times X$. Let $X = \sqcup_{n\geq1} B_n$ with $\mu(B_n) = 2^{-n}$, and let $E_n \subset \Gamma$ be a subset with $|E_n| = 2^n$, for each $n\geq1$. Then the set $S\coloneqq \sqcup_n E_n\times B_n \subset \GG$ is Borel and
\[(\mu\circ\lambda)(S) = \sum_n (\mu\circ\lambda)(E_n\times B_n) = \sum_n |E_n|\mu(B_n) = \sum_n 2^n\cdot 2^{-n} = \infty,\]
but $S$ is a $c_0$-set since the head part $\cup_{n\leq N}E_n\times B_n$ always has finite measure, while the tail part is contained in $\GG_{\cup_{n>N}B_n}$, whose source projection can be chosen to have arbitrarily small $\mu$-measure.

\begin{lemma} \label{lem: non c_0-func}
    If $f\in L^\infty(\GG,\nu)$ and $f\not\in c_0(\GG)$, then there exist $\eps>0$, $k\in \{0,1,2,3\}$, and a Borel subset $S\subset \GG$ which is not a $c_0$-set such that $\Real(i^k f)(\gamma) \geq\eps$ for almost every $\gamma\in S$.
\end{lemma}
\begin{proof}
    Suppose towards a contradiction that $\{\gamma\in \GG:|f(\gamma)| > 1/n\}$ is a $c_0$-set for every $n\geq 1$. Given $\eps >0$, choose $n$ with $1/n<\eps$. Then there exist subsets $B,B'\subset \GG^{(0)}$ with $\mu(B),\mu(B')<\eps$ and a Borel subset $F\subset \GG$ of finite measure such that $\{\gamma\in \GG:|f(\gamma)|>1/n\} \subset F\cup \GG^B \cup \GG_{B'}$. On the complement we have $|f| \leq 1/n <\eps$. Hence $f\in c_0(\GG)$, contradicting our assumption. Therefore, there exists $\eps_0 > 0$ such that the set $A\coloneqq \{\gamma\in \GG:|f(\gamma)|>\eps_0\}$ is not a $c_0$-set.

    For each $k\in \{0,1,2,3\}$, let $A_k\coloneqq \{\gamma\in \GG: \Real(i^k f(\gamma)) \geq \eps_0/2\}$. By the triangle inequality, we have $A \subseteq \cup_{k=0}^3 A_k$. If all $A_k$ were $c_0$-sets, then $A$ would also be a $c_0$-set, a contradiction. Hence there exists some $k\in \{0,1,2,3\}$ for which $A_k$ is not a $c_0$-set. We then take such $k$, $S = A_k$ and $\eps = \eps_0/2$.
\end{proof}

\begin{lemma} \label{lem: non c_0-set}
    Let $S\subset (\GG,\nu)$ be a Borel subset that is not a $c_0$-set. Then there exists $\delta > 0$ such that for every finite $\nu$-measure subset $F\subset S$, there exists a local bisection $A\subset S\setminus F$ with $\nu(A) = \mu(s(A)) = \mu(r(A)) \geq\delta$.
\end{lemma}
\begin{proof}
    Since $S$ is not a $c_0$-set, there is $\delta>0$ such that for any measurable subsets $B,B'\subset \GG^{(0)}$ with $\mu(B),\mu(B')<\delta$ and every finite-measure subset $F\subset\GG$, the set $S\setminus (\GG^B\cup \GG_{B'}\cup F)$ is not null.

    Fix a finite $\nu$-measure subset $F\subset S$. By the Lusin-Novikov theorem, we can write $S\setminus F = \sqcup_{n\geq1}B_n$ as a countable disjoint union of local bisections $B_n\in [[\GG]]$. We define recursively local bisections as follows: let $A_1\coloneqq B_1$, and recursively for $n\geq 2$, 
    \[A_n \coloneqq B_n \setminus (\GG^{r(\cup_{k<n}A_k)}\cup \GG_{s(\cup_{k<n}A_k)}).\]
    Since each $A_n \subset B_n$, every $A_n$ is again a Borel local bisection. Also, by construction we have
    \[r(A_n) \cap r(\cup_{k<n}A_k) = \varnothing\:\:\text{ and }\:\: s(A_n)\cap s(\cup_{k<n}A_k) =\varnothing,\]
    so the union $A\coloneqq \cup_{n\geq1}A_n$ is a Borel local bisection because the range sets $r(A_n)$ and the source sets $s(A_n)$ are pairwise disjoint. Moreover, we have $A\subset S\setminus F$ by construction.

    If $\gamma\in B_n\setminus A_n$, then by definition of $A_n$ we have $r(\gamma)\in r(\cup_{k<n}A_k) \subset r(A)$ or $s(\gamma)\in s(\cup_{k<n}A_k) \subset s(A)$, so $\gamma \in \GG^{r(A)}\cup \GG_{s(A)}$. On the other hand, if $\gamma\in A_n$, then $\gamma\in A$ by construction. This proves that $S\setminus F\subset \GG^{r(A)}\cup \GG_{s(A)}$, and thus, $S = \GG^{r(A)}\cup \GG_{s(A)}\cup F$.

    If $\nu(A) < \delta$, then taking $B = r(A),B' = s(A)$ we have $\mu(B),\mu(B')<\delta$ and $S = \GG^{B}\cup \GG_{B'}\cup F$, contradicting our choice of $\delta>0$ at the beginning. Hence, $\nu(A)\geq \delta$.
\end{proof}

\subsection{Proper proximality} We are now ready to introduce proper proximality for countable p.m.p.\ groupoids.

\begin{defn} A countable p.m.p.\ groupoid $(\GG,\nu)$ with unit space $(\mathcal G^{(0)},\mu)$ is said to be properly proximal if there is no $[[\GG]]$-invariant 
state on $\mathbb{S}(\GG)$ whose restriction to $L^{\infty}(\mathcal G^{(0)})$ is normal.
Here, a state $\varphi$ on $\mathbb{S}(\GG)$ is left $[[\GG]]$-invariant if for every $f\in \S(\GG)$ and $\theta\in[[\GG]]$, $\varphi(f\chi_{\GG^{s(\theta)}} - L_\theta f) = 0$.
\end{defn}

\begin{remark}\label{rem:3.5}
    The full pseudogroup $[[\GG]]$ in the definition of proper proximality can again be replaced by just the full group $[\GG]$, so that $\GG$ is properly proximal if there is no $[\GG]$-invariant state on $\mathbb{S}(\GG)$ whose restriction to $L^{\infty}(\mathcal G^{(0)})$ is normal. Indeed, suppose $\varphi$ is a $[\GG]$-invariant state on $\S(\GG)$, let $f\in \S(\GG),\: \theta\in [[\GG]]$, and take a Borel isomorphism $\theta'\in [\GG]$ such that $\restr{\theta'}{ s(\theta)} = \theta$. Let $\tilde{f} = f\chi_{\GG^{s(\theta)}}$, so then $\tilde{f} \in \S(\GG)$ and $L_{\theta} f = L_{\theta} (f\chi_{\GG^{s(\theta)}}) = L_{\theta'} \tilde{f}$ with
    \[\varphi(f\chi_{\GG^{s(\theta)}} - L_{\theta} f) = \varphi(\tilde{f} - L_{\theta'}\tilde{f}) = 0,\]
    since $\varphi$ is $[\GG]$-invariant by assumption. Thus every $[\GG]$-invariant state on $\S(\GG)$ is necessarily $[[\GG]]$-invariant.
\end{remark}

\begin{lemma}\label{lem:uniformtopdensity} Let $\psi$ be a state on $\mathbb S(\mathcal G)$ whose restriction to
$L^\infty(\mathcal G^{(0)})$ is normal. Let $\mathcal P$ be a symmetric countable set and denote by $\langle\mathcal P\rangle=\{\theta_{n_1}\circ\cdots\circ\theta_{n_k}:\theta_{n_i}\in\mathcal P, k\in\mathbb N\}$. Then, the following are equivalent:
\begin{enumerate}
    \item[(a)] $\psi(L_{\theta}f)=\psi(f\chi_{\mathcal G^{s(\theta)}})$ for every $\theta\in\mathcal P$.
    \item[(b)] $\psi(L_{\theta}f)=\psi(f\chi_{\mathcal G^{s(\theta)}})$ for every $\theta\in\langle\mathcal P\rangle$.
    \item[(c)]  $\psi(L_{\theta}f)=\psi(f\chi_{\mathcal G^{s(\theta)}})$ for every $\theta\in\overline{\langle\mathcal P\rangle}^{d_u}$.
\end{enumerate}
In particular, if $\mathcal P$ generates $[[\mathcal G]]$ in the sense that $\langle \mathcal P\rangle=[[\mathcal G]]$ or $\overline{\mathcal P}^{d_u}=[[\mathcal G]]$, and one of the above holds, then the state $\psi$ is left $[[\mathcal G]]$-invariant.
\end{lemma}

\begin{proof} It is clear that (c) implies (a) and that (a) implies (b). 

Per \Cref{rem:3.5}, assume $\mathcal P\subset[\mathcal G]$ and take $\theta_n\in\langle \mathcal P\rangle$ such that $\theta_n\to \theta$ in the uniform topology. Define $D_n=\{x:\theta_n^{-1}(x)\neq \theta^{-1}(x)\}$ and observe that $L_{\theta}f-L_{\theta_n}f=\chi_{\mathcal G^{D_n}}\cdot (L_{\theta}f-L_{\theta_n}f)$. Therefore, by assumption (b),
\begin{equation*}
    |\psi(L_{\theta}f-f)|=|\psi(\chi_{\mathcal G^{D_n}}(L_{\theta}f-L_{\theta_n}f))|\leq 2\|f\|_{\infty}\cdot \psi(\chi_{\mathcal G^{D_n}}).
\end{equation*}
Since $\mu(D_n)\to 0$ and $\psi|_{L^{\infty}(\mathcal G^{(0)})}$ is normal, the above equation converges to zero. Therefore, $\psi(L_{\theta}f)=\psi(f)$ for every $\theta\in\overline{\langle\mathcal P\rangle}^{d_u}$.
\end{proof}

The subsequent result establishes a characterization of non-proper proximality, which will be applied in the section on ergodic decomposition.

\begin{lemma}\label{lem:finitedimcriterion} For a $(\mathcal{G},\nu)$ be a countable p.m.p.\ groupoid with unit space $(\mathcal G^{(0)},\mu)$, the following are equivalent.
\begin{enumerate}
    \item[(i)] $\mathcal{G}$ is not properly proximal.
    \item[(ii)] There exists a normal state $\omega$ on $L^{\infty}(\mathcal{G}^{(0)})$ such that for every finitely generated unital C$^*$-algebra $A\subset\mathbb S(\mathcal G)$ and every finite set $T\subset[\mathcal G]$, there exists a state $\psi:C^*(A,(L_{\theta}A)_{\theta\in T})\to \mathbb C$ such that $\psi(L_{\theta}a)=\psi(a)$ for every $a\in A$ and $\theta\in T$, and whose restriction to $L^{\infty}(\mathcal G^{(0)})\cap C^*(A,(L_{\theta}A)_{\theta\in T})$ is precisely $\omega$.  
\end{enumerate}
\end{lemma}

\begin{proof} By definition (i) implies (ii). For the converse, take $\omega\in L^{\infty}(\mathcal G^{(0)})_*$ as given in (ii). Let $\text{States}(\mathbb S(\mathcal G))$ be the state space of $\mathbb S(\mathcal G)$, endowed with the weak$^*$ topology. This space is compact. For every finitely generated unital C$^*$-subalgebra $A\subset \mathbb S(\mathcal G)$ and finite set $T\subset [\mathcal G]$, define the set $\Omega(A,T)$ consisting of all $\psi\in\text{States}(\mathbb S(\mathcal G))$ such that $\psi(a)=\psi(L_{\theta}a)$ for all $a\in A$ and $\theta\in T$, and whose restriction to $L^{\infty}(\mathcal G^{(0)})\cap C^*(A,(L_{\theta}A)_{\theta\in T})$ equals $\omega$. Each $\Omega(A,T)$ is weak$^*$ closed. Moreover, by assumption (ii), there is a state on $C^*(A,(L_\theta A)_{\theta\in T})$ satisfying the above conditions. By the Hahn-Banach Theorem, such state extends to $\mathbb S(\mathcal G)$, showing that $\Omega(A,T)\neq\emptyset$.

For a finite family $(A_1,T_1),(A_2,T_2),\ldots, (A_n,T_n)$ of finitely generated unital C$^*$-subalgebras $A_i\subset\mathbb S(\mathcal G)$ and finite sets $T_i\subset [\mathcal G]$, define $A=C^*(A_1,\ldots,A_n)$ and $T=T_1\cup\ldots\cup T_n$. From (ii), we obtain a state $\psi:C^*(A,(L_{\theta}A)_{\theta\in T})\to\mathbb C$ that is $T$-invariant on $A$ and whose restriction to $L^{\infty}(\mathcal G^{(0)})\cap C^*(A,(L_{\theta}A)_{\theta\in T})$ is $\omega$; that is, $\cap_{i=1}^n\Omega(A_i,T_i)\neq\emptyset$. In other words, the family of sets $(\Omega(A,T))_{(A,T)}$ has the finite intersection property. Since $\text{States}(\mathbb S(\mathcal G))$ is compact with respect to the weak$^*$-topology, we obtain $\cap_{(A,T)}\Omega(A,T)\neq\emptyset$. The state $\psi:\mathbb S(\mathcal G)\to\mathbb C$ in the intersection is left $[\mathcal G]$-invariant and restricts to $\omega$ on $L^{\infty}(\mathcal G^{(0)})$, which shows that $\mathcal G$ is not properly proximal.
\end{proof}

\subsection{A bidual characterization} In this section, we establish a bidual characterization of proper proximality for groupoids, adapting the methods used in \cite[Theorem 4.3]{BIP21} and \cite[Lemma 8.5]{ding2023properproximality}.

\begin{lemma} \label{lem: relative separation}
Let $E$ be a unital operator system containing a unital $C^*$-subalgebra $B$, and let a group $\Gamma$ act on $E$ by unital complete order isomorphisms preserving $B$. Let $\tau$ be a $\Gamma$-invariant state on $B$. Suppose that there is no $\Gamma$-invariant state $\Phi$ on $E$ satisfying $\restr{\Phi}{B}=\tau$. Then, for every $\eps>0$, there exist $\theta_1,\dots,\theta_d\in\Gamma$, elements $T_1,\dots,T_d\in E$, an element $T_0\in E_+$, and a self-adjoint $b\in B$ with $\tau(b)=0$, such that
\[\left\Vert1+b+T_0+\sum_{k=1}^d(T_k-\theta_k T_k)\right\Vert<\eps.\]
\end{lemma}

\begin{proof}
We work in the real Banach space $E_{\rm sa}$.  Let
\[\VV=\Span_{\R}\{T-\theta(T):T\in E_{\rm sa}, \theta\in\Gamma\}+\{b\in B_{\rm sa}:\tau(b)=0\}.\]
If the conclusion failed for some $\eps>0$, then the distance from $-1$ to $E_+ + \VV$ would be positive. By Hahn-Banach separation, there is a non-zero real bounded functional $\psi$ on $E_{\rm sa}$ such that $\psi(\VV)=0$, $\psi(E_+)\subset [0,\infty)$, and $\psi(1)>0$. Extending by complex linearity and normalizing gives a state $\Phi=\psi/\psi(1)$ on $E$. It follows that $\Phi$ vanishes on the coboundary part $\{T - \theta(T):T\in E,\theta\in \Gamma\}$, so $\Phi$ is $\Gamma$-invariant. Also, $\Phi$ vanishes on $\{b\in B_{\rm sa}:\tau(b)=0\}$, so $\restr{\Phi}{B}=\tau$, a contradiction.
\end{proof}

\begin{lemma} \label{lem: cutdown mazur}
Let $E$ be a normal operator $L^\infty(X, \mu)$-system and $\tau$ a normal state on $L^\infty X$. Suppose $\{x_i\}_i$ is a uniformly bounded net in $E$ such that $x_i\to 0$ in $\widetilde E=(E^\sharp)^*$ with respect to the weak$^*$ topology. Then for every index lower bound $i_0$, every $\eps>0$, and every $\delta>0$, there are indices $i_1,\ldots,i_N\geq i_0$, coefficients $\lambda_1,\cdots \lambda_N\geq 0$ with $\sum_{j=1}^N\lambda_j=1$, and a projection $p\in L^\infty X$ such that writing $x=\sum_{j=1}^N\lambda_j x_{i_j}$ we have
\[\tau(1-p)\leq\delta\:\:\text{ and }\:\: \| p x p\| \leq  \eps .\]
The same assertion holds simultaneously for finitely many such nets indexed by the same directed set, with the same indices, the same convex coefficients, and the same projection $p$.
\end{lemma}

\begin{proof}
Let $s_\tau$ be the seminorm
\[s_\tau(x)=\inf\{\tau(a^*a)^{1/2}\|y\|\tau(c^*c)^{1/2}:x=ayc,\ a,c\in L^{\infty}X,\ y\in E\}.\]
Since $E^\sharp$ is precisely the space of functionals continuous for the strong $L^\infty X$-$L^\infty X$-topology, by Hahn-Banach, weak$^*$ convergence to $0$ in $(E^\sharp)^*$ implies that for each index $i_0$, the element $0$ belongs to the $s_\tau$-closure of the convex hull $\conv\{x_i: i\geq i_0\}$. 

Choose indices $i_1,\ldots,i_N\ge i_0$ and coefficients $\lambda_j\geq 0$, $\sum_j\lambda_j=1$, such that $x=\sum_j\lambda_j x_{i_j}$ satisfies $s_\tau(x)<\eps\delta/8$.  Pick a factorization $x=ayc$ with
$a,c\in L^\infty X$, $y\in E$, and
\[\tau(|a|^2)^{1/2} \| y \| \tau(|c|^2)^{1/2}<\eps\delta/4.\]
Let $A=\tau(|a|^2)^{1/2}$ and $C=\tau(|c|^2)^{1/2}$. If $A=0$ or $C=0$ then $x=0$, and we pick $p=1$. Otherwise put
\[p_a\coloneq 1_{\{|a|\leq \sqrt2 A/\sqrt\delta\}}, \qquad
 p_c \coloneqq 1_{\{|c|\leq \sqrt2 C/\sqrt\delta\}}, \qquad p\coloneqq p_a p_c.\]
By Chebyshev's inequality, we have $\tau(p_a),\tau(p_c) \geq 1-\delta/2$, so $\tau(1-p) \leq \tau(1-p_a)+\tau(1-p_c) \leq\delta$.  Moreover,
\[\| p x p\| = \| paycp\| \leq (\sqrt2 A/\sqrt\delta)\| y\| (\sqrt2 C/\sqrt\delta)<\eps.\]
The finite simultaneous version follows by applying the argument in the finite $\ell^\infty$ direct sum of the involved operator systems.
\end{proof}

We have the following corollary applied to quotients.

\begin{corollary} \label{cor: cutdown quotient}
Let $E$ be a normal $L^\infty X$-C$^*$-algebra and $J\triangleleft E$ a closed ideal that is also an $L^\infty X$-C$^*$-algebra.  If $\{x_i\}$ is a bounded net in $E$ such that $x_i+J\to 0$ weak$^*$ in $((E/J)^\sharp)^*$, then for every index lower bound $i_0$ and every $\eps,\delta>0$, there are indices $i_1,\ldots,i_N \geq i_0$, coefficients $\lambda_j \geq 0$ with $\sum_j\lambda_j=1$, a projection $p\in L^\infty X$ with $\tau(1-p)<\delta$, and an element $u\in J$ such that writing $x=\sum_j\lambda_j x_{i_j}$ we have
\[ \| p(x+u)p\| <\eps . \]
The same holds simultaneously for finitely many nets.
\end{corollary}

For the rest of the section, $(\mathcal G,\nu)$ is a countable p.m.p.\ groupoid with unit space $(\GG^{(0)},\mu)$. 

\begin{lemma} \label{lem: quasi-central approx unit}
Let a countable group $\Gamma$ act on $L^\infty \GG$ by $*$-automorphisms preserving $c_0(\GG)$. Let $F\subset\Gamma$ be a finite subset, and let $(C_n)_{n\geq 1}$ be a sequence of finite subsets of $c_0(\GG)$. Given $\eps >0$ and positive summable numbers $\{\eps_n\}_{n\geq 1}$, there is an increasing approximate unit of positive contractions
\[0=\alpha_0\leq\alpha_1\leq\alpha_2\leq\cdots\leq1, \qquad \alpha_n\in c_0(\GG),\]
such that $\alpha_n \nearrow1$ pointwise almost everywhere on $\GG$, $\| (1-\alpha_{n})c\| < \eps_n$ for all $c\in C_n$, $n\geq 1$, and letting $e_n=\alpha_n-\alpha_{n-1}$, one has
\begin{equation*}
    \begin{aligned}
        \sum_{n=1}^\infty\| e_n-\gamma(e_n)\| <\eps\:\: \text{ for all } \gamma\in F\:\: \text{ and }\:\:
        \sum_{n=1}^\infty\| e_n-\gamma(e_n)\| < \infty\:\: \text{ for all }\gamma\in \Gamma.
    \end{aligned}
\end{equation*}
\end{lemma}

\begin{proof}
By Arveson's quasi-central approximate unit theorem applied to the crossed-product ideal $c_0(\GG)\rtimes\Gamma\triangleleft L^\infty\GG\rtimes\Gamma$, there exists an increasing quasi-central approximate unit $\{F_j\}_{j \geq 1}\subset c_0(\GG)_+$ such that $\lim_{j\to\infty} \| (1-F_j)c\| = 0$ for all $c\in c_0(\GG)$ and $\lim_{j\to\infty} \| F_j-\gamma(F_j)\| = 0$ for all $\gamma\in\Gamma$.

Note that $F_j \nearrow1$ pointwise almost everywhere on $\GG$. Indeed, since $(\GG,\nu)$ is $\sigma$-finite, there exists an increasing sequence of finite-measure subsets $K_1\subset K_2\subset \cdots \subset \GG$ such that $\cup_{i\geq 1}K_i = \GG$ modulo null sets. For each fixed $i$, we have $\chi_{K_i} \in c_0(\GG)$, so $\lim_{j\to\infty}\| (1-F_j) \chi_{K_i}\| = 0$. Hence $F_j(\gamma) \to 1$ for almost every $\gamma \in K_i$, and since $\cup_{i\geq 1}K_i = \GG$ modulo null set, we have $F_j \nearrow1$ pointwise almost everywhere on $\GG$.

We enumerate $\Gamma = \{\gamma_r\}_{r\geq 1}$, and choose a summable sequence $(a_n)_n$ of positive numbers with $\sum_n a_n<\eps/2$. Inductively choose an increasing subsequence $\alpha_n=F_{j(n)}$ so that $\| \alpha_n-\gamma(\alpha_n) \| < a_n$ for all $\gamma\in F$, $\| \alpha_n-\gamma_r(\alpha_n) \| < 2^{-n}$ for all $1\leq r\leq n$, and $\| (1-\alpha_{n})c \| <\eps_n$ for all $c\in C_n$, $n\geq 1$. The last requirement is possible because $C_n$ is finite and $F_j$ is an approximate unit. We have $\alpha_n \nearrow1$ pointwise almost everywhere on $\GG$ since it is a subnet of $F_j$. For $\gamma\in F$, we have
\[\| e_n-\gamma(e_n)\|\leq \| \alpha_n-\gamma(\alpha_n)\| + \| \alpha_{n-1}-\gamma(\alpha_{n-1})\| \leq a_n+a_{n-1},\]
with $a_0=0$, so $\sum_{n=1}^\infty\| e_n-\gamma(e_n)\| <\eps$. For a fixed $\gamma_r\in\Gamma$, the same estimate with $2^{-n}$ holds for all $n>r$ and the finitely many terms $n\leq r$ are harmless for taking the sum. Hence the sum is finite.
\end{proof}

\begin{lemma} \label{lem: approx bidual}
Let $\tau$ be a normal state on $L^\infty \GG^{(0)}$. Let $\theta_1,\ldots,\theta_d\in [\GG]$, $T_1,\ldots,T_d\in \widetilde{\S(\GG)}_{\rm sa}$, $T_0\in \widetilde{\S(\GG)}_+$, and $b\in (L^\infty \GG^{(0)})_{\rm sa}$. Suppose $\eps >0$ and
\[\left\Vert 1+b+T_0+\sum_{k=1}^d(T_k-\widetilde L_{\theta_k}T_k)\right\|<\eps.\]
Let $M\coloneqq \max\{\| b\|, \| T_0 \|, \| T_1\|, \cdots, \| T_d\|\}$. 
Then, for every $\delta>0$, there exist $y_1,\cdots,y_d\in \S(\GG)_{\rm sa}$, $y_0\in \S(\GG)_+$, and a projection $p\in L^\infty \GG^{(0)}$ with $\tau(1-p)<\delta$ such that, writing
\[H=1+b+y_0+\sum_{k=1}^d(y_k-L_{\theta_k}y_k)\in\mathbb S(\GG),\]
we have 
\begin{enumerate}
    \item $\| H\| \leq 1 + 2(d+1)M$, and
    \item $\| pHp\| < 2\eps$.
\end{enumerate}
\end{lemma}

\begin{proof}
Let $\Sigma=\{\rho_m:m\ge1\}$ be the countable family from Corollary \ref{lem: right countable family}. By Goldstine's theorem and convexity of positive elements, we can choose weak$^*$ approximating nets $t^i_k\in L^\infty\GG$ for $T_k$ with $\| t^i_k\| \leq \| T_k\|$, and $t^i_0\in (L^\infty\GG)_+$ converging in the weak$^*$ topology to $T_0$ with $\| t^i_0\| \leq \| T_0\|$. Also choose a net $s^i\in L^\infty\GG$ with $\| s^i\| <\eps$ converging weak$^*$ to $1+b+T_0+\sum_{k=1}^d(T_k-\widetilde L_{\theta_k}T_k)$.

Let $M\coloneqq \max\{\| b\|, \| T_0 \|, \| T_1\|, \cdots, \| T_d\|\}$. Let $\omega\coloneqq \frac{1}{2}(\mu+\tau)$ be a faithful normal state on $L^\infty (\GG^{(0)},\mu)$, where by abusing notation $\mu$ denotes integration on $L^\infty \GG^{(0)}$ with respect to the measure $\mu$. Choose positive summable numbers $\{\eps_n\}_{n\ge1}$ such that $\sum_{n\geq 1}\eps_n <\eps/2$. For each $n \geq 1$, applying Corollary \ref{cor: cutdown quotient} and \Cref{lem: cutdown mazur} simultaneously for the finitely many nets of quotients
\[t^i_k - R_{\rho_m}t^i_k + c_0(\GG) \in L^\infty(\GG)/c_0(\GG),
 \:\:\text{ for all }\: 0\leq k\leq d,\ 1\leq m\leq n,\]
and for the net of elements
\[ 1 + b + t^i_0 + \sum_{k=1}^d(t^i_k-L_{\theta_k}t^i_k)-s^i\in L^\infty\GG,\]
we obtain elements $x_{k,n}\in (L^\infty \GG)_{\rm sa}$ for $1\leq k \leq d$, $x_{0,n}\in (L^\infty \GG)_+$, $s_n\in L^\infty \GG$, a projection $p_n\in L^\infty \GG^{(0)}$, and elements $u_{k,m,n}\in c_0(\GG)$ for $0\leq k\leq d$, $1\leq m\leq n$ such that
\begin{equation} \label{eqn: quotient ineq}
    \begin{aligned}
        &\| x_{k,n}\| \leq M, \qquad \| s_n\| <\eps, \qquad \omega(1-p_n)<2^{-n-1}\delta,\\
        \| p_n(x_{k,n}&-R_{\rho_m}x_{k,n}+u_{k,m,n})p_n\| < \eps_n\:\:\text{ for all }\:  0\leq k\leq d,\ m\leq n,
    \end{aligned}
\end{equation}
and
\begin{equation} \label{eqn: large ineq}
    \left\Vert p_n\left(1+b+x_{0,n} + \sum_{k=1}^d(x_{k,n}-L_{\theta_k}x_{k,n}) - s_n\right)p_n\right\Vert < \eps_n.
\end{equation}
Let $\Gamma$ be the countable group generated by $L_{\theta_1},\ldots,L_{\theta_d}$ and $\{R_{\rho_m}:m \geq 1\}$ acting on $L^\infty\GG$ and the ideal $c_0(\GG)$. For each $n$ let $C_n \coloneqq \{u_{k,m,n}:0\leq k\leq d,1\leq m\leq n\}$ a finite subset of $c_0(\GG)$. By \Cref{lem: quasi-central approx unit}, there exist an approximate unit $(\alpha_n)_n$ converging pointwise to $1$ almost everywhere on $\GG$ with $0 = \alpha_0 \leq \alpha_1 \leq \alpha_2 \leq \cdots\leq 1$, and elements $e_n \coloneqq \alpha_n-\alpha_{n-1} \geq 0$ in $c_0(\GG)$ such that $\sum_n e_n=1$ pointwise almost everywhere, 
\begin{equation} \label{eqn: translation ineq}
    \begin{aligned}
        \sum_{n=1}^\infty \| e_n - L_{\theta_k}e_n&\| < \frac{\eps}{2Md}\:\:\text{ for all }\: 1\leq k\leq d,\:\:\: \sum_{n=1}^\infty\| e_n-R_{\rho_m}e_n\| <\infty \:\:\text{ for all }\: m\geq 1,\text{ and}\\
        &\| (1-\alpha_{n-1})u_{k,m,n}\| <\eps_n, \:\:\text{ for all }\: n\geq 1, 1\leq k \leq d, 1\leq m \leq n.
    \end{aligned}
\end{equation} 
Since $0\leq e_n\leq 1-\alpha_{n-1}$, the last inequality implies \begin{equation} \label{eqn: bound quotient norm}
    \| e_nu_{k,m,n}\| <\eps_n \:\:\text{ for all }\: n\geq 1, 1\leq k \leq d, 1\leq m \leq n.
\end{equation}
Define
\[ y_k \coloneqq \sum_{n\geq 1}e_nx_{k,n},\quad 1\leq k\leq d,
 \qquad y_0\coloneqq \sum_{n\geq 1}e_nx_{0,n}, \qquad s\coloneqq \sum_{n\ge1}e_ns_n.\]
The sums converge weak$^*$ in $L^\infty\GG$ and are bounded. Since $\{e_n\}$ is a positive partition of unity, we have $\| y_k\| \leq M$ for $0\leq k \leq d$, $y_0\geq 0$, and $\| s\| \leq \eps$.

We now show $y_k\in \S(\GG)$ for $0\leq k\leq d$. For this, for each $\ell \geq 1$ we define $P_\ell \coloneqq \bigwedge_{n\ge\ell}p_n\in L^{\infty}\GG^{(0)}$ and first show that $P_\ell (y_k - R_{\rho_m}y_k)P_\ell \in c_0(\GG)$ for every $\ell$ and $m$. Note that since $\omega(1-p_n)$ is summable, we have $\omega(1-P_\ell)\to 0$, so in particular $\mu(1-P_\ell)\to 0$.

Fix $\ell,k$, and $m$, and let $N_0 \coloneqq \max\{m,\ell\}$. Since $\sum_n e_n = 1$, we have
\[P_\ell(y_k-R_{\rho_m}y_k)P_\ell = \sum_{n\geq 1} \left( P_\ell e_n(x_{k,n}-R_{\rho_m}x_{k,n})P_\ell + P_\ell (e_n -R_{\rho_m}e_n) R_{\rho_m}x_{k,n}P_\ell \right).\]
The part with $n<N_0$ belongs to $c_0(\GG)$, because each $e_n\in c_0(\GG)$, $c_0(\GG)$ is an ideal and it is invariant under right translations. Also, the series $\sum_{n\geq N_0}P_\ell(e_n-R_{\rho_m}e_n)R_{\rho_m}x_{k,n}P_\ell$ is a norm-convergent sum of elements of $c_0(\GG)$ because $\| x_{k,n}\| \leq M$ for all $k,n$ and $\sum_n \| e_n-R_{\rho_m}e_n\|<\infty$. It remains to handle the term $A\coloneqq \sum_{n\ge N_0}P_\ell e_n(x_{k,n}-R_{\rho_m}x_{k,n})P_\ell$. For $n\geq N_0 = \max\{m,\ell\}$, we have $m\leq n$ and $P_\ell\leq p_n$, so by (\ref{eqn: quotient ineq}), $\| P_\ell e_n(x_{k,n}-R_{\rho_m}x_{k,n}+u_{k,m,n})P_\ell\| <\eps_n$, and thus we have $\sum_{n\geq N_0} P_\ell e_n(x_{k,n}-R_{\rho_m}x_{k,n}+u_{k,m,n})P_\ell \in c_0(\GG)$ since the series is again a norm-convergent sum of elements of $c_0(\GG)$. Similarly, by (\ref{eqn: bound quotient norm}) we have $\| P_\ell e_n u_{k,m,n}\| <\eps_n$, so the series $\sum_{n\geq N_0} P_\ell e_n u_{k,m,n} \in c_0(\GG)$. Therefore, we have
\[A 
= \sum_{n\geq N_0} P_\ell e_n(x_{k,n}-R_{\rho_m}x_{k,n}+u_{k,m,n})P_\ell -  \sum_{n\geq N_0} P_\ell e_n u_{k,m,n}P_\ell \in c_0(\GG).\]
Combining all terms together, we obtain that $P_\ell (y_k - R_{\rho_m}y_k)P_\ell \in c_0(\GG)$, for every $\ell$.

Since $\mu(1-P_\ell)\to 0$, for any $\eta>0$ there exists $\ell \geq1$ such that $\mu(1-P_\ell)<\eta/2$. Since $P_\ell (y_k - R_{\rho_m}y_k)P_\ell \in c_0(\GG)$, there exist $B,B'\subset \GG^{(0)}$ with $\mu(B),\mu(B')<\eta/2$ and a finite-measure subset $F\subset \GG$ such that 
\[ \| (1-\chi_{\GG^B})(1-\chi_{\GG_{B'}})P_\ell (y_k - R_{\rho_m}y_k)P_\ell (1-\chi_F)\| < \eta/2.\]
If we let $C\subset \GG^{(0)}$ such that $\chi_C = 1-P_\ell$, then $\mu(C)<\eta/2$ and $\mu(B\cup C),\mu(B'\cup C)<\eta$, and by construction it follows that
\[ \| (1-\chi_{\GG^{B\cup C}})(1-\chi_{\GG_{B'\cup C}}) (y_k - R_{\rho_m}y_k) (1-\chi_F)\| < \eta/2.\]
Since $\eta>0$ was arbitrary, this proves that $y_k-R_{\rho_m}y_k\in c_0(\GG)$. Since this holds for every $m$, by Corollary \ref{lem: right countable family} we conclude that $y_k\in \S(\GG)$.

Set $H=1+b+y_0+\sum_{k=1}^d(y_k-L_{\theta_k}y_k)$. Since $\| b\|, \| y_k\| \leq M$ for $0\leq k \leq d$, it is clear that $\|H\|\leq 1+2(d+1)M$, establishing (1). 
Let $p\coloneqq P_1 = \bigwedge_{n\ge1}p_n$, so then $\tau(1-p)\leq 2\omega(1-p) < \delta$. We claim that $\| pHp\| < 2\eps$. 
We have
\[ H-\sum_{n\geq 1}e_n\left(1+b+x_{0,n} + \sum_{k=1}^d(x_{k,n}-L_{\theta_k}x_{k,n})\right) = \sum_{k=1}^d\sum_{n\geq 1}(e_n-L_{\theta_k}e_n)L_{\theta_k}x_{k,n},\]
and by (\ref{eqn: translation ineq}),
\[\left\Vert \sum_{k=1}^d\sum_{n\geq 1}(e_n-L_{\theta_k}e_n)L_{\theta_k}x_{k,n}\right\Vert \leq M\sum_{k=1}^d \sum_{n\geq 1}\| e_n - L_{\theta_k}e_n\| < \frac{\eps}{2}.\]
Moreover, since $p\leq p_n$ for every $n\geq 1$, by (\ref{eqn: large ineq}),
\[\left\Vert p\left(\sum_{n\ge1}e_n\left(1+b+x_{0,n} + \sum_{k=1}^d(x_{k,n}-L_{\theta_k}x_{k,n})-s_n\right)\right)p\right\Vert \leq \sum_{n\geq1}\eps_n < \frac{\eps}{2}.\]
Since $\| s\| = \| \sum e_n s_n\| \leq\eps$, we obtain 
\begin{equation*}
    \begin{aligned}
        \| pHp\| &\leq \left\Vert p\sum_{n\ge1}e_n\left(1+b+x_{0,n} + \sum_{k=1}^d(x_{k,n}-L_{\theta_k}x_{k,n})-s_n\right)\right\Vert + \left\Vert p\sum_{k=1}^d\sum_{n\geq 1}(e_n-L_{\theta_k}e_n)L_{\theta_k}x_{k,n}\right\Vert + \left\Vert p\sum e_n s_n\right\Vert\\
        &< \eps/2 + \eps/2+ \eps = 2\eps.
    \end{aligned}
\end{equation*}
This proves (2) and finishes the proof.
\end{proof}

\begin{thm} \label{lem: bidual characterization}
    Let $(\GG,\nu)$ be a countable p.m.p.\ groupoid with unit space $(\GG^{(0)},\mu)$. Then $\GG$ is properly proximal if and only if $\widetilde{\S(\GG)}\coloneqq \{T\in L^\infty\GG^{\sharp *} :T - R_\theta^{\sharp *}(T) \in c_0(\GG)^{\sharp *}\}$ does not admit a $[\GG]$-invariant state that is normal on $L^\infty \GG^{(0)}$.
\end{thm}
\begin{proof}
    Suppose $\S(\GG)$ admits a left $[\GG]$-invariant state $\varphi$ such that the restriction $\tau\coloneqq \restr{\varphi}{L^\infty \GG^{(0)}}$ is normal. Note that $\tau$ is also left $[\GG]$-invariant. Suppose towards contradiction that $\widetilde{\S(\GG)}$ does not admit any left $[\GG]$-invariant state whose restriction to $L^\infty \GG^{(0)}$ is normal. By \Cref{lem: relative separation}, there exist $\theta_1,\cdots,\theta_d\in [\GG]$, elements $T_1,\cdots, T_d \in \widetilde{\S(\GG)}_{sa}$, $T_0\in \widetilde{\S(\GG)}_+$, and a self-adjoint $b\in L^\infty \GG^{(0)}$ with $\tau(b) = 0$ such that
    \[\left\Vert1+b+T_0+\sum_{k=1}^d(T_k-\widetilde{L}_{\theta_k} T_k)\right\Vert< \frac{1}{8}.\]
    Let $M\coloneqq \max\{\| b\|, \| T_0 \|, \| T_1\|, \cdots, \| T_d\|\}$. By \Cref{lem: approx bidual}, there exist $y_1,\cdots,y_d\in \S(\GG)_{\rm sa}$, $y_0\in \S(\GG)_+$, and a projection $p\in L^\infty \GG^{(0)}$ such that $H=1+b+y_0+\sum_{k=1}^d(y_k-L_{\theta_k}y_k) \in \S(\GG)$ satisfies $\|H\|\leq 1+2(d+1)M$, $\|pHp\|<1/4$ and
    \begin{equation*}
        \tau(1-p) < \bigg(\frac{1}{4(1+2(d+1)M)}\bigg)^2.
    \end{equation*}
    Since $\varphi$ is left $[\GG]$-invariant, $y_0\geq 0$, and $\varphi(b)=\tau(b) = 0$, we have $\varphi(H) = 1+ \varphi(y_0) \geq 1$. On the other hand, by Cauchy-Schwarz inequality, we have
    \[|\varphi(H)| \leq |\varphi(pHp)| + |\varphi((1-p)H)| < \frac{1}{4} + \tau(1-p)^{1/2}\| H\| <\frac{1}{4} + \frac{1}{4} = \frac{1}{2}, \]
    a contradiction. 
\end{proof}

\section{Permanence of proper proximality}\label{sec:permanence}

\noindent In this section, we discuss permanence of proper proximality under the following constructions: inflations, restrictions, finite-index inclusions, direct products, ergodic decompositions, transformation groupoids and measure equivalence. We adapt some of these proofs from the group case \cite{BIP21,dingKEexamples} to the groupoid setting.

\subsection{Inflations and restrictions}

\begin{prop}\label{restriction} Suppose $(\mathcal{G},\nu)$ is a countable p.m.p.\ groupoid with unit space $(\GG^{(0)},\mu)$. Let $A\subset \GG^{(0)}$ be a measurable subset with positive measure. Then the following hold:

\begin{enumerate}
    \item[(i)] Assume $A$ is invariant for $\mathcal{G}$. If $(\mathcal{G},\nu)$ is properly proximal, then $(\mathcal{G}^A_A,\nu_A)$ is properly proximal.
    \item[(ii)] Assume $(\GG,\nu)$ is ergodic. If $(\mathcal{G}^A_A,\nu_A)$ is properly proximal, then $(\mathcal{G},\nu)$ is properly proximal. 
\end{enumerate}
\end{prop}

\begin{proof} (i) Suppose $\mathcal{G}^A_A$ is not properly proximal, and let $\phi_0:\mathbb S(\mathcal{G}^A_A)\to\mathbb C$ be a left $[\mathcal{G}^A_A]$-invariant state whose restriction to $L^{\infty}(A)$ is normal. Let $f\in\mathbb S(\mathcal{G})$ and observe that $f|_{\mathcal{G}^A_A}\in\mathbb S(\GG_A^A)$. Indeed, fix $\theta\in [[\GG_A^A]]\subset[[\GG]]$ and $\eps>0$. Then, there exists measurable subsets $B,B'\subset \GG^{(0)}$ with $\mu(B),\mu(B')\leq\eps\cdot \mu(A)$, and a subset $F\subset\GG$ of finite $\nu$-measure for which $\|(1-\chi_{\GG^B})(1-\chi_{\GG_{B'}})(f\chi_{\mathcal G_{r(\theta)}}-R_{\theta}f)(1-\chi_F))\|_{\infty}\leq\eps$. Notice that letting $C=A\cap B$, $C'=A\cap B'$ and $L=F\cap\GG_{A}^A$, we get that $\mu_A(C),\mu_A(C')\leq\eps$ and $\nu_A(L)<\infty$, and moreover,
\begin{equation*}
    \|(1-\chi_{(\GG_A^A)^C})(1-\chi_{(\GG_A^A)_{C'}})(f\chi_{(\mathcal G_A^A)_{r(\theta)}}-R_{\theta}f)(1-\chi_L)\|_{\infty}\leq\|(1-\chi_{\GG^B})(1-\chi_{\GG_{B'}})(f\chi_{\mathcal G_{r(\theta)}}-R_{\theta}f)(1-\chi_F)\|_{\infty}\leq \eps.
\end{equation*}
Define the state $\phi:\mathbb S(\mathcal{G})\to\mathbb C$ by $\phi(f)=\phi_0(f|_{\mathcal{G}^A_A})$. Thus, $\phi$ is a left $[[\mathcal{G}^A_A]]$-invariant state whose restriction to $L^{\infty}(\GG^{(0)})$ is normal since $\phi|_{L^{\infty}(\GG^{(0)})}=\phi_0|_{L^{\infty}(A)}$ is normal. From the assumption that $A\subset \GG^{(0)}$ is $\mathcal{G}$-invariant, it follows that for each $\theta\in[\mathcal{G}]$, $\theta|_A\in[\mathcal{G}^A_A]$. Therefore, we have $\phi(L_{\theta}f)=\phi_0(L_{\theta|_A}f|_{\mathcal{G}^A_A})=\phi_0(f|_{\mathcal{G}^A_A})=\phi(f)$. Thus, $\phi$ is also left $[[\mathcal{G}]]$-invariant, and hence, $\mathcal{G}$ is not properly proximal.

\vspace{2mm}

(ii) Again we prove the contrapositive. Assume that $\GG$ is not properly proximal. Since $\GG$ is ergodic and $\mu(A)>0$, the saturation of $A$ is a conull set, or equivalently $r(\GG_A)=\GG^{(0)}$ modulo null sets. By Lusin-Novikov Theorem, there exist local bisections $B_n\subset\GG_A$ satisfying $B_1=A$, $r(B_n)$ are pairwise disjoint and $\sqcup_{n\geq 1}r(B_n)=\GG^{(0)}$ modulo null sets.

For $f\in L^\infty(\GG_A^A)$, let $\widehat f\in L^\infty(\GG)$ denote its extension by zero, and define $\Ind_A(f)\coloneqq \sum_{n\ge1}R_{B_n}\widehat f$. The summands have pairwise disjoint sources. More explicitly, for almost every $\gamma\in\GG$,
\[(R_{B_n}\widehat f)(\gamma) =\chi_{r(B_n)}(s(\gamma))\chi_A(r(\gamma))f\bigl(\gamma B_n\bigr).\]
Indeed, if $s(\gamma)\in r(B_n)$, then $\gamma B_n=\gamma r_{B_n}^{-1}(s(\gamma))$ has source in $A$ and range $r(\gamma)$; hence it belongs to $\GG_A^A$ exactly when $r(\gamma)\in A$. It follows that $\Ind_A$ is a positive $*$-homomorphism into the corner $\chi_{\GG^A}L^\infty(\GG)$ and $\Ind_A(1) = \Ind_A(\chi_{\GG_A^A})=\chi_{\GG^A}$. We also note that extension by zero maps $c_0(\GG^A_A)$ into $c_0(\GG)$.

We claim that $\Ind_A(\S(\GG_A^A))\subset\S(\GG)$. Let $\theta\in[\GG]$ be a global bisection. Let $\bar\theta: \GG^{(0)}\to \GG^{(0)}$ be its induced p.m.p.\ automorphism. Thus, for every $x\in \GG^{(0)}$, the unique arrow $r_\theta^{-1}(x)\in\theta$ satisfies $r(r_\theta^{-1}(x))=x$ and $s(r_\theta^{-1}(x))=\bar\theta^{-1}(x)$. For $i,j\ge1$, let $C_{ij}\coloneqq r(B_i)\cap\bar\theta(r(B_j))$. On $C_{ij}$ define the local bisection
\begin{equation*}
   \Sigma_{ij}\coloneqq \Bigl\{  \bigl(r_{B_i}^{-1}(x)\bigr)^{-1} r_\theta^{-1}(x) r_{B_j}^{-1}(\bar\theta^{-1}(x)): x\in C_{ij} \Bigr\}\subset\GG_A^A.
\end{equation*}
This is the corresponding restriction of $B_i^{-1}\theta B_j$. Fix $\gamma\in\GG$ with $r(\gamma)\in A$ and
$x=s(\gamma)\in C_{ij}$, and set $\eta=\gamma r_{B_i}^{-1}(x)\in\GG_A^A$. Associativity gives
\[\gamma r_\theta^{-1}(x)r_{B_j}^{-1}(\bar\theta^{-1}(x))
   =\eta\Bigl[\bigl(r_{B_i}^{-1}(x)\bigr)^{-1} r_\theta^{-1}(x) r_{B_j}^{-1}(\bar\theta^{-1}(x)) \Bigr].\]
Consequently, for $\gamma\in \GG$ with $s(\gamma) \in C_{ij}$,
\begin{equation} \label{eqn: C_ij subtract}
    \bigl(\Ind_A(f)-R_\theta\Ind_A(f)\bigr)(\gamma) =\bigl(f\chi_{(\GG_A^A)_{r(\Sigma_{ij})}}  -R_{\Sigma_{ij}}f\bigr)(\eta).
\end{equation} 
Let $f\in\mathbb S(\GG_A^A)$ and denote by $h\coloneqq \Ind_A(f)-R_\theta\Ind_A(f)$ and $d_{ij}\coloneqq f\chi_{(\GG_A^A)_{r(\Sigma_{ij})}} -R_{\Sigma_{ij}}f\in c_0(\GG_A^A)$. For $N\ge1$, define
\[h_N\coloneqq \sum_{1\le i,j\le N} \chi_{\GG_{C_{ij}}}\,R_{B_i}\widehat d_{ij} \in c_0(\GG).\]
Let $U_N\coloneqq \bigcup_{n>N}r(B_n)$ and $F_N\coloneqq U_N\cup\bar\theta(U_N)$. Since the sets $r(B_n)$ partition $\GG^{(0)}$, we have $\mu(U_N)\to 0$ and therefore $\mu(F_N)\le2\mu(U_N)\to 0$ as $N\to\infty$. Note that $s(\gamma\theta)=\bar\theta^{-1}(s(\gamma))$, so the index associated with $\gamma\theta$ exceeds $N$ precisely when
$s(\gamma)\in\bar\theta(U_N)$. By \eqref{eqn: C_ij subtract}, $h-h_N$ is supported in $\GG_{F_N}=s^{-1}(F_N)$. 

Given $\eps>0$, choose $N$ with $\mu(F_N)<\eps/2$. Since $h_N\in c_0(\GG)$, there exist $B,B'\subset \GG^{(0)}$ with $\mu(B),\mu(B')<\eps/2$ and a finite-measure subset $F\subset \GG$ such that $\| (1-\chi_{\GG^B})(1-\chi_{\GG_{B'}})h_N(1-\chi_{F})\| _\infty\leq \eps$. Since $h-h_N$ is supported in $\GG_{F_N}=s^{-1}(F_N)$, up to replacing $B'$ by its union with $F_N$, we see that $h\in c_0(\GG)$, proving our claim.

Now choose a left $[[\GG]]$-invariant state $\phi \colon \S(\GG)\to\C$ whose restriction to $L^\infty\GG^{(0)}$ is normal. Then there exists $h\in L^\infty \GG^{(0)}$ such that $\phi|_{L^\infty(X)}(f)=\int_X\! fh\,\mathrm d\mu$ for any $f \in L^\infty\GG^{(0)}$. Since $\GG$ is ergodic and $\phi$ is left $[[\GG]]$-invariant, it follows that $h = 1$ and $\phi|_{L^\infty(X)}=\int_X\!\cdot\,\mathrm d\mu$. Define 
\[\phi_A(f)\coloneqq \frac{1}{\mu(A)}\phi\bigl(\Ind_A(f)\bigr), \:\:\text{ for } f\in\S(\GG_A^A).\]
It follows that $\phi(\Ind_A(1))=\phi(\chi_{\GG^A})=\mu(A)$, so $\phi_A$ is a state. If $a\in L^\infty(A)$ is viewed on
$\GG_A^A$ through the range map and $\widetilde a$ denotes its zero extension to $X$, then $\Ind_A(a)=\widetilde a\circ r$. Consequently, $\phi_A(a)=\frac{1}{\mu(A)}\int_A a\,\mathrm{d}\mu$, so $\phi_A|_{L^\infty(A)}$ is normal. Finally, let $\sigma\in[[\GG_A^A]]$. Left and right multiplication commute, and extension by zero is compatible with the left action of $\sigma$, therefore $\Ind_A(L_\sigma f)=L_\sigma\Ind_A(f)$. Since right multiplication does not change the range coordinate, $\Ind_A\bigl(f\chi_{(\GG_A^A)^{s(\sigma)}}\bigr) =\Ind_A(f)\chi_{\GG^{s(\sigma)}}$. Using the left $[[\GG]]$-invariance of $\phi$, we obtain
\[\phi_A(L_\sigma f) =\frac{1}{\mu(A)}\phi(L_\sigma\Ind_A(f)) =\frac{1}{\mu(A)}\phi(\Ind_A(f)\chi_{\GG^{s(\sigma)}}) =\phi_A(f\chi_{(\GG_A^A)^{s(\sigma)}}).\]
Thus $\phi_A$ is left $[[\GG_A^A]]$-invariant. This proves that $\GG_A^A$ is not properly proximal and finishes the proof.
\end{proof}

\subsection{Finite index inclusions} Let $(\GG,\nu)$ be a countable p.m.p.\ groupoid and let $\mathcal H$ be a Borel subgroupoid of $\GG$. For each $x\in\GG^{(0)}$, we have an equivalence relation on $\GG^x\coloneqq \GG\cap r^{-1}(x)$ given by: two elements $\gamma,\delta\in \GG^x$ are equivalent if and only if $\gamma^{-1}\delta\in\mathcal{H}$. The function assigning to each $x\in\GG^{(0)}$ the number of equivalence classes in $\GG^{(0)}$ is Borel and $\GG$-invariant, and hence constant on a conull set if $(\GG,\nu)$ is ergodic. If $(\GG,\nu)$ is ergodic, this constant value is called the index of $\mathcal{H}$ in $\GG$. This definition extends the index of a subrelation of a countable p.m.p.\ equivalence relation given in \cite[Section 1]{FSZ}.

\begin{prop} \label{prop: finite index subgpoid}
    Suppose $(\mathcal{G},\nu)$ is an ergodic countable p.m.p.\ groupoid with unit space $(\GG^{(0)},\mu)$ and $\mathcal{H}\subset\mathcal{G}$ is a finite index ergodic subgroupoid. If $(\mathcal{H},\nu)$ is properly proximal, then $(\mathcal{G},\nu)$ is properly proximal.
\end{prop}

\begin{proof} For $\gamma\in\GG$, we set $\gamma\mathcal{H}\coloneqq \{\gamma\delta\in\GG:\delta\in\mathcal{H}\cap r^{-1}(s(\gamma))\}$. Let $N$ be the index of $\mathcal{H}$ in $\GG$. Since $\mathcal{H}$ is ergodic, we may choose $\{T_1,...,T_N\}\subset[\mathcal{G}]$ such that for all $x\in\GG^{(0)}$, the sets $r_{T_i}^{-1}(x)\mathcal{H}$, with $1\leq i\leq N$, partition $\mathcal{G}\cap r^{-1}(x)$. 

Let $f\in\mathbb S(\mathcal{H})$ and consider $\tilde f=\sum_{i=1}^NR_{T_i}(f\chi_{\mathcal{H}})\in L^{\infty}(\mathcal G,\nu)$. We claim $\tilde f\in\mathbb S(\mathcal{G})$. Fix $\eps>0$ and $\rho\in[[\mathcal{G}]]$. For $\gamma\in\GG_{r(\rho)}$, we have 
\begin{equation*}
    R_{\rho}\tilde{f}(\gamma)=\sum_{i=1}^Nf(\gamma\rho T_i)\delta_{\gamma\rho T_i, \mathcal H}.
\end{equation*}
Let $A_{i,j}\subset\GG^{(0)}$ be the subset of all $x\in s(\rho)$ for which $s_{\rho}^{-1}(x)r_{T_i}^{-1}(x)\mathcal{H}=r_{T_j}^{-1}(y)\mathcal{H}$ with $y=\rho(x)=r(s_{\rho}^{-1}(x))$. Notice that for almost every $x\in s(\rho)$ and for every $1\leq i\leq N$, there exists a unique $1\leq j\leq N$ for which $s_{\rho}^{-1}(x)r_{T_i}^{-1}(x)\mathcal{H}=r_{T_j}^{-1}(y)\mathcal{H}$, and moreover, the assignment $i\mapsto j$ is bijective. Therefore, $s(\rho)=\sqcup_{i=1}^NA_{i,j}=\sqcup_{j=1}^NA_{i,j}$, for every $1\leq i,j\leq N$. Define $\theta_{i,j}\coloneqq \{(s_{\rho}^{-1}(x)r_{T_i}^{-1}(x))^{-1}r_{T_j}^{-1}(y):x\in A_{i,j},\: y=\rho(x)\}$. Then, $\theta_{i,j}\in[[\mathcal{H}]]$ is a local bisection with $s(\theta_{i,j})=T_j^{-1}\rho(A_{i,j})$ and $r(\theta_{i,j})=T_i^{-1}(A_{i,j})$. 

For $\gamma\in\mathcal{G}_{r(\rho)}$, we have that $\rho^{-1}s(\gamma)\in s(\rho)$, and therefore, there exists $j$ for which $\rho^{-1}s(\gamma)\in A_{i,j}$. From this, we have $T_i^{-1}\rho^{-1}(s(\gamma))=\theta_{i,j}T_j^{-1}\rho(\rho^{-1}s(\gamma))=\theta_{i,j}T_j^{-1}(s(\gamma))$, which gives $\gamma\rho T_i=\gamma T_j\theta_{i,j}^{-1}$. Thus, for $\gamma\in\mathcal{G}_{r(\rho)\cap \rho(A_{i,j})}$, we have
\begin{equation*}
    R_{\rho}R_{T_i}(f\chi_{\mathcal{H}})(\gamma)=f(\gamma \rho T_i)\delta_{\gamma\rho T_i,\mathcal{H}}=f(\gamma T_j\theta_{i,j}^{-1})\delta_{\gamma T_j\theta_{i,j}^{-1},\mathcal H}=R_{T_j}R_{\theta_{i,j}^{-1}}(f\chi_{\mathcal{H}})(\gamma).
\end{equation*}
Hence, we obtain that
\begin{equation*}
    \tilde{f}\chi_{\GG_{r(\rho)}}-R_{\rho}\tilde{f}=\sum_{i=1}^NR_{T_i}(f\chi_{\mathcal{H}})\chi_{\GG_{r(\rho)}}-\sum_{i=1}^NR_{\rho T_i}(f\chi_{\mathcal{H}})=\sum_{i=1}^N\sum_{j=1}^N(R_{T_i}-R_{T_i}R_{\theta_{j,i}^{-1}})(f\chi_{\mathcal{H}})\chi_{\GG_{r(\rho)\cap\rho(A_{j,i})}}.
\end{equation*}
Since $f\in \mathbb S(\mathcal H)$ and $\theta_{j,i}\in[[\mathcal{H}]]$, let $B_{j,i},B_{j,i}'\subset\mathcal{H}^{(0)}=\GG^{(0)}$, and $F_{i,j}\subset\mathcal{H}\subset\GG$ satisfying $\mu(B_{j,i}),\mu(B_{j,i}')\leq\eps/N^2$ and $\nu(F_{j,i})<\infty$, and for which $\|(1-\chi_{\GG^{B_{{j,i}}}})(1-\chi_{\GG_{B'_{{j,i}}}})(f\chi_{\GG_{r(\theta_{j,i}^{-1})}}-R_{\theta_{j,i}^{-1}}f)(1-\chi_{F_{j,i}})\|_{\infty}\leq\eps/N^2$. Denote by $B=\cup_{i,j}B_{j,i}$, $B'=\cup_{i,j}T_iB_{j,i}'$ and $F=\cup_{i,j}F_{j,i}T_i^{-1}$ so that $\mu(B),\mu(B')\leq\eps$ and $\nu(F)<\infty$. It follows that 
\begin{align*}
    \|(1-\chi_{\GG^B})&(1-\chi_{\GG_{B'}})(\tilde{f}\chi_{\GG_{r(\rho)}}-R_{\rho}\tilde{f})(1-\chi_{F})\|_{\infty}\leq N\cdot\sup_{1\leq i,j\leq N}\sup_{\gamma\in\GG_{\rho(A_{j,i})}\setminus (\GG^B\cup\GG_{B'}\cup F)}|(R_{T_i}-R_{T_i}R_{\theta_{j,i}^{-1}})(f\chi_{\mathcal{H}})(\gamma)|\\
    &=N\cdot\sup_{1\leq i,j\leq N}\sup_{\gamma\in\GG_{\rho(A_{j,i})}\setminus (\GG^B\cup\GG_{B'}\cup F)}|f(\gamma T_i)\delta_{\gamma T_i,\mathcal{H}}-f(\gamma T_i\theta_{j,i}^{-1})\delta_{\gamma T_i\theta_{j,i}^{-1},\mathcal{H}}|\\
    &=N\cdot\sup_{1\leq i,j\leq N}\sup_{\gamma T_i\in\mathcal{H}_{T_i^{-1}\rho(A_{j,i})}\setminus(\mathcal H^{B_{j,i}}\cup \mathcal H_{B_{j,i}'}\cup F_{j,i})}|f(\gamma T_i)-R_{\theta_{j,i}^{-1}}f(\gamma T_i)|\\
    &\leq N\cdot\sup_{1\leq i,j\leq N}\|(1-\chi_{\GG^{B_{{j,i}}}})(1-\chi_{\GG_{B'_{{j,i}}}})(f\chi_{\GG_{r(\theta_{j,i}^{-1})}}-R_{\theta_{j,i}^{-1}}f)(1-\chi_{F_{j,i}})\|_{\infty}\leq\eps.
\end{align*}
Now, if $\mathcal{G}$ is not properly proximal, there exists a left $[[\mathcal{G}]]$-invariant state $\phi:\mathbb S(\mathcal{G})\to\mathbb C$ whose restriction to $L^{\infty}(\GG^{(0)})$ is normal. Define the state $\phi_0:\mathbb S(\mathcal{H})\to \mathbb C$ given by $\phi_0(f)= \phi(\sum_{i=1}^NR_{T_i}f\chi_{\mathcal{H}})$. It is clear that $\phi_0|_{L^{\infty}(\mathcal H^{(0)})}$ is normal, and since the left action commutes with the right action, $\phi_0$ is left $[[\mathcal{H}]]$-invariant. This means $\mathcal{H}$ is not properly proximal.
\end{proof}

\begin{lemma}\label{extension[H]invariant} Suppose $(\mathcal{G},\nu)$ is an ergodic countable p.m.p.\ groupoid with unit space $(\GG^{(0)},\mu)$ and $\mathcal{H}\subset\mathcal{G}$ is an ergodic subgroupoid. Suppose $\varphi:\mathbb S(\mathcal{G})\to\mathbb C$ is a left $[[\mathcal{H}]]$-invariant state whose restriction to $L^{\infty}(\GG^{(0)})$ is normal. If $\mathcal{H}\subset\mathcal G$ is finite index, then there exists a left $[[\mathcal{G}]]$-invariant state on $\mathbb S(\mathcal{G})$ whose restriction to $L^{\infty}(\GG^{(0)})$ is normal.
\end{lemma}

\begin{proof} Let $N$ be the index of $\mathcal{H}$ in $\GG$. Since $\mathcal{H}$ is ergodic, we may choose $\{T_1,...,T_N\}\subset[\mathcal{G}]$ such that for all $x\in\GG^{(0)}$, the sets $r_{T_i}^{-1}(x)\mathcal{H}$, with $1\leq i\leq N$, partition $\mathcal{G}\cap r^{-1}(x)$. 

Define $\psi:\mathbb S(\mathcal{G})\to\mathbb C$ by $\psi(f)=N^{-1}\sum_{i=1}^N\varphi(L_{T_i^{-1}}f)$. Notice that $\psi$ is a state whose restriction to $L^{\infty}(\GG^{(0)})$ is normal. It remains to show $\psi$ is left $[[\GG]]$-invariant. 

Fix $\theta\in[[\GG]]$. Define $X_{i,j}=\{x\in s(\theta):r_{T_j}^{-1}(\theta(x))\mathcal{H}=s_{\theta}^{-1}(x)r_{T_i}^{-1}(x)\mathcal{H}\}$, and notice that for almost every $x\in s(\theta)$ and for every $1\leq i\leq N$, there exists a unique $1\leq j\leq N$ for which $s_{\theta}^{-1}(x)r_{T_i}^{-1}(x)\mathcal{H}=r_{T_j}^{-1}(\theta(x))\mathcal{H}$, with the assignment $i\mapsto j$ being bijective. Therefore, $s(\theta)=\sqcup_{i=1}^NX_{i,j}=\sqcup_{j=1}^NX_{i,j}$, for every $1\leq i,j\leq N$. Define $\theta_{i,j}\coloneqq \{(r_{T_j}^{-1}(\theta(x)))^{-1}s_{\theta}^{-1}(x)r_{T_i}^{-1}(x):x\in X_{i,j}\}$, which is a local bisection in $[[\mathcal{H}]]$ with $s(\theta_{i,j})=T_i^{-1}X_{i,j}$ and $r(\theta_{i,j})=T_j^{-1}\theta(X_{i,j})$. For $\gamma\in\mathcal{G}^{T_i^{-1}(s(\theta))}$, there exists $1\leq j\leq N$ with $r(\gamma)\in T_i^{-1}(X_{i,j})$, which gives $\theta_{i,j}(r(\gamma))=(T_j^{-1}\circ \theta\circ T_i)(r(\gamma))$ or $(T_j\circ\theta_{i,j})(r(\gamma))=(\theta\circ T_i)(r(\gamma))$. It follows that $L_{T_i^{-1}\theta^{-1}}(f)\cdot \chi_{\GG^{s(\theta_{i,j})}}=L_{\theta_{i,j}^{-1}}L_{T_j^{-1}}(f)\cdot \chi_{\GG^{s(\theta_{i,j})}}=L_{\theta_{i,j}^{-1}}(L_{T_j^{-1}}(f)\chi_{\GG^{r(\theta_{i,j})}})$. Then, using the left $[[\mathcal{H}]]$-invariance of $\varphi$, we obtain
\begin{align*}
    \psi(L_{\theta^{-1}}f)&=\frac{1}{N}\sum_{i=1}^N\varphi(L_{T_i^{-1}\circ\theta^{-1}}f)=\frac{1}{N}\sum_{i=1}^N\varphi(L_{T_i^{-1}\circ\theta^{-1}}(f)\cdot \chi_{\GG^{r(T_i^{-1}\circ\theta^{-1})}})\\
    &=\frac{1}{N}\sum_{i,j=1}^N\varphi(L_{\theta_{i,j}^{-1}}(L_{T_j^{-1}}(f)\chi_{\GG^{r(\theta_{i,j})}}))=\frac{1}{N}\sum_{i,j=1}^N\varphi(L_{T_j^{-1}}(f)\chi_{\GG^{r(\theta_{i,j})}})=\frac{1}{N}\sum_{j=1}^N\varphi(L_{T_j^{-1}}(f\cdot\chi_{\GG^{r(\theta)}}))\\
    &=\psi(f\cdot \chi_{\GG^{r(\theta)}}).
\end{align*}
This shows $\psi$ is left $[[\mathcal{G}]]$-invariant, finishing the proof.
\end{proof}

\begin{prop} \label{prop:f.i.subgroupoid} 
    Suppose $(\GG,\nu)$ is a countable p.m.p.\ groupoid with unit space $(\GG^{(0)},\mu)$ and $\mathcal{H}\subset\GG$ is an ergodic finite index subgroupoid. If $(\GG,\nu)$ is properly proximal, then $(\mathcal H,\nu)$ is properly proximal.
\end{prop}

\begin{proof} If $\mathcal{H}$ is not properly proximal, there exists a state $\varphi_0:\mathbb S(\mathcal{H})\to\mathbb C$ that is left $[[\mathcal{H}]]$-invariant and whose restriction to $L^{\infty}(\GG^{(0)})$ is normal. The same proof as in \Cref{restriction}(i) shows that for $f\in\mathbb S(\GG)$, $f|_{\mathcal{H}}\in\mathbb S(\mathcal H)$, and so, defining $\varphi:\mathbb S(\GG)\to\mathbb C$ by $\varphi(f)=\varphi_0(f|_{\mathcal H})$, we obtain a left $[[\mathcal{H}]]$-invariant state whose restriction to $L^{\infty}(\GG^{(0)})$ is normal. From \Cref{extension[H]invariant}, $\GG$ is not properly proximal.
\end{proof}

Having established that proper proximality passes to finite index subgroups, we can drop the assumption that $A\subset\GG^{(0)}$ is invariant in \Cref{restriction}(i) when $\GG$ is an ergodic equivalence relation $\RR$.

\begin{corollary} \label{cor: f.i. equiv rel restr}
    Suppose $(\RR,\nu)$ is a countable ergodic p.m.p.\ equivalence relation with unit space $(\RR^{(0)},\mu)$ and $A\subset \RR^{(0)}$ is a Borel subset such that $\mu(A) = 1/n$ for some $n\in\mathbb{N}$. If $(\RR,\nu)$ is properly proximal, then the restriction $(\RR_A^A,\nu_A)$ is properly proximal.
\end{corollary}
\begin{proof}
    Since $(\RR,\nu)$ is ergodic, there exists $\varphi\in [\RR]$ such that the sets $\{\varphi^i(A)\}_{i=0}^{n-1}$ form a partition of $\RR^{(0)}$, where $\varphi^i$ denotes the $i$-th iterate of $\varphi$. Write $A_i \coloneqq \varphi^i(A)$,
    and define $\SSS \coloneqq \bigsqcup_{i=0}^{n-1} \RR_{A_i}^{A_i}$, namely the subequivalence relation of $\RR$ consisting of all arrows whose source and range both belong to the same $A_i$.

    We claim that $\SSS$ has index at most $n$ in $\RR$. Fix $x\in \RR^{(0)}$. For each $0\leq i\leq n-1$, consider two elements $(x,y),(x,z)\in \RR$ with $y,z\in A_i$. Write $y=\varphi^i(y_0)$ and $z=\varphi^i(z_0)$ for some $y_0,z_0\in A$. Since $\bigl(\varphi^i(y_0),\varphi^i(z_0)\bigr)=(y,z)\in \RR$ and $\varphi\in[\RR]$, it follows that $(y_0,z_0)\in\RR$. Hence, $(y_0,z_0)\in \RR_A^A$ which further implies $(y,z)=\bigl(\varphi^i(y_0),\varphi^i(z_0)\bigr)\in \RR_{A_i}^{A_i}\subset \SSS$. Therefore, any two arrows in $\RR^x$ whose sources belong to the same $A_i$ are $\SSS$-equivalent. Since the sets $A_i$ partition $\RR^{(0)}$, it follows that $\RR^x$ decomposes into at most $n$ $\SSS$-equivalence classes. Thus, $\SSS$ has index at most $n$ in $\RR$.

    By construction, for every $x\in \RR^{(0)}$, the sets $r^{-1}_{\varphi^i}(x)\SSS$, $i=0,\dots,n-1$, partition $\RR\cap r^{-1}(x)$. Since $\SSS$ has finite index in $\RR$, if $\RR$ is properly proximal, then the same argument as in \Cref{prop:f.i.subgroupoid} shows that $\SSS$ is properly proximal as well. Indeed, although $\SSS$ need not be ergodic, the proof applies here because the explicitly displayed bisections $\varphi^0,\ldots,\varphi^{n-1}$ form a global system of finite-index coset representatives, and the subgroupoid ergodicity in \Cref{prop:f.i.subgroupoid} was used only to obtain such representatives. Finally, since $A$ is $\SSS$-invariant, we have $(\RR_A^A,\nu_A)=(\SSS_A^A,\nu_A)$ is properly proximal by \Cref{restriction}.
\end{proof}

\subsection{Co-amenable inclusions} For countable groups, one has that if $G$ is a properly proximal group and $H\leq G$ is a co-amenable subgroup, then $H$ is also properly proximal \cite[Proposition 4.10(2)]{BIP21}. Since the concept of co-amenability is defined for inclusions of equivalence relations $\mathcal{S}\subset\mathcal{R}$, we extend \Cref{prop:f.i.subgroupoid} to show that proper proximality is preserved under co-amenable inclusions of equivalence relations. We briefly recall the definition introduced in \cite{hayescoamenable}.

Let $\mathcal{S}\subset\mathcal{R}$ be a subequivalence relation. If $\mathcal{R}$ is ergodic, there exists Borel functions $(T_i)_i$ for which $\{[T_i(x)]_{\mathcal{S}}:1\leq i\leq n(x)\}$ partitions each fiber $[x]_{\mathcal{R}}/\mathcal{S}$. The subrelation $\mathcal{S}\subset\mathcal{R}$ is said to be co-amenable if there exists an $[\mathcal{R}]$-invariant mean on $L^{\infty}(\mathcal{R}/\mathcal{S})$ such that $m|_{L^{\infty}X}=\int\cdot\: \,\mathrm{d}\mu$. Here, $\mathcal{R}/\mathcal{S}=\{(x,[T_i(x)]_{\mathcal{S}}):x\in X,i\}$, $L^{\infty}X\subset L^{\infty}(\mathcal{R}/\mathcal{S})$ by $f(x,c)=f(x)$, and $\lambda_{\mathcal{R}/\mathcal{S}}(\theta)f(x,c)=f(\theta^{-1}(x),c)$ for $\theta\in [\mathcal{R}]$, and $f\in L^{\infty}(\mathcal{R}/\mathcal{S})$. As an example, if $\mathcal{S}$ is finite index in $\mathcal{R}$, then $\mathcal{S}\subset\RR$ is co-amenable. Indeed, if $\mathcal{S}\subset\mathcal{R}$ is finite index, each fiber $[x]_{\mathcal{R}}/\mathcal{S}$ is finite and so, there exists Borel functions $T_i \colon X\to X$ for which $\{[T_i(x)]_{\mathcal{S}}:1\leq i\leq n(x)\}$ partitions $[x]_{\mathcal{R}}/\mathcal{S}$, with $n(x)<\infty$ for almost every $x$. The map $\Phi \colon L^{\infty}(\mathcal{R}/\mathcal{S})\to L^{\infty}X$ given by $\Phi(F)(x)\coloneqq \frac{1}{n(x)}\sum_{i=1}^{n(x)}F(x,T_i(x))$ is positive, unital, $[\mathcal{R}]$-equivariant and $\Phi|_{L^{\infty}X}=\text{id}$. By \cite[Theorem 3.5]{hayescoamenable}, $\mathcal{S}\subset\mathcal{R}$ is coamenable.

\begin{prop}\label{prop:coamenable} Suppose $\mathcal{R}$ is a countable ergodic measured equivalence relation on $(X,\mu)$ and $\mathcal{S}\subset\mathcal{R}$ is a co-amenable ergodic subrelation. If $\mathcal{R}$ is properly proximal, so is $\mathcal{S}$. 
\end{prop}

\begin{proof} Using the proof of \Cref{prop:f.i.subgroupoid}, this result will follow after we prove the following version of \Cref{extension[H]invariant}.

\begin{claim}\label{claim:coamenable} If $\mathcal{S}\subset\mathcal{R}$ is a co-amenable ergodic subrelation and $\varphi:\mathbb S(\mathcal{R})\to\mathbb C$ is an $[\mathcal{S}]$-invariant state whose restriction to $L^{\infty}X$ is normal, then there exists an $[\mathcal{R}]$-invariant state on $\mathbb S(\mathcal{R})$ whose restriction to $L^{\infty}X$ is normal.
\end{claim}

Since the restriction of $\varphi$ to $L^{\infty}X$ is normal, there is a function $a\in L^1(X,\mu)_+$ such that $\varphi(g)=\int_X g(x)a(x)\,\mathrm d\mu(x)$, for all $g\in L^{\infty}X$. The $[\mathcal S]$-invariance of $\varphi$ implies that $a$ is $\mathcal S$-invariant, and since $\mathcal S$ is ergodic, $a$ is essentially constant. From $\varphi(1)=1$, we conclude that $a=1$ almost everywhere. Thus $\varphi(g)=\int_X g\,\mathrm d\mu$ for all  $g\in L^{\infty}X$.

Take $f\in\mathbb S(\mathcal R)$ and define $\omega_f\colon L^{\infty}X\to\mathbb C$ by $\omega_f(g)=\varphi(fg)$. Notice that $\omega_f$ is a bounded functional on $L^{\infty}X$ and, moreover, we have $|\omega_f(g)|\leq \varphi(|fg|)\leq \|f\|_{\infty}\cdot \|g\|_1$, where the first inequality holds since $\mathbb S(\mathcal R)$ is a commutative C$^*$-algebra. This shows $\omega_f$ is continuous on $L^{\infty}X$ with respect to the $L^1$-norm, and therefore, extends uniquely to a bounded linear functional on $L^1X$. By the duality $(L^1X)^*=L^\infty X$, there is a unique function $\Phi(f)\in L^{\infty}X$, with $\|\Phi(f)\|_{\infty}\leq\|f\|_{\infty}$, such that 
\begin{equation*} 
    \varphi(gf) = \int_X g(x)\Phi(f)(x)\,\mathrm d\mu(x), \:\:\text{ for all } g\in L^{\infty}X. 
\end{equation*} 
The map $\Phi:\mathbb S(\mathcal R)\to L^{\infty}X$ is linear, positive and unital. For $f,g\geq 0$, $\int_Xg(x)\Phi(f)(x)\,\mathrm d\mu(x)=\varphi(gf)\geq 0$, so $\Phi(f)\geq 0$, and for $b\in L^{\infty}X$, $\int_Xgb\,\mathrm d\mu=\varphi(gb)=\int_Xg\Phi(b)\,\mathrm d\mu$ for every $g\in L^{\infty}X$ so that $\Phi(b)=b$. The same computation also gives $\Phi(bf)=b\Phi(f)$, for $b\in L^{\infty}X$ and $f\in\mathbb S(\mathcal R)$. For $\sigma\in [\mathcal S]$, we have 
\begin{equation*}
    \int_X(g\circ\sigma^{-1})(x)\Phi(L_{\sigma}f)(x)\,\mathrm d\mu(x)=\varphi((g\circ\sigma^{-1})L_{\sigma}f)=\varphi(L_{\sigma}(gf))=\varphi(gf)=\int_Xg(x)\Phi(f)(x)\,\mathrm d\mu(x),
\end{equation*}
showing that $\Phi(L_{\sigma}(f))(\sigma x)=\Phi(f)(x)$, for almost every $x\in X$.

Let $f\in\mathbb S(\mathcal R)$, and define $h_f:\mathcal R/\mathcal S\to\mathbb C$ by $h_f(x,[T(x)]_{\mathcal S})=\Phi(L_Tf)(T(x))$. We first show $h_f$ is well-defined. Let $\rho\in[\mathcal R]$, and set $A=\{x\in X:[T(x)]_{\mathcal S}=[\rho(x)]_{\mathcal S}\}$. Define $\sigma\in[[\mathcal S]]$ with source $\rho(A)$ and range $T(A)$ by $\sigma(\rho(x))=T(x)$, for $x\in A$, and extend $\sigma$ to an element in $[\mathcal S]$. It follows that $\chi_{T(A)\times X}L_Tf=L_{\sigma}(\chi_{\rho(A)\times X}L_{\rho}f)$. Applying $\Phi$, we obtain $\chi_{T(A)}\Phi(L_Tf)=\chi_{\rho(A)}\Phi(L_{\rho}f)\circ\sigma^{-1}$, giving
\begin{equation*}
    (\Phi(L_Tf)\circ T)\chi_A=(\Phi(L_{\rho}f)\circ\rho)\chi_A.
\end{equation*}
Therefore, $h_f$ defines a measurable bounded function on $\mathcal R/\mathcal S$, since $|h_f(x,[T(x)]_{\mathcal S})| \leq \|\Phi(L_Tf)\|_\infty \leq \|L_Tf\|_\infty = \|f\|_\infty$. 

Define $\psi\colon \mathbb S(\mathcal R)\to\mathbb C$ by $\psi(f)=m(h_f)$, where $m:L^{\infty}(\mathcal{R}/\mathcal{S})\to\mathbb C$ is the $[\mathcal{R}]$-invariant mean with $m|_{L^{\infty}X}=\int_X\cdot\:\,\mathrm{d}\mu$ given by co-amenability. Since both $f\mapsto h_f$ and $m$ are positive and unital, $\psi$ is a state. We finish by proving $\psi$ is the desired state. For $\theta\in[\mathcal R]$, we have $\lambda_{\mathcal R/\mathcal S}(\theta)h_f(x,[T(x)]_{\mathcal S})=h_f(\theta^{-1}(x),[(T\circ\theta)(\theta^{-1}(x))]_{\mathcal S})=\Phi(L_{T\circ\theta}f)(T(x))=h_{L_{\theta}f}(x,[T(x)]_{\mathcal S})$, so that $h_{L_{\theta}f}=\lambda_{\mathcal R/\mathcal S}(\theta)h_f$. The $[\mathcal R]$-invariance of $m$ gives $\psi$ is $[\mathcal R]$-invariant since
\begin{equation*}
    \psi(L_{\theta}f)=m(h_{L_{\theta}f})=m(\lambda_{\mathcal R/\mathcal S}(\theta)h_f)=m(h_f)=\psi(f).
\end{equation*}
Finally, let $f\in L^{\infty}X$. Since $L_Tf=f\circ T^{-1}$ and $\Phi$ restricts to the identity on $L^{\infty}X$, we have $h_f(x,[T(x)]_{\mathcal S})=\Phi(L_Tf)(T(x))=(f\circ T^{-1})(T(x))=f(x)$. It follows that $h_f$ is exactly the canonical copy of $f$ in $L^{\infty}(\mathcal R/\mathcal S)$ and hence, $\psi(f)=m(h_f)=\int_Xfd\mu$ is normal. This proves the claim.
\end{proof}

\begin{remark} We expect the preceding result to extend to general countable p.m.p.\ groupoids, rather than only to equivalence relations. We restrict to equivalence relations here because the notion of co-amenability used in the argument has so far been developed only for inclusions of equivalence relations; see \cite{hayescoamenable}.
\end{remark}

\subsection{Direct products} Let $(\GG_1,\nu_1)$ and $(\GG_2,\nu_2)$ be countable p.m.p.\ groupoids with unit space $(\GG_1^{(0)},\mu_1)$ and $(\GG_2^{(0)},\mu_2)$ respectively. The direct product $\mathcal G\coloneqq \GG_1\times \GG_2$ is defined as follows. The unit space of $\mathcal G$ is $\GG_1^{(0)}\times \GG_2^{(0)}$ equipped with the product measure $\mu_1\times\mu_2$. The elements of $\GG$ consist on pairs of arrows $(\gamma_1,\gamma_2)$, where $\gamma_i\in\GG_i$ for $i=1,2$. The source map $s$ and $r$ are defined component-wise; that is, $s(\gamma_1,\gamma_2)=(s_1(\gamma_1),s_2(\gamma_2))\in\mathcal G^{(0)}$ and $r(\gamma_1,\gamma_2)=(r_1(\gamma_1),r_2(\gamma_2))\in\mathcal G^{(0)}$. The multiplication is defined on composable pairs $(\gamma_1,\gamma_2)(\rho_1,\rho_2)=(\gamma_1\rho_1,\gamma_2\rho_2)$, whenever $s(\gamma_1,\gamma_2)=r(\rho_1,\rho_2)$, and the inverse map is defined by $(\gamma_1,\gamma_2)^{-1}=(\gamma_1^{-1},\gamma_2^{-1})$. The product measure $\nu=\nu_1\times\nu_2$ is a well-defined measure on $\mathcal G$.

\begin{remark}\label{rem:prodofbisections} Every $\theta\in [[\mathcal G]]$ can be approximated in the uniform metric by finite
disjoint unions $\rho=\sqcup_{k=1}^n \alpha_k\times \beta_k$, where $\alpha_k\in [[\mathcal G_1]]$, $\beta_k\in [[\mathcal G_2]]$, and the
rectangles are pairwise disjoint. Indeed, since each $\mathcal G_i$ is a countable p.m.p.\ groupoid, we can partition $\mathcal G_i$ into countably many Borel bisections, and hence, $\mathcal G$ can be partitioned into countably many rectangular bisections. Intersecting $\theta$ with these rectangles, we can reduce to the case where $\theta$ is contained in a rectangle $\alpha\times\beta$. Its domain is a measurable set $D\subset s(\alpha)\times s(\beta)$, and this can be approximated in measure by a finite disjoint union of measurable rectangles $\sqcup_{j=1}^m A_j\times B_j$. Then, $\rho_m=\sqcup_{j=1}^m\alpha|_{A_j}\times\beta|_{B_j}$ agrees with $\theta$ except on a set of arbitrarily small measure.
\end{remark}

We begin this subsection with straightforward lemma establishing that proper proximality is preserved when taking the product with a finite measured groupoid.

\begin{lemma}\label{prop:products} Let $(\mathcal G_1,\nu_1)$ and $(\mathcal G_2,\nu_2)$ be countable p.m.p.\ groupoids with unit spaces $(\GG^{(0)}_1,\mu_1)$ and $(\GG^{(0)}_2,\mu_2)$, respectively. Denote by $\mathcal G=\mathcal G_1\times\mathcal G_2$ the direct product of $\mathcal G_1$ and $\mathcal{G}_2$ with measure $\nu$. Assume $\nu_2(\mathcal G_2)<\infty$. Then $\mathcal G$ is properly proximal if and only if $\mathcal G_1$ is properly proximal.
\end{lemma}

\begin{proof}
    Since $\GG_2^{(0)}\subset\GG_2$ and $\mu_2(\GG_2^{(0)})=1$, we have $\nu_2(\GG_2^{(0)})\geq 1$. We first record two facts that will be used later.

    \begin{claim} \label{claim: first}
        If $g\in c_0(\GG_1)$ and $h\in L^\infty(\GG_2)$, then $g\otimes h\in c_0(\GG_1\times\GG_2)$.
    \end{claim}
    \begin{subproof}[Proof of Claim \ref{claim: first}]
        Fix $\eps>0$ and let $M\coloneqq \max\{\|h\|_\infty,1\}$, $\delta \coloneqq \min\{\eps,\eps/M\}$. Since $g\in c_0(\GG_1)$, there exist measurable subsets $B,B'\subset\GG_1^{(0)}$ with $\mu_1(B),\mu_1(B')<\delta$ and a Borel subset $F\subset\GG_1$ with $\nu_1(F)<\infty$ such that $\| (1-\chi_{\GG_1^B})(1-\chi_{(\GG_1)_{B'}}) g(1-\chi_F)\|_\infty \leq\delta$. Consider the sets $B\times\GG_2^{(0)}$, $B'\times\GG_2^{(0)}$, and the Borel subset $F\times\GG_2$ of $\GG_1\times\GG_2$. Their measures satisfy $(\mu_1\times\mu_2)(B\times\GG_2^{(0)})=\mu_1(B)<\eps$ with the same estimate for $B'$, and $(\nu_1\times\nu_2)(F\times\GG_2)=\nu_1(F) \nu_2(\GG_2)<\infty$. If $r_1(\gamma_1)\notin B$, $s_1(\gamma_1)\notin B'$, and $\gamma_1\notin F$, then $|g(\gamma_1)h(\gamma_2)| \leq \delta\|h\|_\infty \leq\eps$. This proves the first claim.
    \end{subproof}
    Define a normal positive unital map $\Psi \colon L^\infty(\GG_1\times\GG_2)\longrightarrow L^\infty(\GG_1)$ by $\Psi(f)(\gamma_1) =\frac{1}{\nu_2(\GG_2)}\int_{\GG_2}f(\gamma_1,\gamma_2)\,\mathrm{d}\nu_2(\gamma_2)$.
    \begin{claim} \label{claim: second}
        $ \Psi\bigl(c_0(\GG_1\times\GG_2)\bigr) \subset c_0(\GG_1)$.
    \end{claim}
    \begin{subproof}[Proof of Claim \ref{claim: second}]
        Let $H\in c_0(\GG_1\times\GG_2)$ and fix $0<\eps<1$. Let $M \coloneqq \max\{\|H\|_\infty,1\}$ and $\kappa \coloneqq \frac{\eps\, \nu_2(\GG_2)}{8M}$. The two functions $y\mapsto |r_2^{-1}(y)|$ and $y\mapsto |s_2^{-1}(y)|$ belong to $L^1(\GG_2^{(0)},\mu_2)$, because their integrals are both $\nu_2(\GG_2)<\infty$.  Hence there exists $\delta>0$, with $\delta<\eps/4$, such that every measurable subset $E\subset\GG_1^{(0)}\times\GG_2^{(0)}$ satisfying $(\mu_1\times\mu_2)(E)<\delta$ also satisfies
        \begin{equation} \label{eqn: absolute range and source}
            \int_E |r_2^{-1}(y)|\,\mathrm d(\mu_1\times\mu_2)(x,y) <\eps \kappa, \qquad \int_E |s_2^{-1}(y)|\,\mathrm d(\mu_1\times\mu_2)(x,y) <\eps \kappa.  
        \end{equation}
        Since $H\in c_0(\GG_1\times\GG_2)$, there exist measurable subsets $E,E'\subset\GG_1^{(0)}\times\GG_2^{(0)}$ with $(\mu_1\times\mu_2)(E),(\mu_1\times\mu_2)(E')<\delta$ and a Borel subset $F\subset\GG_1\times\GG_2$ with $(\nu_1\times\nu_2)(F)<\infty$ such that $ \left\|(1-\chi_{\GG^E})(1-\chi_{\GG_{E'}})H(1-\chi_F)\right\|_\infty\leq\delta$. For $x\in\GG_1^{(0)}$, define
        \[a(x) \coloneqq \nu_2\bigl(\{\gamma_2\in\GG_2:(x,r_2(\gamma_2))\in E\} \bigr), \qquad a'(x) \coloneqq \nu_2\bigl(\{\gamma_2\in\GG_2: (x,s_2(\gamma_2))\in E'\} \bigr).\]
        By Fubini's theorem and \eqref{eqn: absolute range and source},
        \[\int_{\GG_1^{(0)}}a(x)\,\mathrm d\mu_1(x) =\int_E |r_2^{-1}(y)|\, \mathrm d(\mu_1\times\mu_2)(x,y) <\eps \kappa, \qquad \int_{\GG_1^{(0)}}a'(x)\,\mathrm d\mu_1(x)<\eps \kappa.\]
        Therefore the measurable sets $C\coloneqq \{x\in\GG_1^{(0)}:a(x)\geq \kappa\}$ and $C' \coloneqq \{x\in\GG_1^{(0)}:a'(x)\geq \kappa\}$ satisfy $\mu_1(C),\mu_1(C')<\eps$.

        For $\gamma_1\in\GG_1$, let $F_{\gamma_1} \coloneqq \{\gamma_2\in\GG_2:(\gamma_1,\gamma_2)\in F\}$. Again by Fubini's theorem, $\int_{\GG_1}\nu_2(F_{\gamma_1})\,\mathrm d\nu_1(\gamma_1) =(\nu_1\times\nu_2)(F)<\infty$. Consequently, the set $K \coloneqq \{\gamma_1\in\GG_1:\nu_2(F_{\gamma_1})\geq \kappa\}$ has finite $\nu_1$-measure.

        Fix $\gamma_1\notin (\GG_1)^C\cup(\GG_1)_{C'}\cup K$. The set of $\gamma_2\in\GG_2$ for which $(r_1(\gamma_1),r_2(\gamma_2))\in E$ has $\nu_2$-measure smaller than $\kappa$, and the set for which $(s_1(\gamma_1),s_2(\gamma_2))\in E'$ also has $\nu_2$-measure smaller than $\kappa$. Also, $\nu_2(F_{\gamma_1})<\kappa$.  Outside the union of these three subsets of $\GG_2$, we have $|H(\gamma_1,\gamma_2)|\leq\delta$. Hence, $\nu_2(\GG_2) \cdot |\Psi(H)(\gamma_1)|\leq 3M\kappa+\delta \nu_2(\GG_2)$, and therefore $|\Psi(H)(\gamma_1)| \leq 3\eps/8+\eps/4<\eps$. Together with $\nu_1(K)<\infty$, this proves the claim.
    \end{subproof}
    We now prove the two implications.
    Suppose first that $\GG_1\times\GG_2$ is not properly proximal, and let $\omega\colon \S(\GG_1\times\GG_2)\to \C$ be a left $[[\GG_1\times\GG_2]]$-invariant state whose restriction to $L^\infty(\GG_1^{(0)}\times\GG_2^{(0)})$ is normal. We show that the map $\S(\GG_1)\to\S(\GG_1\times\GG_2)$ defined by $f\mapsto f\otimes1$ is well defined. By Remark \ref{rem:prodofbisections} and Lemma \ref{claim:generating set}, it suffices to check well-definedness for a finite disjoint union of rectangular bisections $\rho=\bigsqcup_{k=1}^m\alpha_k\times\beta_k \in[[\GG_1\times\GG_2]]$. Since the source and range sets of these rectangles are pairwise disjoint, one has
    \begin{equation}\label{eqn: rectangular right calculation}
        (f\otimes1)\chi_{(\GG_1\times\GG_2)_{r(\rho)}}-R_\rho(f\otimes1)= \sum_{k=1}^m\left(f\chi_{(\GG_1)_{r(\alpha_k)}}-R_{\alpha_k}f\right)\otimes\chi_{(\GG_2)_{r(\beta_k)}}.
    \end{equation}
    Each first-coordinate factor on the right-hand side belongs to $c_0(\GG_1)$, and hence every summand belongs to $c_0(\GG_1\times\GG_2)$ by Claim \ref{claim: first}.  Thus the left-hand side of \eqref{eqn: rectangular right calculation} belongs to $c_0(\GG_1\times\GG_2)$, and the map $\S(\GG_1)\to\S(\GG_1\times\GG_2)$ given by $f\mapsto f\otimes1$ is well defined.


    Define a state $\varphi\colon \S(\GG_1)\to \C$ by $\varphi(f)=\omega(f\otimes1)$. The restriction of $\varphi$ to $L^\infty(\GG_1^{(0)})$ is normal. For every $\alpha\in[\GG_1]$, $\varphi(L_\alpha f)=\omega((L_\alpha f)\otimes1)=\omega(L_{\alpha\times\mathrm{id}}(f\otimes1))=\omega(f\otimes1)=\varphi(f)$.
    Hence $\varphi$ is left $[\GG_1]$-invariant, and therefore $\GG_1$ is not properly proximal. 
    
    \vspace{2mm}
    
    Conversely, suppose that $\GG_1$ is not properly proximal.  Let $\varphi\colon\S(\GG_1)\to\C$ be a left $[[\GG_1]]$-invariant state whose restriction to $L^\infty(\GG_1^{(0)})$ is normal. Define $\psi\colon\mathbb S(\mathcal G)\to\mathbb C$ by $\psi(f)=\varphi(\Psi(f))$. This is well-defined state since, for $f\in\mathbb S(\GG)$ and $\theta\in[\GG_1]$, $\Psi(f)-R_{\theta}\Psi(f)=\Psi(f-R_{\theta\times\text{id}}f)\in c_0(\GG_1)$ by Claim \ref{claim: second}; that is, $\Psi(\S(\GG))\subset \S(\GG_1)$. 

    For $f\in L^{\infty}(\GG^{(0)})$, $\Psi(f)(\gamma_1)=\frac{1}{\nu_2(\GG_2)}\int_{\GG_2}f(r(\gamma_1),r(\gamma_2))\,\mathrm{d}\nu_2(\gamma_2)=\Psi(f)(r(\gamma_1))\in L^{\infty}(\GG_1^{(0)})$. Thus, the normality of $\psi|_{L^{\infty}(\GG^{(0)})}$ follows from the normality of $\varphi|_{L^{\infty}(\GG_1^{(0)})}$. Finally, to check $\psi$ is left $[[\GG]]$-invariant, by Remark \ref{rem:prodofbisections} and Lemma \ref{lem:uniformtopdensity}, it suffices to consider a finite disjoint union of rectangular bisections $\rho=\bigsqcup_{k=1}^m\alpha_k\times\beta_k \in[[\GG_1\times\GG_2]]$. For $f\in\mathbb S(\GG)$,
    \begin{equation*}
        \Psi(L_{\rho}f)=\sum_{k=1}^{m}\Psi(L_{\alpha_k\times\beta_k}f)=\sum_{k=1}^m\Psi(L_{\alpha_k\times \text{id}}(f\chi_{(\GG_2)^{s(\beta_k)}}))=\sum_{k=1}^mL_{\alpha_k}\Psi(f\chi_{(\GG_2)^{s(\beta_k)}}),
    \end{equation*}
    and thus, $\psi(L_{\rho}f)=\varphi(\Psi(L_{\rho}f))=\sum_{k=1}^m\varphi(\Psi(f\chi_{\GG_2^{s(\beta_k)}})\chi_{\GG_1^{s(\alpha_k)}})=\psi(f\chi_{\GG^{s(\rho)}})$. Altogether, these show that $\mathcal{G}$ is not properly proximal. 
\end{proof}

The previous lemma requires $\nu_2(\mathcal G_2)<\infty$, however, in the next result we remove this assumption for the backwards direction. For the forward direction, we can only drop it in special cases, as will be shown in \Cref{prop:productsvNalg}. We first record a consequence of \Cref{rem:prodofbisections} and \Cref{lem:uniformtopdensity}.

\begin{lemma} \label{claim: prod in small-at-infty} Let $(\GG_1, \nu_1)$ and $(\GG_2,\nu_2)$ be countable p.m.p.\ groupoids with unit spaces $(\GG^{(0)}_1,\mu_1)$ and $(\GG^{(0)}_2,\mu_2)$, respectively. Denote by $\mathcal G=\mathcal G_1\times\mathcal G_2$ the direct product of $\mathcal G_1$ and $\mathcal{G}_2$ with measure $\nu$. Let $T\in (L^\infty(\GG)^{L^\infty \GG^{(0)}\sharp L^\infty \GG^{(0)}})^*$. If $T-\widetilde R_{\alpha\times\id}(T),\:T-\widetilde R_{\id\times\beta}(T) \in (c_0(\GG)^{L^\infty \GG^{(0)}\sharp L^\infty \GG^{(0)}})^*$ for every $\alpha\in [\GG_1]$ and $\beta\in [\GG_2]$, then $T\in\widetilde{\S(\GG)}$.
\end{lemma}

\begin{proof} Denote by $X_1=\GG^{(0)}_1$, $X_2=\GG^{(0)}$ and $X:=\GG^{(0)}=X_1\times X_2$. 

If $\sigma\in[[\GG]]$ is a local bisection, choose a global bisection $\widehat\sigma\in[\GG]$ extending $\sigma$ and define $\widetilde R_\sigma(T)\coloneqq\chi_{\GG_{r(\sigma)}}\widetilde R_{\widehat\sigma}(T)$. This definition is independent of the chosen extension, because two global extensions of $\sigma$ agree on $r(\sigma)$ and hence their right actions agree after multiplication by $\chi_{\GG_{r(\sigma)}}$.

Let $\alpha\in[\GG_1]$ and $\beta\in[\GG_2]$. The two coordinate right actions commute, and therefore
    \[T-\widetilde R_{\alpha\times\beta}(T) =T-\widetilde R_{\alpha\times\id}(T)+\widetilde R_{\alpha\times\id}\bigl(T-\widetilde R_{\id\times\beta}(T)\bigr).\]
Since $\bigl(c_0(\GG)^{L^\infty X\sharp L^\infty X}\bigr)^*$ is invariant under the normal extensions of the right actions, the right-hand side belongs to this weak$^*$ closed ideal. Next let $\rho=\bigsqcup_{k=1}^m\alpha_k\times\beta_k$ be a finite disjoint union of rectangular local bisections. Extend $\alpha_k$ and $\beta_k$ to global bisections $\widehat\alpha_k\in[\GG_1]$ and $\widehat\beta_k\in[\GG_2]$. By the preceding paragraph, $T-\widetilde R_{\widehat\alpha_k\times\widehat\beta_k}(T)\in \bigl(c_0(\GG)^{L^\infty X\sharp L^\infty X}\bigr)^*$. Consequently,
\begin{equation*}
    T\chi_{\GG_{r(\rho)}}-\widetilde R_\rho(T) =\sum_{k=1}^m     \chi_{\GG_{r(\alpha_k)\times r(\beta_k)}} \bigl(T-\widetilde R_{\widehat\alpha_k\times\widehat\beta_k}(T)\bigr)
\end{equation*}
belongs to $\bigl(c_0(\GG)^{L^\infty X\sharp L^\infty X}\bigr)^*$. Finally, let $\theta\in[\GG]$. By \Cref{rem:prodofbisections}, there are finite disjoint unions $\rho_n$ of rectangular bisections such that
\[D_n \coloneqq r(\theta)\mathbin{\triangle}r(\rho_n) \cup\left\{x\in r(\theta)\cap r(\rho_n): (r_\theta)^{-1}(x)\neq(r_{\rho_n})^{-1}(x)\right\}\]
satisfies $\mu(D_n)\to0$. Put $e_n=\chi_{\GG_{D_n}}$. From the definition of the right action, the elements $T-\widetilde R_\theta(T)$ and $T\chi_{\GG_{r(\rho_n)}}-\widetilde R_{\rho_n}(T)$
agree after multiplication by $1-e_n$. This identity first holds for the canonical image of $L^\infty(\GG)$ and then for the whole normal envelope by ultraweak density and normality of the extended right actions. Hence, if $Y_n \coloneqq T-\widetilde R_\theta(T) -\bigl(T\chi_{\GG_{r(\rho_n)}}-\widetilde R_{\rho_n}(T)\bigr)$, then $Y_n=e_nY_n$ and $\|Y_n\|\leq 4\|T\|$.

Let $\Omega$ be a normal functional on $\bigl(L^\infty(\GG)^{L^\infty X\sharp L^\infty X}\bigr)^*$, and let $|\Omega|$ be its absolute value. Then $|\Omega(Y_n)|\leq 4\|T\||\Omega|(e_n)$. The restriction of $|\Omega|$ to the source copy of $L^\infty(X)$ is normal, so $|\Omega|(e_n)\to0$ because $\mu(D_n)\to0$. Therefore,  $T\chi_{\GG_{r(\rho_n)}}-\widetilde R_{\rho_n}(T) \to T-\widetilde R_\theta(T)$ ultraweakly. The terms on the left belong to $\bigl(c_0(\GG)^{L^\infty X\sharp L^\infty X}\bigr)^*$, and this ideal is ultraweakly closed. Thus $T-\widetilde R_\theta(T) \in \bigl(c_0(\GG)^{L^\infty X\sharp L^\infty X}\bigr)^*$. Since $\theta$ was arbitrary, $T\in\widetilde{\S(\GG)}$.
\end{proof}

\begin{prop}\label{prop:productsG1G2}
    Let $(\GG_1, \nu_1)$ and $(\GG_2,\nu_2)$ be countable p.m.p.\ groupoids with unit spaces $(\GG^{(0)}_1,\mu_1)$ and $(\GG^{(0)}_2,\mu_2)$, respectively. Denote by $\mathcal G=\mathcal G_1\times\mathcal G_2$ the direct product of $\mathcal G_1$ and $\mathcal{G}_2$ with measure $\nu$. If $(\GG_1,\mu_1)$ and $(\GG_2,\mu_2)$ are properly proximal, then $(\GG,\mu)$ is properly proximal.
\end{prop}
\begin{proof}
    Denote by $X_1=\GG^{(0)}_1$, $X_2=\GG^{(0)}$ and $X:=\GG^{(0)}=X_1\times X_2$. 
       
    Inside $L^\infty(\GG,\mu) \simeq L^\infty(\GG_1,\mu_1) \overline{\otimes}L^\infty(\GG_2,\mu_2)$, we consider the ideals 
    \begin{equation*}
        \begin{aligned}
            I_1 &\coloneqq \overline{\operatorname{span} \left\{F(1\otimes f_2): F\in L^{\infty}(\GG),\ f_2\in c_0(\GG_2)\right\}}^{\|\cdot\|},\\
            I_2 &\coloneqq \overline{\operatorname{span} \left\{F(f_1\otimes1): F\in L^\infty(\GG),\ f_1\in c_0(\GG_1)\right\}}^{\|\cdot\|}.
        \end{aligned}
    \end{equation*}
    Let $p,q\in (L^\infty(\GG)^{L^\infty X\sharp L^\infty X})^*$ be the projections for which $(I_1^{L^\infty X\sharp L^\infty X})^*  = p (L^\infty(\GG)^{L^\infty X\sharp L^\infty X})^*$ and $(I_2^{L^\infty X\sharp L^\infty X})^* = q(L^\infty(\GG)^{L^\infty X\sharp L^\infty X})^*$ (see \cite{Sakai}). Since $(L^\infty(\GG)^{L^\infty X\sharp L^\infty X})^*$ is abelian, $p,q$ are central projections.
    
    We note that $I_1 I_2 \subset c_0(\GG)$. To show this, we first take $f_1\in c_0(\GG_1)$ and $f_2\in c_0(\GG_2)$ and check that $f_1\otimes f_2\in c_0(\GG)$. Let $M\coloneqq \max\{\| f_1\|,\| f_2\|,2\}$ and fix $\eps >0$. For $i=1,2$, there exist measurable subsets $B_i,B_i' \subset X_i$ with $\mu_i(B_i),\mu_i(B_i') < \eps/M$ and a finite $\nu_i$-measure subset $F_i\subset \GG_i$ such that $\|(1-\chi_{{\GG_i^{B_i}}})(1-\chi_{{(\GG_i)_{B_i'}}})f_i(1-\chi_{F_i})\|_\infty \leq \eps/M$ for $i=1,2$. Let $B\coloneqq B_1\times X_2\cup X_1\times B_2$, $B'\coloneqq B_1'\times X_2\cup X_1\times B_2'$, and $F\coloneqq F_1\times F_2$. Then $\mu(B),\mu(B')<\eps$, the set $F$ has finite $\nu$-measure, and $\left\|(1-\chi_{\GG^{B}})(1-\chi_{{\GG_{B'}}})(f_1\otimes f_2)(1-\chi_{F})\right\|_\infty \leq \eps$. Thus $f_1\otimes f_2\in c_0(\GG)$. If $F,G\in L^\infty(\GG)$, $f_1\in c_0(\GG_1)$, and $f_2\in c_0(\GG_2)$, then $\bigl(F(1\otimes f_2)\bigr)\bigl(G(f_1\otimes1)\bigr) =FG(f_1\otimes f_2)\in c_0(\GG)$. Passing to finite sums and norm closures gives $I_1I_2\subset c_0(\GG)$.
    
    Let $z\in\bigl(L^\infty(\GG)^{L^\infty X\sharp L^\infty X}\bigr)^*$ satisfy $(c_0(\GG)^{L^\infty X\sharp L^\infty X})^* = z(L^\infty(\GG)^{L^\infty X\sharp L^\infty X})^*$.
    Then $pq\leq z$. Indeed, if $(e_\lambda)_\lambda$ and $(f_\kappa)_\kappa$ are contractive approximate units of $I_1$ and $I_2$, respectively, their canonical images converge ultraweakly to $p$ and $q$. Since $e_\lambda f_\kappa\in c_0(\GG)$ for every $\lambda,\kappa$, taking first the ultraweak limit in $\lambda$ and then in $\kappa$ gives $pq\in z\bigl(L^\infty(\GG)^{L^\infty X\sharp L^\infty X}\bigr)^*$.

    The ideals $I_1$ and $I_2$ are invariant under the left and right coordinate actions $\alpha\times\id$ and $\id\times\beta$. For example,
    \[R_{\id\times\beta}\bigl(F(1\otimes f_2)\bigr)=R_{\id\times\beta}(F)(1\otimes R_\beta f_2)\in I_1,\]
    while $R_{\alpha\times\id}$ fixes the factor $1\otimes f_2$; with the remaining cases being identical. Hence $p$ and $q$ are fixed by all coordinate left and right actions. Applying \Cref{claim: prod in small-at-infty} to $p$ and $q$ gives $p,q\in\widetilde{\S(\GG)}$.
    
    Now we consider the embedding $\iota_1\colon L^\infty\GG_1\to L^\infty\GG$ by $\iota_1(f) = f\otimes 1$. Clearly the restriction $\restr{\iota_1}{L^\infty X_1}\colon L^\infty X_1\to L^\infty X \subset L^\infty \GG$ is normal, so $\iota_1$ extends to a normal unital $*$-homomorphism on the biduals $\widetilde{\iota_1}\colon (L^\infty\GG_1^{L^\infty X_1 \sharp L^\infty X_1})^* \to (L^\infty\GG^{L^\infty X \sharp L^\infty X})^*$ such that 
    \[\widetilde{\iota_1}((c_0(\GG_1)^{L^\infty X_1 \sharp L^\infty X_1})^*) \subset (I_2^{L^\infty X\sharp L^\infty X})^* = q(L^\infty(\GG)^{L^\infty X\sharp L^\infty X})^*.\] 
    Moreover, $\widetilde{\iota_1}$ is left and right $[\GG_1]$-equivariant, in the sense that for any $\alpha\in [\GG_1]$, $\beta\in [\GG_2]$ and $T \in (L^\infty\GG_1^{L^\infty X_1 \sharp L^\infty X_1})^*$, we have $\widetilde{R}_{\alpha\times\beta}(\widetilde{\iota_1}(T)) = \widetilde{\iota_1}(\widetilde{R}_{\alpha}(T))$ and $\widetilde{L}_{\alpha\times\beta}(\widetilde{\iota_1}(T)) = \widetilde{\iota_1}(\widetilde{L}_{\alpha}(T))$.
    
    Define $\Phi_1\colon \widetilde{\S(\GG_1)} \to p\widetilde{\S(\GG)}$ by $\Phi_1(T) = p\widetilde{\iota_1}(T)$, for $T\in \widetilde{\S(\GG_1)}$. We check that the image of $\Phi_1$ indeed lies in $\widetilde{\S(\GG)}$. If $\alpha\in[\GG_1]$, then
    \[\Phi_1(T)-\widetilde R_{\alpha\times\id}(\Phi_1(T)) = p\widetilde\iota_1\bigl(T-\widetilde R_\alpha(T)\bigr) \in pq\bigl(L^\infty(\GG)^{L^\infty X\sharp L^\infty X}\bigr)^* \subset \bigl(c_0(\GG)^{L^\infty X\sharp L^\infty X}\bigr)^*.\]
    If $\beta\in[\GG_2]$, then $\Phi_1(T)-\widetilde R_{\id\times\beta}(\Phi_1(T))=0$. Therefore, by  \Cref{claim: prod in small-at-infty}, $\Phi_1(T)\in\widetilde{\S(\GG)}$. Similarly, we let $\iota_2\colon L^\infty\GG_2\to L^\infty\GG$ be the embedding given by $\iota_2(f) = 1\otimes f$, then the embedding $\iota_2$ extends to a normal unital $*$-homomorphism on the biduals $\widetilde{\iota_2}\colon (L^\infty\GG_2^{L^\infty X_2 \sharp L^\infty X_2})^* \to (L^\infty\GG^{L^\infty X \sharp L^\infty X})^*$ that is left and right $[\GG_2]$-equivariant. Define $\Phi_2\colon \widetilde{\S(\GG_2)} \to (1-p)\widetilde{\S(\GG)}$ by $\Phi_2(T) = (1-p)\widetilde{\iota_2}(T)$ for $T\in \widetilde{\S(\GG_2)}$. 
    If $\alpha\in[\GG_1]$, then $\Phi_2(T)-\widetilde R_{\alpha\times\id}(\Phi_2(T))=0$, and if $\beta\in[\GG_2]$, then $\Phi_2(T)-\widetilde R_{\id\times\beta}(\Phi_2(T))=(1-p)\widetilde\iota_2\bigl(T-\widetilde R_\beta(T)\bigr)=0$. \Cref{claim: prod in small-at-infty} then gives $\Phi_2(T)\in\widetilde{\S(\GG)}$ as well. Since $p$ is fixed by the left actions, the maps $\Phi_1$ and $\Phi_2$ satisfy $\widetilde L_{\alpha\times\id}\circ\Phi_1 =\Phi_1\circ\widetilde L_\alpha$ and $\widetilde L_{\id\times\beta}\circ\Phi_2  =\Phi_2\circ\widetilde L_\beta$.
    
    Now suppose towards to contradiction that $\GG_1$ and $\GG_2$ are properly proximal but $\GG = \GG_1 \times \GG_2$ is not. By \Cref{lem: bidual characterization}, there is a left $[\GG]$-invariant state $\omega\colon \widetilde{\S(\GG)}\to \C$ whose restriction to the unit space $L^\infty X = L^\infty X_1\overline{\otimes} L^\infty X_2$ is normal. If $\omega(p) > 0$, then we define a state $\psi_1\colon \widetilde{\S(\GG_1)}\to \C$ by
    \[\psi_1(f) = \frac{1}{\omega(p)}\omega(\Phi_1(f)), \:\:\text{ for }\: f\in \widetilde{\S(\GG_1)}.\]
    Clearly $\psi_1$ is left $[\GG_1]$-invariant by left $[\GG]$-invariance of $\omega$ and the $[\GG_1]$-equivariance of $\Phi_1$. Also, $\psi_1$ is normal on $L^\infty X_1$ because $\psi_1|_{L^\infty X_1}\colon L^\infty X_1 \ni f\mapsto \frac{1}{\omega(p)}\omega(p(f\otimes1))\in\mathbb C$ is dominated by the positive functional $L^\infty X_1 \ni f \mapsto \frac{1}{\omega(p)}\omega(f\otimes 1)\in\mathbb C$ and $\omega$ is normal on $L^\infty X = L^\infty X_1\overline\otimes L^\infty X_2$. Since a positive functional dominated by a normal positive functional on a von Neumann algebra is normal, we conclude that $\psi_1$ is normal on $L^\infty X_1$. By \Cref{lem: bidual characterization}, this contradicts the assumption that $\GG_1$ is properly proximal.

    If $\omega(p) = 0$, then by Cauchy-Schwarz we have $\omega(pT) = 0$ for all $T\in \widetilde{\S(\GG)}$, so $\omega$ factors through the corner $(1-p)\widetilde{\S(\GG)}$. Define a state $\psi_2\colon \widetilde{\S(\GG_2)}\to \C$ by
    \[\psi_2(f) = \omega(\Phi_2(f)), \:\:\text{ for }\: f\in \widetilde{\S(\GG_2)}.\]
    Similar as above, it follows that $\psi_2$ is a left $[\GG_2]$-invariant state on $\widetilde{\S(\GG_2)}$ that is normal on $L^\infty X_2$, which contradicts proper proximality of $\GG_2$ by \Cref{lem: bidual characterization}. Hence, in either case we obtain a contradiction, so $\GG$ must be properly proximal.
\end{proof}

\subsection{Ergodic decompositions} For the results in this section, we recall the ergodic decomposition for a countable p.m.p.\ groupoid $(\mathcal{G},\nu)$ with unit space $(\mathcal G^{(0)},\mu)$, see \cite{Vara63,Hahn} for a proof. 

\begin{thm}\label{thm:disintegration} Let $\mathcal{G}$ be a countable p.m.p.\ groupoid with unit space $(\mathcal G^{(0)},\mu)$. Then there is a standard probability space $(Z,\zeta)$, a Borel, measure-preserving map $\pi:\mathcal G^{(0)}\to Z$, and probability measures $\mu_z\in\text{Prob}(\pi^{-1}(\{z\}))$ such that
\begin{enumerate}
    \item[(i)] $\pi^{-1}(\{z\})$ is $\mathcal G$-invariant for almost every $z\in Z$,
    \item[(ii)] $\mathcal G_z\coloneqq \mathcal G_{\pi^{-1}(\{z\})}^{\pi^{-1}(\{z\})}$ preserves the measure $\mu_z$, and $\mathcal G_z$ is an ergodic countable p.m.p.\ groupoid with unit space $(\pi^{-1}(\{z\}),\mu_z)$,
    \item[(iii)] for every $f:\mathcal G^{(0)}\to\mathbb C$ bounded and Borel function we have that $z\mapsto \int_{\mathcal G^{(0)}}f\,\mathrm{d}\mu_z$ is Borel, and
    \begin{equation*}
        \int_{\mathcal G^{(0)}}f\,\mathrm{d}\mu=\int_Z\left(\int_{\pi^{-1}(\{z\})}f\,\mathrm{d}\mu_z\right)\,\mathrm{d}\zeta(z).
    \end{equation*}
\end{enumerate}
\end{thm}

As is usual in arguments involving the ergodic decomposition, we will rely on a measurable selection theorem. Recall that if $X$ is a standard Borel space, then $A\subset X$ is \textit{analytic} if it is the image of a standard Borel space under a Borel map. Such sets are measurable with respect to every Borel probability measure on $X$ (see \cite{Kechris}). For a proof of the following theorem, one can check \cite[Theorems 18.1, 21.10]{Kechris}.

\begin{thm}[{Jankov-von Neumann Theorem}]\label{thm:selection} Let $X,Y$ be standard Borel spaces and $P\subseteq X\times Y$ be a Borel subset. Let $\pi_X(P)=\{x\in X:$ there exists $y\in Y$ with $(x,y)\in P\}$. Then there is a map $t:\pi_X(P)\to Y$ with $(x,t(x))\in P$ for every $x\in \pi_X(P)$ and such that $t$ is $\mu$-measurable for every Borel probability measure $\mu$ on $X$.
\end{thm}

\begin{prop}\label{prop:disintegrationS(G)} 
Let $(\mathcal G,\nu)$ be a countable p.m.p.\ groupoid with unit space $(\mathcal G^{(0)},\mu)$. Let $\pi\colon (\mathcal G^{(0)},\mu)\to(Z,\zeta)$ be the space of ergodic components, $\mu_z\in \text{Prob}(\pi^{-1}(\{z\}))$ be the ergodic decomposition of $\mu$, and consider the disintegration $(\mathcal G,\nu)=\int_Z(\mathcal G_z,\nu_z)\,\mathrm{d}\zeta(z)$. Then, $\mathbb S(\mathcal G)=\int_Z^{\oplus}\mathbb S(\mathcal G_z)\,\mathrm{d}\zeta(z)$ inside $L^\infty(\mathcal G,\nu)=\int_Z^{\oplus} L^\infty(\mathcal G_z,\nu_z)\,\mathrm{d}\zeta(z)$ in the following sense: if $f=(f_z)_z\in L^\infty(\mathcal G)$, then $f\in \mathbb S(\mathcal G)$ if and only if $f_z\in \mathbb S(\mathcal G_z)$ for
$\zeta$-almost every $z$.
\end{prop}

\begin{proof} Let $f\in\mathbb S(\mathcal G)$ and let $\mathcal P\subset[[\mathcal G]]$ be a symmetric countable generating subset. Up to extending each of the local bisections to a global bisection and taking the symmetric family of these extensions, we may actually assume $\PP\subset [\GG]$. Denote by $\langle \mathcal P\rangle = \{\theta_{n_1}\circ\ldots\circ\theta_{n_k} : \theta_{n_i}\in\mathcal P,k\in\mathbb N\rangle$. Since $\mathcal P$ is countable, so is $\langle\mathcal P\rangle$, and thus, we can enumerate $\langle\mathcal P\rangle=\{\sigma_n\}_{n=1}^{\infty}$. Since $f\in\mathbb S(\mathcal G)$, for every $n$, $f-R_{\sigma_n}f\in c_0(\mathcal G)$. Denote by $g_n=f-R_{\sigma_n}f\in c_0(\mathcal G)$.

\begin{claim}\label{claim:c_0G} If $g\in c_0(\mathcal G)$, then $g_z\in c_0(\mathcal G_z)$ for $\zeta$-almost every $z\in Z$.
\end{claim}

\begin{subproof}[Proof of Claim \ref{claim:c_0G}] Fix $\eps>0$. Since $g\in c_0(\mathcal G)$, for every $m\geq 1$ there exist measurable sets $B_m,B_m'\subset\mathcal G^{(0)}$ and a finite $\nu$-measure set $F_m\subset\mathcal G$ such that $\mu(B_m),\mu(B_m')\leq \eps\cdot 2^{-m}$ and $\|(1-\chi_{\mathcal G^{B_m}})(1-\chi_{\mathcal G_{B'_m}})g(1-\chi_{F_m})\|_\infty\leq \eps$. From the disintegration, we have $\mu(B_m)=\int_Z\mu_z(B_m\cap\mathcal G_z^{(0)})\,\mathrm{d}\zeta(z)$, and similarly for $B_m'$. Hence, by Chebyshev's inequality, denoting by $C_m \coloneqq \{z:\mu_z(B_m\cap\mathcal G_z^{(0)})\geq\eps\}$, one has $\zeta(C_m)\leq 2^{-m}$, and similarly for $B_m'$. Observe that $\sum_m\zeta(C_m)\leq 1<\infty$. By the Borel-Cantelli Lemma, $\zeta(\lim\sup C_m)=0$; that is, for almost every $z\in Z$, there exists $m_z\in\mathbb N$ for which $\mu_z(B_m\cap\mathcal G_z^{(0)}),\mu_z(B_m'\cap\mathcal G_z^{(0)})\leq\eps$ for all $m\geq m_z$. Moreover, for every $m\in\mathbb N$, $\nu(F_m)=\int_Z\nu_z(F_m\cap\mathcal G_z)\,\mathrm{d}\zeta(z)<\infty$, which means $\nu_z(F_m\cap\mathcal G_z)<\infty$ for almost every $z\in Z$. After intersecting the resulting countably many conull subsets of $Z$, we may assume that this finiteness holds simultaneously for every $m \geq 1$.

Denote by $D_z=B_{m_z}\cap\mathcal G_z^{(0)}$, $D_z'=B_{m_z}'\cap \mathcal G_z^{(0)}$ and $F_z=F_{m_z}\cap\mathcal G_z$. Observe that $\mu_z(D_z),\mu_z(D_z')\leq\eps$, and $\nu_z(F_z)<\infty$. From the $c_0$ estimate of $g$, for almost every $z\in Z$ we get 
\begin{equation*}
    \|(1-\chi_{ (\GG_z)^{D_z}})(1-\chi_{(\GG_z)_{D_z'}})g_z(1-\chi_{F_z})\|_{\infty}\leq\eps.
\end{equation*}
For every $k\geq 1$, apply the preceding argument for $\eps = 2^{-k}$, and intersect the resulting countably many conull subsets of $Z$. For every $z$ in this intersection and every $k$, the defining estimate for $c_0(\GG_z)$ holds with $\eps = 2^{-k}$. This gives $g_z\in c_0(\mathcal G_z)$.
\end{subproof}

By the claim above, there is a conull set $Z_n\subset Z$ such that for every $z\in Z_n$, $f_z-R_{\sigma_{n,z}}f_z\in c_0(\mathcal G_z)$, where $\sigma_{n,z}=\sigma_n|_{\mathcal G_z^{(0)}}$. Define $Z_0=\cap_{n\geq 1}Z_n$, and notice that $\zeta(Z_0)=1$, and for every $z\in Z_0$ and every $n$, $f_z-R_{\sigma_{n,z}}f_z\in c_0(\mathcal G_z)$. For each $z\in Z_0$, the family $\mathcal P_z=\{\sigma_{z}:\sigma\in\mathcal P\}$ generates $[[\mathcal G_z]]$. Indeed, this follows because any local bisection in $[[\mathcal G_z]]$ can be seen as a local bisection in $[[\mathcal G]]$, $\mathcal G_z$ is invariant, and by the generating property of $\mathcal P$. Applying \Cref{claim:generating set} to $\mathcal G_z$, $f_z$, and the generating family $\mathcal P_z$ gives that $f_z\in \mathbb S(\mathcal G_z)$ for all $z\in Z_0$, as required.

\vspace{2mm}

For the converse part, let $f\in L^{\infty}(\mathcal G)$ with decomposition $(f_z)_{z\in Z}$ such that $f_z\in \mathbb S(\mathcal G_z)$. 
Fix $\theta\in[\mathcal G]$ and $\eps>0$. Since $\pi^{-1}(\{z\})$ is $\mathcal G$-invariant, $\theta_z\coloneqq \theta|_{\pi^{-1}(\{z\})}\in[\mathcal G_z]$, so by the assumption $f_z\in\mathbb S(\mathcal G_z)$, we obtain the existence of measurable subsets $B_z,B_z'\subset \mathcal G_z^{(0)}$ with $\mu_z(B_z),\mu_z(B_z')\leq\eps/2$ and a set of finite $\nu_z$-measure $F_z\subset\mathcal G_z$ for which $\Phi(B_z,B_z',F_z,f_z)\coloneqq \|(1-\chi_{\mathcal G_z^{B_z}})(1-\chi_{(\mathcal G_z)_{B_z'}})(f_z-R_{\theta_z}f_z)(1-\chi_{F_z})\|\leq\eps$. 

We now use a reduction for the groupoid $\mathcal G$ that we take from the proof of \cite[Proposition 3.23]{Kida-TuckerDrob}. Let $\mu_n$ be the counting measure on the set $I_n=\{1,\ldots,n\}$, for $n\in\mathbb N$, and let $\mu_{\infty}:=\mu_{\mathbb N}$ be the countaing measure on the set $I_{\infty}:=\mathbb N$. Since $s:\mathcal G\to\mathcal G^{(0)}$ is countable-to-one, define $X_n=\{x\in\GG^{(0)}:|s^{-1}(x)|=n\}$, for $n\in\mathbb N\cup\{\infty\}$, and notice that the sets $X_n$ form a $\GG$-invariant partition of $\GG^{(0)}$. Indeed, if $x\in X_n$ and $\gamma\in\GG_x$, letting $y=r(\gamma)$ we see that $\rho\mapsto \rho\gamma^{-1}$ gives a bijection between $s^{-1}(x)$ and $s^{-1}(y)=s^{-1}(r(\gamma))$. This shows $r(\gamma)\in X_n$, and thus, $r(\GG_x)\subset X_n$. Since the statement is preserved under restriction to invariant measurable subsets, it suffices to prove this proposition for each restriction $\GG^{X_n}_{X_n}$; the general case then follows by taking the countable direct sum over the non-null $X_n$. Hence, we may assume $\GG^{(0)}=X_n$, for some $n\in\mathbb N\cup\{\infty\}$ with $\mu(X_n)>0$. By Lusin's Theorem, we can find an isomorphism of measure spaces $\varphi:(I_n\times\mathcal G^{(0)},\mu_{n}\times\mu)\to (\mathcal G,\nu)$ such that $\varphi(1,x)=x\in\mathcal G^{(0)}$ and $s(\varphi(k,x))=x$ for every $k\in I_n$ and for $\mu$-almost every $x\in\mathcal G^{(0)}$. We thus assume 
that $(\mathcal G,\nu)=(I_n\times \GG^{(0)},\mu_{n}\times\mu)$ with $(\mathcal G^{(0)},\mu)=(\{1\}\times \GG^{(0)},\delta_{\{1\}}\times \mu)$ and with source map $s:\mathcal G\to\mathcal G^{(0)}$ given by $s((k,x))=(1,x)$, for $\nu$-almost every $(k,x)\in\mathcal G$. Let $Z_0\subset Z$ denote the collection of all $z\in Z$ for which $\mu_z$ is atomless on $\mathcal G_z^{(0)}$, and $Z_m\subset Z$ all $z\in Z$ such that $\mu_z$ is uniformly distributed on $m$ points, for $m\geq 1$. It follows that $(\pi^{-1}(Z_m))_{m\geq 0}$ is a partition of $\mathcal G^{(0)}$ into $\mathcal G$-invariant sets. Hence, we may take $Z=Z_m$ for some $m$, and from the proof of \cite[Theorem 3.18]{Glasner}, we may further assume $(\GG^{(0)},\mu)=(Y\times Z,\mu_Y\times\zeta)$ for a standard probability space $(Y,\mu_Y)$ with $\pi:\mathcal G^{(0)}\to Z$ given by $\pi(1,(y,z))=z$. Then, for each $z\in Z$, the probability measures $\mu_z$ on $\mathcal G_z^{(0)}=\{1\}\times Y\times\{z\}$ and $\nu_z$ on $\mathcal G_z=I_n\times Y\times\{z\}$ are respectively given by $\mu_z=\delta_{\{1\}}\times\mu_Y\times\delta_{\{z\}}$ and $\nu_z=\mu_{n}\times\mu_Y\times\delta_{\{z\}}$.

We claim that the sets $B_z,B'_z,F_z$ from the definition of $f_z\in\mathbb S(\mathcal G_z)$ may be chosen measurably in $z\in Z$. To see this, use the groupoid reduction above, so that $\mathcal G_z^{(0)}=\{1\}\times Y\times\{z\}$ and $\mathcal G_z=I_n\times Y\times\{z\}$. Thus, after identifying $\mathcal G_z^{(0)}$ with $Y$ and $\mathcal G_z$ with $I_n\times Y$, the possible choices of $B_z,B'_z,F_z$ may be regarded as elements of the standard measure algebras $\mathcal B(Y,\mu_Y)$ and $\mathcal B(I_n\times Y,\mu_{n}\times\mu_Y)$. Consider the set $\Omega$ of all quadruples $(z,B,B',F)\in Z\times \mathcal B(Y)\times\mathcal B(Y)\times \mathcal B(I_n\times Y)$ such that
\begin{equation*}
    \mu_Y(B),\mu_Y(B')\leq \eps/2,\:\:(\mu_{n}\times\mu_Y)(F)<\infty,\:\:\text{ and }\:\: \Phi(B,B',F,f_z)\leq\eps.
\end{equation*}
Here, $B,B'$ are viewed as subsets of $\mathcal G_z^{(0)}$ and $F$ as a subset of $\mathcal G_z$ through the above identifications. The set $\Omega$ is Borel. Indeed, the maps $B\mapsto \mu_Y(B)$, $F\mapsto (\mu_{n}\times\mu_Y)(F)$ are Borel on the corresponding measure algebras. Moreover, since $f=(f_z)_z$ and $\theta=(\theta_z)_z$ are measurable fields, the map $(z,B,B',F)\mapsto \Phi(B,B',F,f_z)$ is Borel on bounded measurable fields. Therefore $\Omega$ is Borel. From the assumption $f_z\in\mathbb S(\GG_z)$ for almost every $z$, the projection of $\Omega$ onto $Z$ contains a conull subset of $Z$. Hence, by the Selection \Cref{thm:selection}, there exists a $\zeta$-measurable section $t:Z\to \mathcal B(Y)\times\mathcal B(Y)\times\mathcal B(I_n\times Y)$, defined on a conull subset of $Z$, such that $(z,t(z))\in \Omega$ for almost every $z$. This shows that we may assume that $z\mapsto B_z$, $z\mapsto B'_z$, and $z\mapsto F_z$ are measurable fields satisfying the above estimates for almost every $z\in Z$.

Let $M>0$ be such that the set $Z_M=\{z\in Z:\nu_z(F_z)\leq M\}$ satisfies $\zeta(Z\setminus Z_M)\leq\eps/2$. Define 
\begin{equation*}
    B=\bigsqcup_{z\in Z}B_z\cup \pi^{-1}(Z\setminus Z_M),\:\: B'=\bigsqcup_{z\in Z}B_z'\cup\pi^{-1}(Z\setminus Z_M)\:\:\text{ and }\:\:F=\bigsqcup_{z\in Z_M}F_z.
\end{equation*}
It follows that $\mu(B),\mu(B')\leq\eps$, and that $\nu(F)\leq M<\infty$, and thus,
\begin{align*}
    \|(1-\chi_{\mathcal G^B})&(1-\chi_{\mathcal G_{B'}})(f-R_{\theta}f)(1-\chi_F)\|_{\infty}={\text{ess}\sup}_{z\in Z}\sup_{\gamma\in \mathcal G_z\setminus (\mathcal G^{B_z}\cup\mathcal G_{B_z'}\cup F)}|f_z(\gamma)-f_z(\gamma\theta_z)|\leq\eps.
\end{align*}
The equality follows from the invariance of $\mathcal G_z$, and the inequality follows because for $\gamma\in\mathcal G_z$ with $r(\gamma)\not\in B$, we have $z=\pi(r(\gamma))\in Z_M$, so the fact that $\gamma\in\mathcal G_z\setminus F$ implies $\gamma\not\in F_z$. Therefore, $f\in\mathbb S(\mathcal G)$, as wanted.
\end{proof}

\begin{thm}\label{thm:ergodicdecomp0} Let $(\mathcal G,\nu)$ be a countable p.m.p.\ groupoid with unit space $(\mathcal G^{(0)},\mu)$. Let $\pi:(\mathcal G^{(0)},\mu)\to (Z,\zeta)$ be the space of ergodic components, $\mu_z\in\text{Prob}(\pi^{-1}(\{z\}))$ be the ergodic decomposition of $\mu$, and consider the disintegration $(\mathcal{G},\nu)=\int_Z(\mathcal G_z,\nu_z)\,\mathrm{d}\zeta(z)$. The following hold:
\begin{enumerate}
    \item If $(\mathcal G,\nu)$ is properly proximal, then $(\mathcal{G}_z,\nu_z)$ is properly proximal for $\zeta$-almost every $z\in Z$.
    \item Assume $\mathcal G^{(0)}$ is completely atomic. Then the converse holds. 
\end{enumerate}
\end{thm}

\begin{proof} 

Let $E_0 \coloneqq \{z\in Z:\GG_z\text{ is not properly proximal}\}$. Suppose that $\zeta^*(E_0)>0$, where $\zeta^*$ denotes outer measure. Choose a measurable hull $E\subset Z$ of $E_0$, so that $E_0\subset E$ and $\zeta(E)=\zeta^*(E_0)>0$. Recall that whenever $Q\subset Z$ is measurable in the $\zeta$-completion and $E_0\subset Q$, we have $\zeta(E\setminus Q)=0$.

For every $z\in E_0$, choose a state $\varphi_z \colon \S(\GG_z)\to \C$ that is left $[\GG_z]$-invariant and whose restriction to $L^\infty(\GG_z^{(0)})$ is normal. There exists $h_z\in L^1(\GG_z^{(0)},\mu_z)_+$ with $\int_{\GG_z^{(0)}}h_z\,\mathrm{d}\mu_z=1$ such that
\[ \varphi_z(a)=\int_{\GG_z^{(0)}}a h_z\,\mathrm d\mu_z,\:\:\text{ for all }\: a\in L^\infty(\GG_z^{(0)}).\]
The left $[\GG_z]$-invariance of $\varphi_z$ implies that $h_z$ is $\GG_z$-invariant. Since $\GG_z$ is ergodic, $h_z$ is essentially constant, and the normalization above gives $h_z=1$. Therefore, $\varphi_z|_{L^\infty(\GG_z^{(0)})} = \int_{\GG_z^{(0)}}\!\cdot\,\mathrm{d}\mu_z$ for all $z\in E_0$.

Define a normal state $\omega_E\colon L^{\infty}(\mathcal G^{(0)})\to\mathbb C$ by 
\begin{equation*}
    \omega_E(f)=\frac{1}{\zeta(E)}\int_E\varphi_z(f_z)\,\mathrm{d}\zeta(z)=\frac{1}{\zeta(E)}\int_E\int_{\GG_z^{(0)}}f_z\,\mathrm{d}\mu_z\,\mathrm{d}\zeta(z).
\end{equation*}
To show $\mathcal G$ is not properly proximal, we show $\mathcal{G}$ satisfies the finite dimensional criterion from \Cref{lem:finitedimcriterion} with $\omega_E$ as the normal state on $L^{\infty}(\mathcal G^{(0)})$. Towards this, take a finitely generated C$^*$-subalgebra $A\subset\mathbb S(\mathcal G)$ and $T\subset[\mathcal G]$ a finite subset. Denote by $\mathcal A\coloneqq C^*(A,(L_{\theta}A)_{\theta\in T})$ and define $T_z=\{\theta_z \coloneqq \theta|_{\mathcal G_z^{(0)}}:\theta\in T\}\subset [\mathcal G_z]$.

Choose a norm-dense sequence $(a_n)_{n\geq1}$ in the unit ball of $A$, a norm-dense sequence $(b_n)_{n\geq1}$ in the unit ball of $\AAA\cap L^\infty(\GG^{(0)})$, and a norm-dense sequence $(c_n)_{n\geq1}$ in the unit ball of $\AAA$ whose range contains $1$, every $a_n$, every $b_n$, and every $L_\theta a_n$ with $\theta\in T$. By \Cref{prop:disintegrationS(G)} and by the disintegration for left translations, together with a countable-intersection argument, we obtain the existence of a conull Borel set $Z_{A,T}\subset Z$ such that, for every $z\in Z_{A,T}$, we have $\theta_z\in[\GG_z]$ for every $\theta\in T$, $(c_n)_z\in\mathbb S(\GG_z)$ for every $n\geq 1$, and $(L_\theta a_n)_z=L_{\theta_z}(a_n)_z$ for $ \theta\in T,\ n\geq1$. For $z\in Z_{A,T}$, set $A_z=C^*((a_n)_z:n\geq1)$ and $\AAA_z=C^*((c_n)_z:n\geq1)$.
Then $A_z\subset\AAA_z\subset\mathbb S(\GG_z)$,
and the sections $z\mapsto(a_n)_z$ and $z\mapsto(c_n)_z$ make $(A_z)_{z\in Z_{A,T}}$ and $(\AAA_z)_{z\in Z_{A,T}}$ Borel fields of separable unital $C^*$-algebras.

By \cite[Proposition 5.4]{VaesWouters}, the weak$^*$ compact balls $(\operatorname{ball}(\AAA_z^*))_{z\in Z_{A,T}}$ form a Borel field of compact metrizable spaces. Inside their total space, the state spaces form a Borel subfield. More explicitly, using the dense family $(c_n)_z$, the states on $\mathcal A_z$ are described by the countably many conditions $\rho(1_z)=1$ and $\rho((c_n)_z^*(c_n)_z)\in[0,\infty)$ for $n\geq 1$. Let $\Omega$ be the set of all pairs $(z,\rho)$ such that $z\in Z_{A,T}$, $\rho$ is a state on $\AAA_z$, and $\rho((L_\theta a_n)_z)=\rho((a_n)_z)$ for $\theta\in T,\ n\geq1$, while $\rho((b_n)_z) = \int_{\GG_z^{(0)}}(b_n)_z\,\mathrm{d}\mu_z$ for $n\geq1$. The set $\Omega$ is Borel, since evaluation of a dual functional on a Borel section is Borel, and the preceding requirements are countably many equalities of Borel scalar-valued functions.

Let $P\subset Z$ be the image of the projection of $\Omega$ onto $Z$. Then $P$ is analytic and hence measurable in the $\zeta$-completion. For every $z\in E_0\cap Z_{A,T}$, the restriction $\varphi_z|_{\AAA_z}$ belongs to the fiber of $\Omega$ over $z$. Consequently, $E_0\subset P\cup(Z\setminus Z_{A,T})$.
Since $Z\setminus Z_{A,T}$ is null and $E$ is a measurable hull of $E_0$, it follows that $\zeta(E\setminus P)=0$. By the Jankov-von Neumann selection theorem, there is a $\zeta$-measurable section $z\mapsto \rho_z\in\operatorname{States}(\AAA_z)$ defined on $P$ such that $(z,\rho_z)\in\Omega$ for every $z\in P$.

Define $\psi\colon \AAA\to\CC$ by
\[\psi(f)=\frac{1}{\zeta(E)}\int_{E\cap P}\rho_z(f_z)\,\mathrm{d}\zeta(z).\]
The function $z\mapsto\rho_z(f_z)$ is measurable for every $f\in\AAA$, because $z\mapsto f_z$ is a Borel section of $(\AAA_z)_z$ and $z\mapsto\rho_z$ is a measurable section of the dual field. This is a state since it is positive, and $\psi(1)=\zeta(E\cap P)/\zeta(E)=1$ because $\zeta(E\setminus P)=0$. By the first family of equalities in the definition of $\Omega$ and the density of $(a_n)_n$, we have $\psi(L_\theta a)=\psi(a)$ for $a\in A$ and $\theta\in T$. By the second family of equalities and the density of $(b_n)_n$, for every $f\in\AAA\cap L^\infty(\GG^{(0)})$ we have
\[\psi(f)=\frac{1}{\zeta(E)}\int_{E\cap P}\int_{\GG_z^{(0)}}f_z\,\mathrm{d}\mu_z\,\mathrm d\zeta(z)=\omega_E(f).\]
Therefore, by \Cref{lem:finitedimcriterion}, $\mathcal G$ is not properly proximal. This finishes (1).

\vspace{4mm}

Suppose now $\mathcal G^{(0)}$ is completely atomic. Assume that $\mathcal G$ is not properly proximal, and let $\sigma\in\mathbb S(\mathcal G)^*$ denote a left $[[\mathcal G]]$-invariant state whose restriction to $L^\infty(\mathcal G^{(0)})$ is normal. Since the unit space of $\mathcal G$ is completely atomic, so is $Z$. After discarding a null subset, we regard $Z$ as a countable set of atoms, each satisfying $\zeta(\{z\})>0$. \Cref{prop:disintegrationS(G)} then gives
\begin{equation*}
    \mathcal G=\bigsqcup_{z\in Z}\mathcal G_z\quad\text{ and }\quad\mathbb S(\mathcal G)=\bigoplus_{z\in Z}\mathbb S(\mathcal G_z).
\end{equation*}
In this case, the prior decomposition of $\mathbb S(\mathcal G)$ is much stronger, as we show below. Let $z\in Z$, $a\in\mathbb S(\mathcal G_z)$ and define the extension of $a$ to $L^{\infty}(\mathcal G)$ by $\widetilde{a}|_{\mathcal G_z}=a$ and $\widetilde{a}|_{\mathcal G\setminus\mathcal G_z}=0$. Notice that $\widetilde{a}\in\mathbb S(\mathcal G)$. To see this, take $\theta\in[\mathcal G]$, which from the invariance of $\mathcal G_w$ satisfies $\theta|_{\mathcal G_w^{(0)}}\in[\mathcal G_w]$ for almost every $w\in Z$, and take $\gamma\in\mathcal G$ and $w\in Z$ for which $\gamma\in\mathcal G_w$. Then, $\widetilde{a}$ satisfies
\begin{equation*}
    (\widetilde{a}-R_{\theta}\widetilde{a})(\gamma)=\begin{cases}
        (a-R_{\theta|_{\mathcal G_z^{(0)}}}a)(\gamma)&\text{ if }w=z,\\
        0&\text{ otherwise.}
    \end{cases}
\end{equation*}
Since $a-R_{\theta_z}a\in c_0(\GG_z)$, its zero extension belongs to $c_0(\GG)$. Indeed, fix $\eps>0$ and choose measurable sets $B_z,B_z'\subset\GG_z^{(0)}$ and $F_z\subset\GG_z$ such that $\mu_z(B_z),\mu_z(B_z')<\eps$, $\nu_z(F_z)<\infty$, and $\|(1-\chi_{\GG_z^{B_z}})(1-\chi_{(\GG_z)_{B_z'}})(a-R_{\theta_z}a)(1-\chi_{F_z})\|_\infty\leq\eps$. Regarded as subsets of $\GG^{(0)}$ and $\GG$, respectively, they satisfy $\mu(B_z)=\zeta(\{z\})\mu_z(B_z)<\eps$, $\mu(B_z')<\eps$, and $\nu(F_z)=\zeta(\{z\})\nu_z(F_z)<\infty$. The same norm estimate therefore proves that the zero extension of $a-R_{\theta_z}a$ belongs to $c_0(\GG)$. Hence $\widetilde a\in\S(\GG)$.

Define $p_z=\chi_{\mathcal G_z}$ and let $E=\{z\in Z:\sigma(p_z)>0\}$. By the $\mathcal G$-invariance of $\mathcal G_z^{(0)}$, it follows that $p_z=\chi_{\mathcal G_z^{(0)}}\circ s = \chi_{\mathcal G_z^{(0)}}\circ r$. Because $\sum_z\chi_{\mathcal G_z^{(0)}}=1$ and the restriction $\sigma|_{L^{\infty}(\mathcal G^{(0)})}$ is normal, there must exist some $z\in Z$ for which $\sigma(\chi_{\mathcal G_z^{(0)}})>0$, which consequently implies $\sigma(p_z)>0$. Thus, $E\neq\emptyset$ and since $\zeta$ is a completely atomic measure, we conclude that $\zeta(E)>0$. 

For $z\in E$, define $\sigma_z \colon \mathbb S(\mathcal G_z)\to\mathbb C$ by $\sigma_z(a)=\sigma(\widetilde{a})/\sigma(p_z)$. The zero extension is a $*$-homomorphism from $\mathbb S(\GG_z)$ into the corner $p_z\mathbb S(\GG)$ which sends the unit of $\mathbb S(\GG_z)$ to $p_z$. Hence $\sigma_z$ is positive and $\sigma_z(1)=\sigma(p_z)/\sigma(p_z)=1$, so $\sigma_z$ is a state. Let $\theta\in[\GG_z]$. Since $\GG_z^{(0)}$ is invariant, the bisection $\widehat\theta =\theta\sqcup (\GG^{(0)}\setminus\GG_z^{(0)})$ belongs to $[\GG]$, where the second term denotes the identity bisection on the complementary components. Moreover, $\widetilde{L_\theta a}=L_{\widehat\theta}\widetilde a$. The left $[[\GG]]$-invariance of $\sigma$ therefore implies that $\sigma_z$ is left $[\GG_z]$-invariant. If $a\in L^\infty(\GG_z^{(0)})$, then its zero extension belongs to $L^\infty(\GG^{(0)})$, and the zero-extension map on the unit-space algebras is normal. Hence $\sigma_z|_{L^\infty(\GG_z^{(0)})}$ is normal. Therefore, $\GG_z$ is not properly proximal for every $z\in E$. This proves (2).
\end{proof}

\subsection{Transformation groupoids} Let $(\mathcal{G},\nu)$ be a countable p.m.p.\ groupoid with unit space $(\mathcal{G}^{(0)},\mu)$ and let $(A,\tau)$ be a standard probability space. Suppose $\GG\curvearrowright A$ is an action with anchor map $t:A\to\GG^{(0)}$, and form the associated transformation groupoid $\GG\ltimes A=:\widetilde{\GG}$ as in Section \ref{Sec:actionsofgroupoids}. 

Fix a basis $\mathcal B$ for the groupoid $\mathcal{G}$; that is, $\mathcal B$ is a countable collection of partial bisections $\{B_i\}_i\subset[[\mathcal G]]$ with $\mathcal G^{(0)},B_i^{-1}\in\mathcal{B}$ for each $i$ and such that $\mathcal G=\sqcup_iB_i$ (see \cite{AD13,groupoidfactor}). Let $E\subset[[\widetilde{\mathcal G}]]$ be a partial bisection. Notice that for each $x\in \widetilde s(E)$ there exists a unique $\gamma_x\in \mathcal G$ for which $(\gamma_x,x)\in E$. With respect to the basis $\mathcal B$, there exists a unique $i\in\mathbb N$ for which $\gamma_x\in B_i$. Thus, define the measurable sets $E_i=\{x\in \widetilde s(E):$ there exists $\gamma\in B_i$ with $(\gamma,x)\in E\}\subset A$ and notice that they form a partition of $s(E)$; that is, $\widetilde s(E)=\sqcup_iE_i$. 

Conversely, given a partial bisection $B\in[[\mathcal G]]$, if $E\subseteq t^{-1}(s(B))$ is the measurable set for which the map $E\to A$ mapping $x\mapsto\gamma_x\cdot x$ is injective (where $\gamma_x\in B$ is the unique element satisfying $s(\gamma_x)=t(x)$), then the set $B\times_{\GG^{(0)}} E=\{(\gamma,x)\in\widetilde{\mathcal G}:\gamma\in B,\: x\in E\text{ and }s(\gamma)=t(x)\}$ is a partial bisection of $\widetilde{\mathcal G}$.

\begin{thm} \label{prop: transf preserv prop prox}
Let $(\mathcal G,\nu)$ be a countable p.m.p.\ groupoid over the unit space $(\mathcal G^{(0)},\mu)$ and let $(A,\tau)$ be a standard probability space. Suppose $\alpha \colon \mathcal{G}\curvearrowright A$ is a measure preserving action with anchor map $t:A\to\GG^{(0)}$ such that that $t_*\tau\sim\mu$ on $\GG^{(0)}$. If $\mathcal G$ is properly proximal, then the action groupoid $\widetilde{\mathcal G}=\mathcal G\ltimes A$ is properly proximal. Conversely, if the action $\GG\acts A$ is free and $\widetilde{\GG}$ is properly proximal, then $\GG$ is properly proximal.

In particular, if $\Gamma\curvearrowright (X,\mu)$ is an essentially free p.m.p.\ action, then $\Gamma$ is properly proximal if and only if $\Gamma\ltimes X$ is properly proximal. 
\end{thm}

\begin{proof} We define a map $\Phi\colon L^{\infty}\mathcal G\to L^\infty(\widetilde{\mathcal G})$ by 
\[\Phi(f)(\gamma,z) = f(\gamma), \quad f\in L^\infty\mathcal G,\: (\gamma,z)\in \widetilde{\mathcal G}.\]
We show $\Phi(\mathbb S(\mathcal{G}))\subset\mathbb S(\widetilde{\mathcal G})$. Let $\mathcal B=\{B_i\}_i$ be a basis for the groupoid $\GG$. Fix $E\in[[\widetilde{\mathcal G}]]$ and $\eps>0$. From above, there exists a countable partition of measurable subsets $\widetilde s(E)=\sqcup_{i=1}^{\infty}E_i$ with $E_i\subset A$ satisfying that $E_i=\{x\in \widetilde{s}(E):(\gamma_x,x)\in E\text{ with }\gamma_x\in B_i\}$. 
Observe that for $(\gamma,a)\in \widetilde{\mathcal G}$ we have
\begin{align*}
    R_{E}\Phi(f)(\gamma,a)&=\Phi(f)((\gamma,a)E)\delta_{a\in \widetilde{r}(E)}=\sum_{i=1}^{\infty}\Phi(f)((\gamma,a)E)\delta_{a\in B_iE_i}=\sum_{i=1}^{\infty}f(\gamma B_i)\delta_{a\in B_iE_i}\\
    &=\sum_{i=1}^{\infty}R_{B_i}f(\gamma)\delta_{a\in B_iE_i}.
\end{align*}
Since $\tau(\widetilde s(E))<\infty$, there exists $N\in\mathbb N$ for which $\tau(\sqcup_{i=N}^{\infty}E_i)\leq\eps/3$. Now, for each $B_i$, $1\leq i\leq N-1$, let $C_i,C_i'\subset\mathcal{G}^{(0)}$ be such that $\mu(C_i),\mu(C_i')\leq\eps/(3N)$ and $F_i\subset\mathcal{G}$ with finite $\nu$-measure satisfying $\|(1-\chi_{\mathcal{G}^{C_i}})(1-\chi_{\mathcal{G}_{C_i'}})(f\chi_{\mathcal{G}_{r(B_i)}}-R_{B_i}f)(1-\chi_{F_i})\|_{\infty}\leq\eps$. 

Since $t_*\tau\sim\mu$, let $h$ be the Radon-Nikodym derivative $\mathrm{d}(t_*\tau)/\mathrm{d}\mu$, which is integrable and non-negative $\mu$-almost everywhere. By Chebyshev's inequality, there exists $M>0$ such that $\mu(\{y\in\GG^{(0)}:h(y)\geq M\})\leq 1/M\int h(y)\,\mathrm{d}\mu(y)\leq \eps/3$. Denote by $C\coloneqq \cup_{i=1}^{N-1}t^{-1}(C_i)$, $C'\coloneqq\cup_{i=1}^{N-1}t^{-1}(C_i')\cup\sqcup_{i=N}^{\infty}B_iE_i\cup t^{-1}(\{h(y)\geq M\})$ and notice that $\tau(C),\tau(C')\leq\eps$ from $t_*\tau\sim\mu$ by making $\eps>0$ smaller, if necessary. Let $F=\cup_{i=1}^{N-1}(F_i\times_{\GG^{(0)}} A)\cap \widetilde{\GG}_{t^{-1}(\{h(y)\leq M\})}$. Observe that,
\begin{align*}
    \tilde\nu(F_i\times_{\GG^{(0)}}A\cap &\widetilde{\GG}_{t^{-1}(\{h(y)\leq M\})})=\int_A |F_i\times_{\GG^{(0)}}A\cap \widetilde{\GG}^x_{t^{-1}(\{h(y)\leq M\})}|\,\mathrm{d}\tau(x)\leq\int_{t^{-1}(\{h(y)\leq M\})} |F_i\cap \GG^{t(x)}|\,\mathrm{d}\tau(x)\\
    &=\int_{\{h(y)\leq M\}}|F_i\cap \GG^z|h(z)\,\mathrm{d}\mu(z)\leq M\cdot \nu(F_i)<\infty,
\end{align*}
and thus, $\widetilde{\nu}(F)<\infty$. With these in hand, we have
\begin{equation*}
    \begin{aligned}
    \|(1-&\chi_{\widetilde{\mathcal G}^C})(1-\chi_{\widetilde{\mathcal G}_{C'}})(\Phi(f)\chi_{\widetilde{\mathcal G}_{\widetilde{r}(E)}}-R_{E}\Phi(f))(1-\chi_F)\|_{\infty}=\sup_{(\gamma,a)\in\widetilde{\mathcal{G}}_{\widetilde{r}(E)}\setminus(\widetilde{\mathcal G}^{C}\cup \widetilde{\mathcal G}_{C'}\cup F)}|\Phi(f)(\gamma,a)-R_{E}\Phi(f)(\gamma,a)|\\
    &=\max_{1\leq i\leq N-1}\sup_{(\gamma,a)\in\widetilde{\mathcal{G}}_{B_iE_i}\setminus(\widetilde{\mathcal G}^{C}\cup \widetilde{\mathcal G}_{C'}\cup F)}|f(\gamma)-R_{B_i}f(\gamma)|\leq\max_{1\leq i\leq N-1}\sup_{\gamma\in\mathcal{G}_{r(B_i)}\setminus(\mathcal{G}^{C_i}\cup\mathcal{G}_{C_i'}\cup F_i)}|f(\gamma)-R_{B_i}f(\gamma)|\\
    &=\max_{1\leq i\leq N-1}\|(1-\chi_{\mathcal{G}^{C_i}})(1-\chi_{\mathcal{G}_{C_i'}})(f\chi_{\mathcal{G}_{r(B_i)}}-R_{B_i}f)(1-\chi_{F_i})\|_{\infty}\leq\eps,
\end{aligned}
\end{equation*}
so $\Phi(f)\in\mathbb S(\widetilde{\mathcal G})$.

Now, if $\widetilde{\mathcal{G}}$ is not properly proximal, there exists a left $[[\widetilde{\mathcal{G}}]]$-invariant state $\varphi:\mathbb S(\widetilde{\mathcal G})\to\mathbb C$ whose restriction to $L^{\infty}(A)$ is normal. Define $\psi:\mathbb S(\mathcal G)\to\mathbb C$ by $\psi(f)=(\varphi\circ\Phi)(f)$. For $f\in L^{\infty}(\mathcal{G}^{(0)})$, we have $\Phi(f)(\gamma,a)=f(\gamma)=f(r(\gamma))=(f\circ t)(r(\gamma,a))=(f\circ t)(\gamma,a)$, where $f\circ t\in L^{\infty}(A)$, and therefore, $\psi|_{L^{\infty}(\mathcal G^{(0)})}$ is normal. Finally, for $B\subset[[\GG]]$, denote by $B\times_{\GG^{(0)}} A$ the corresponding bisection in $[[\widetilde{\mathcal G}]]$. Then, for $(\gamma,a)\in\widetilde{\GG}$, 
$L_{B\times_{\GG^{(0)}} A}\Phi(f)(\gamma,a)=\Phi(f)(B^{-1}\gamma,a)=f(B^{-1}\gamma)=\Phi(L_{B}f)(\gamma,a)$. Since $\varphi$ is left $[[\widetilde{\mathcal{G}}]]$-invariant, the prior computation shows $\psi$ is $[[\mathcal{G}]]$-invariant. We conclude $\mathcal{G}$ is not properly proximal.

\vspace{2mm}

For the converse, suppose $\widetilde{\GG}=\GG\ltimes A$ is properly proximal. Define $\Psi \colon L^{\infty}(\widetilde{\GG},\tau\circ\lambda)\to L^{\infty}(\GG,\nu)$ by $\Psi(f)(\gamma)=\int_{A} f(\gamma,a)\,\mathrm{d}\tau_{s(\gamma)}(a)$ for $\gamma\in \Gamma$ and $f\in L^{\infty}(\widetilde{\GG})$. This map is well-defined, because we have for every non-negative Borel function $f$ on $\widetilde{\GG}$,
\begin{equation*}
    \begin{aligned}
        \int_{A}\sum_{\gamma:s(\gamma)=t(a)}f(\gamma,a)\,\mathrm{d}\tau(a) &= \int_{\GG^{(0)}}\bigg(\int_{\{a\in A: t(a)=x\}}\sum_{\gamma:s(\gamma)=t(a)}f(\gamma,a)\,\mathrm{d}\tau_x(a)\bigg)\,\mathrm{d}(t_*\tau)(x)\\
        &= \int_{\GG^{(0)}}\bigg(\sum_{\gamma:s(\gamma)=x}\int_{\{a\in A: t(a)=x\}}f(\gamma,a)\,\mathrm{d}\tau_x(a)\bigg)\,\mathrm{d}(t_*\tau)(x),
    \end{aligned}
\end{equation*}
and $t_*\tau \sim \mu$ on $\GG^{(0)}$ by assumption.

We claim that $\Psi(\mathbb S(\widetilde{\GG}))\subseteq\mathbb S(\GG)$. Fix $f\in\mathbb S(\widetilde{\GG})$, $\eps>0$, and a global bisection $E\in [\GG]$. Let $\widetilde{E} \coloneqq E \times_{\GG^{(0)}} A \subset \widetilde{\GG}$, which is a global bisection in $[\widetilde{\GG}]$ since $E\in [\GG]$. %
Let $\eps_0\in (0,\min\{1,\eps\})/3$ be a positive number such that $\mu(S)<\eps$ for every subset $S\subset \GG^{(0)}$ such that $(t_*\tau)(S)<\eps_0$. 
Let $\delta>0$ be such that $\delta\leq(\eps_0)^2/(10+10\left\Vert f\right\Vert_{\infty})^2$. Let $C,C'\subset A$ with $\tau(C),\tau(C')\leq\delta$ and let $\widetilde{F}\subset \widetilde{\GG}$ be a set of finite $(\tau\circ\lambda)$-measure for which $\left\Vert(1-\chi_{\widetilde{\GG}^C})(1-\chi_{\widetilde{\GG}_{C'}})(f-R_{\widetilde{E}}f)(1-\chi_{\widetilde{F}})\right\Vert_{\infty}\leq\delta$. 

Since we have
\begin{align*}
    (\eps_0)^2/(10+10 \lVert f\rVert_\infty)^2 > \delta > \tau(C') &= \int_A \chi_{C'}(a) \,\mathrm{d}\tau(a) =  \int_{\GG^{(0)}}\int_{t^{-1}(x)}\chi_{C'}(a)\,\mathrm{d}\tau_x(a)\,\mathrm{d}(t_*\tau)(x)\\
    &= \int_{\GG^{(0)}} \tau_x(C') \,\mathrm{d}(t_*\tau)(x),
\end{align*}
if we let $C_0' \coloneqq \{x\in \GG^{(0)}: \tau_x(C') \geq\eps/(10+10\| f\|_\infty)\}$, then by Chebyshev's inequality, $(t_*\tau)(C_0') \leq \eps_0/(10+10\| f\|_\infty)$. Similarly, the set $C_0\coloneqq \{x\in \GG^{(0)}: \tau_x(C) \geq\eps_0/(10+10\| f\|_\infty)\}$ satisfies $(t_*\tau)(C_0) \leq \eps_0/(10+10\| f\|_\infty)$.

By Lusin-Novikov's Theorem, $\widetilde{F}=\sqcup_{i=1}^{\infty}\widetilde{F}_i$ for some $\widetilde{F}_i \in[[\widetilde{\GG}]]$. Since $(\tau\circ \lambda)(\widetilde{F})<\infty$, there exists $N\in\N$ such that $(\tau\circ\lambda)(\sqcup_{i=N}^{\infty}\widetilde{F}_i) \leq \eps_0^2/(10+10 \lVert f\rVert_\infty)^2$, so that $\tau(\cup_{i=N}^{\infty}\widetilde{r}(\widetilde{F}_i)),\tau(\cup_{i=N}^{\infty}\widetilde{s}(\widetilde{F}_i)) \leq\eps_0^2/(10+10 \lVert f\rVert_\infty)^2$. By similar argument as above, if we let $C_1 \coloneqq\{x\in \GG^{(0)}: \tau_x(\cup_{i=N}^{\infty}\widetilde{r}(\widetilde{F}_i)) \geq \eps_0/(10+10\| f\|_\infty)\}$ and $C_1' \coloneqq \{x\in \GG^{(0)}: \tau_x(\cup_{i=N}^{\infty}\widetilde{s}(\widetilde{F}_i)) \geq \eps_0/(10+10\| f\|_\infty)\}$ then $(t_*\tau)(C_1),(t_*\tau)(C_1') \leq \eps_0/(10+10\| f\|_\infty)$. 

Now fix a basis $\BB = \{B_j\}_{j\in\mathbb{N}}$ for the groupoid $\GG$. For $1\leq i\leq N-1$ and $j\in \mathbb{N}$, define $F_{i,j}\coloneqq \{y\in \widetilde{s}(\widetilde{F}_i): (\gamma,y) \in \widetilde{F}_i \text{ for some }\gamma \in B_j\} \subset A$, so then $\widetilde{F}_i = \sqcup_{j=1}^\infty B_j \times_{\GG^{(0)}}F_{i,j}$ and $\widetilde{s}(\widetilde{F}_i) = \sqcup_{j=1}^{\infty} F_{i,j}$ since $\GG\acts A$ is free. Since $\tau(\widetilde{s}(\widetilde{F}_i)) = (\tau\circ\lambda)(\widetilde{F}_i)<\infty$ for each $i$, there exists $M_i\in\mathbb{N}$ such that $\tau(\sqcup_{j=M_i}^{\infty} F_{i,j}) < \eps_0^2/((10+10\| f\|_\infty)^2N)$. Again by similar argument as above, the sets $C_{2,i} \coloneqq\{x\in \GG^{(0)}: \tau_x(\cup_{j=M_i}^{\infty}F_{i,j}) \geq \eps_0/(10+10\| f\|_\infty)\}$ and $C_{2,i}' \coloneqq \{x\in \GG^{(0)}: \tau_x(\cup_{j=M_i}^{\infty}B_j\cdot F_{i,j}) \geq \eps_0/(10+10\| f\|_\infty)\}$ satisfy $(t_*\tau)(C_{2,i}),(t_*\tau)(C_{2,i}') \leq \eps_0/((10+10 \|f\|_\infty)N)$ for $1\leq i\leq N-1$. Now let $C_2\coloneqq \cup_{i=1}^{N-1}C_{2,i}$ and $C_2'\coloneqq \cup_{i=1}^{N-1}C_{2,i}'$, so then $(t_*\tau)(C_{2}),(t_*\tau)(C_{2}') \leq \eps_0/(10+10\| f\|_\infty)$.

Let $F\coloneqq \cup \{B_j: 1\leq j \leq \max\{M_1,\cdots,M_{N-1}\}\}$, which is a finite measure subset of $\GG$. Also define $D\coloneqq C_0\cup C_1\cup C_2$ and $D'\coloneqq C_0'\cup C_1'\cup C_2'$, so that $(t_*\tau)(D),(t_*\tau)(D') \leq \eps_0$ and thus $\mu(D),\mu(D') < \eps$ by our choice of $\eps_0$. Let $\gamma\in \GG \setminus F$ such that $s(\gamma)\not\in D'$ and $r(\gamma)\not\in D$, and $B_\gamma\in \BB$ be the bisection of $\GG$ such that $\gamma\in B_\gamma$. By the construction of $D'$, we have \[\tau_{s(\gamma)}(C'),\:\tau_{s(\gamma)}(\cup_{i=N}^{\infty}\widetilde{s}(\widetilde{F}_i)),\: \tau_{s(\gamma)}(\cup_{i=1}^{N-1}\cup_{j=M_i}^{\infty}F_{i,j})) < \eps_{0}/(10+10\| f\|_\infty),\]
and by the construction of $D$, we have 
\[\tau_{r(\gamma)}(C),\:\tau_{r(\gamma)}(\cup_{i=N}^{\infty}\widetilde{r}(\widetilde{F}_i)),\: \tau_{r(\gamma)}(\cup_{i=1}^{N-1}\cup_{j=M_i}^{\infty}B_j\cdot F_{i,j})) < \eps_{0}/(10+10\| f\|_\infty),\]
so by change of variable, $\tau_{s(\gamma)}(B_\gamma^{-1}\cdot C),\:\tau_{s(\gamma)}(\cup_{i=N}^{\infty}B_\gamma^{-1}\cdot \widetilde{r}(\widetilde{F}_i)),\: \tau_{s(\gamma)}(\cup_{i=1}^{N-1}\cup_{j=M_i}^{\infty}B_\gamma^{-1}\cdot B_j\cdot F_{i,j})) < \eps_{0}/(10+10\| f\|_\infty)$. Therefore, if we define $D_0\coloneqq C' \cup \cup_{i=N}^{\infty}\widetilde{s}(\widetilde{F}_i)\cup \cup_{i=1}^{N-1}\cup_{j=M_i}^{\infty}F_{i,j} \cup B_\gamma^{-1}\cdot C \cup \cup_{i=N}^{\infty}B_\gamma^{-1}\cdot \widetilde{r}(\widetilde{F}_i) \cup \cup_{i=1}^{N-1}\cup_{j=M_i}^{\infty}B_\gamma^{-1}\cdot B_j\cdot F_{i,j}$, then $\tau_{s(\gamma)}(D_0) < \eps_0/\left\Vert f\right\Vert_{\infty}$, and for any $a\in t^{-1}(s(\gamma))\setminus D_0$, we have $(\gamma,a), (\gamma\cdot E,E^{-1}\cdot a)\not\in \widetilde{\GG}^{C}\cup\widetilde{\GG}_{C'}\cup \widetilde{F}$ since $\GG\acts A$ is free, $\gamma\not\in F$, and $s(\gamma)\not\in D',r(\gamma)\not\in D$. Hence for every such $\gamma$ and $a\in t^{-1}(s(\gamma))\setminus D_0$, we have 
\[|(f-R_{\widetilde{E}}f)(\gamma,a)| = |f(\gamma,a) - f(\gamma\cdot E,E^{-1}a)|\leq \delta < \eps/3.\]
Therefore, for $\gamma \in \GG\setminus F$ with $s(\gamma)\not\in D',r(\gamma)\not\in D$,
\begin{equation*}
    \begin{aligned}
    |\Psi(f)(\gamma)-R_E(\Psi(f))(\gamma)|&=\left|\int_A(f(\gamma,a)\,\mathrm{d}\tau_{s(\gamma)}(a)-\int_A f(\gamma\cdot E,a))\,\mathrm{d}\tau_{s(\gamma\cdot E)}(a)\right|\\
    &= \left|\int_{t^{-1}(s(\gamma))}(f(\gamma,a)\,\mathrm{d}\tau_{s(\gamma)}(a)-\int_{t^{-1}(s(\gamma))} f(\gamma\cdot E,E^{-1}\cdot a))\,\mathrm{d}\tau_{s(\gamma)}(a)\right|\\
    &\leq 2\left\Vert f\right\Vert_{\infty}\tau_{s(\gamma)}(D_0)+ \int_{t^{-1}(s(\gamma))\setminus D_0}|f(\gamma,a)-f(\gamma \cdot E,E^{-1}\cdot a)|\,\mathrm{d}\tau_{s(\gamma)}(a)\\
    &< 2\left\Vert f\right\Vert_{\infty} \cdot (\eps_0/\left\Vert f\right\Vert_{\infty}) + \eps/3\\
    &< \eps,
\end{aligned}
\end{equation*}
which implies that $\left\Vert(1-\chi_{\GG^D})(1-\chi_{\GG_{D'}})(\Psi(f)-R_{E}\Psi(f))(1-\chi_{F})\right\Vert_{\infty}\leq\eps$. Thus $\Psi(f)\in \mathbb S(\GG)$ by definition, and this proves our claim.

We next check that $\Psi(L^\infty(A,\tau))\subset L^{\infty}(\GG^{(0)},\mu)$. Since $t_*\tau \sim \mu$, it suffices to show $\Psi(f) \in L^\infty(\GG^{(0)},t_*\tau)$ for every $f\in L^\infty(A,\tau)$. Recall that the inclusion $ L^\infty(A,\tau) \subset L^\infty(\widetilde{\GG},\tau\circ\lambda)$ is realized by identifying each function $f\in L^\infty(A,\tau)$ with $\widetilde{f}\in L^\infty(\widetilde{\GG},\tau\circ\lambda)$ given by $\widetilde{f}(\gamma,a) = f(\gamma\cdot a)$. For any composable pair of elements $(\gamma,a),(\eta,b)\in \widetilde{\GG}$, $r(\gamma,a) = r(\eta,b)$ implies $r(\gamma) = t(\gamma\cdot a) = t(r(\gamma,a)) = t(r(\eta,b)) = t(\eta\cdot b) = r(\eta)$. Hence for $\gamma,\eta\in \GG$ with $r(\gamma) = r(\eta)$, if we let $B_\gamma,B_\eta\in \BB$ be the bisections of $\GG$ such that $\gamma\in B_\gamma, \eta\in B_\eta$, respectively, then by the change of variable formula,
\begin{equation*}
    \begin{aligned}
        \Psi(\widetilde{f})(\gamma) &= \int_{A} \widetilde{f}(\gamma,a)\,\mathrm{d}\tau_{s(\gamma)}(a) = \int_{A} \widetilde{f}(B_{\eta}\cdot B_{\eta}^{-1}\gamma,a)\,\mathrm{d}\tau_{s(\gamma)}(a) = \int_{A} \widetilde{f}(\eta\cdot B_{\eta}^{-1}B_\gamma,a)\,\mathrm{d}\tau_{s(\gamma)}(a)\\
        &= \int_{A} \widetilde{f}(\eta, B_{\gamma}^{-1}B_\eta a)\,\mathrm{d}\tau_{s(\gamma)}(B_{\gamma}^{-1}B_\eta a) = \int_A \widetilde{f}(\eta,a)\,\mathrm{d}\tau_{B_\eta^{-1}B_\gamma\cdot s(\gamma)}(a) = \int_{A}\widetilde{f}(\eta,a)\,\mathrm{d}\tau_{s(\eta)}(a) = \Psi(\widetilde{f})(\eta).
    \end{aligned}
\end{equation*}
Hence $ \Psi(\widetilde{f})(\gamma) =  \Psi(\widetilde{f})(\eta)$ whenever $r(\gamma) = r(\eta)$, so $\Psi(\widetilde{f}) \in L^{\infty}(\GG^{(0)},t_*\tau)$, which proves our claim. Moreover, $\Psi$ is normal on $L^\infty(A,\tau)$. Indeed, for any $h\in L^{1}(\GG^{(0)},t_*\tau)$, we define $h_0\in L^{1}(A,\tau)$ by $h_0(a) = h(t(a))$ for $a\in A$, so then
\[\int_A h_0(a)\,\mathrm{d}\tau(a) = \int_A h(t(a))\,\mathrm{d}\tau(a) = \int_{\GG^{(0)}}\int_{t^{-1}(x)}h(t(a))\,\mathrm{d}\tau_x(a)\,\mathrm{d}(t_*\tau)(x) = \int_{\GG^{(0)}}h(x)\,\mathrm{d}(t_*\tau)(x).\]
Then for any weakly convergent sequence $f_n\to f$ in $L^{\infty}(A,\tau)$, we have 
\begin{equation*}
    \begin{aligned}
        \lim_{n\to\infty}\int_{\GG^{(0)}}\Psi(f_n - f)(x)h(x)\,\mathrm{d}(t_*\tau)(x) &= \lim_{n\to\infty}\int_{\GG^{(0)}}\bigg(\int_A \widetilde{f_n}(x,a)-\widetilde{f}(x,a)\,\mathrm{d}\tau_x(a)\bigg) h(x)\,\mathrm{d}(t_*\tau)(x) \\
        &= \lim_{n\to\infty}\int_{\GG^{(0)}}\int_A (\widetilde{f_n}(x,a)-\widetilde{f}(x,a)\,\mathrm{d}\tau_x(a)) \widetilde{h_0}(x,a)\,\mathrm{d}(t_*\tau)(x)\\
        &= \lim_{n\to\infty}\int_A  (f_n(a) - f(a))h_0(a)\,\mathrm{d}\tau(a)\\
        &=0,
    \end{aligned}
\end{equation*}
so $\Psi(f_n)\to \Psi(f)$ weakly in $L^{\infty}(\GG^{(0)},t_*\tau)$, and hence in $L^{\infty}(\GG^{(0)},\mu)$ since $t_*\tau \sim \mu$. Thus $\Psi$ is normal on $L^\infty(A,\tau)$.

Now suppose towards contradiction that $\GG$ is not properly proximal, and let $\varphi\colon \S(\GG)\to\C$ be a left $[[\GG]]$-invariant state that is normal on $L^\infty(\GG^{(0)},\mu)$. Consider $\psi\coloneqq \varphi\circ \Psi\colon \S(\widetilde{\GG})\to\C$. Since the restriction $\restr{\Psi}{L^\infty(A,\tau)}\colon L^\infty(A,\tau)\to L^{\infty}(\GG^{(0)},\mu)$ is normal, it follows that $\psi$ is normal on $L^\infty(A,\tau)$. To show $\psi$ is left $[\widetilde{\GG}]$-invariant, let $E\in[\widetilde{\GG}]$ be a global section, so then there exist pairwise disjoint subsets $E_i\subset A$ such that $A = \sqcup_{i=1}^{\infty}E_i = \sqcup_{i=1}^{\infty}B_i\cdot E_i$ and $E =\sqcup_{i=1}^{\infty} B_i \times_{\GG^{(0)}} E_i$ since $\GG\acts A$ is free. It follows that for $f\in\mathbb S(\widetilde{\GG})$ and $\gamma\in \GG$ we have
\begin{equation*}
    \begin{aligned}
        \Psi(L_{E}f)(\gamma)&=\int_A(L_{E}f)(\gamma,a)\,\mathrm{d}\tau_{s(\gamma)}(a) = \sum_{i=1}^{\infty} \int_{A\cap \gamma^{-1}\cdot B_i\cdot E_i} (L_E f)(\gamma,a)\,\mathrm{d}\tau_{s(\gamma)}(a)\\
    &=\sum_{i=1}^{\infty} \int_{A\cap \gamma^{-1}\cdot B_i\cdot E_i} f(B_{i}^{-1}\cdot \gamma,a)\,\mathrm{d}\tau_{s(\gamma)}(a) \\
    &= \sum_{i=1}^{\infty} \int_{A} f(B_{i}^{-1}\cdot \gamma,a) \cdot \chi_{(\widetilde{\GG})_{\gamma^{-1}\cdot B_i\cdot E_i}}(B_i^{-1}\cdot \gamma, a)\,\mathrm{d}\tau_{s(\gamma)}(a)\\
    &= \sum_{i=1}^{\infty} \int_{A} f(B_{i}^{-1}\cdot \gamma,a) \cdot \chi_{(\widetilde{\GG})^{E_i}}(B_i^{-1}\cdot \gamma, a)\,\mathrm{d}\tau_{s(\gamma)}(a)\\
    &=\sum_{i=1}^{\infty} \int_{A} (f\cdot \chi_{(\widetilde{\GG})^{E_i}})(B_{i}^{-1}\cdot \gamma,a)\,\mathrm{d}\tau_{s(\gamma)}(a)\\
    &=\sum_{i=1}^{\infty}L_{B_i}\Psi(f\cdot \chi_{(\widetilde{\GG})^{E_i}})(\gamma).
    \end{aligned}
\end{equation*}
Since $\mathcal{G}\curvearrowright A$ is measure preserving, the above equation in particular implies for $n\geq 1$,
\begin{equation} \label{eqn: compare finite sum}
    \begin{aligned}
        \abs{\Psi(L_E f)(\gamma) - \sum_{i=1}^n L_{B_i}\Psi(f\chi_{(\widetilde{\GG})^{E_i}})(\gamma)} &= \abs{\sum_{i=n+1}^{\infty} \int_{A\cap \gamma^{-1}\cdot B_i\cdot E_i} f(B_{i}^{-1}\cdot \gamma,a)\,\mathrm{d}\tau_{s(\gamma)}(a)} \\
        &\leq \|f\|_\infty \cdot \tau_{s(\gamma)}\bigl(\bigsqcup_{i>n}\gamma^{-1}\cdot B_i\cdot E_i\bigr)\\
        &\leq \|f\|_\infty \cdot \tau_{r(\gamma)}\bigl(\bigsqcup_{i>n}\cdot B_i\cdot E_i\bigr)
    \end{aligned}
\end{equation}
For each $n$, define functions $u_n,u_n'\in L^\infty \GG^{(0)}$ by $u_n(x) = \tau_x(\bigsqcup_{i>n}E_i)$ and $u_n'(x) = \tau_x(\bigsqcup_{i>n}B_i\cdot E_i)$. Since both $\{E_i\}_i$ and $\{B_i\cdot E_i\}_i$ partition $A$, it follows that $u_n\searrow 0$ and $u_n'\searrow 0$ almost everywhere. Moreover, by \eqref{eqn: compare finite sum},
\[\abs{\Psi(f)(\gamma) - \sum_{i=1}^n \Psi(f\chi_{(\widetilde{\GG})^{E_i}})(\gamma)} \leq \|f\|_{\infty}u_n(r(\gamma)), \;\; \abs{\Psi(L_E f)(\gamma) - \sum_{i=1}^n L_{B_i}\Psi(f\chi_{(\widetilde{\GG})^{E_i}})(\gamma)} \leq \|f\|_{\infty}u_n'(r(\gamma)).\]
Since $u_n,u_n' \in L^\infty\GG^{(0)}$ decrease to $0$ almost everywhere and $\varphi$ is normal on $L^\infty\GG^{(0)}$, we have $\varphi(u_n),\varphi(u_n')\to 0$. Therefore, applying the left $[[\GG]]$-invariant state $\varphi$ to above and passing to limit, we have
\[\psi(L_Ef) = \varphi(\Psi(L_E f)) = \lim_{n}\sum_{i=1}^n \varphi\bigl(L_{B_i}\Psi(f\chi_{(\widetilde{\GG})^{E_i}})\bigr) = \lim_n\sum_{i=1}^n \varphi\bigl(\Psi(f\chi_{(\widetilde{\GG})^{E_i}})\bigr)= \varphi(\Psi(f))=\psi(f),\]
so $\psi$ is a $[\widetilde{\GG}]$-invariant state on $\S(\widetilde{\GG})$. This contradicts the assumption that $\widetilde{\GG}$ is properly proximal, so $\GG$ must be properly proximal.
\end{proof}

\subsection{Measure equivalence of groupoids}

We conclude this section by proving that proper proximality is preserved under measure equivalence. We refer to Appendix \ref{sec: appendix} for the definition of measure class preserving measure equivalence, and for a more in-depth treatment of this notion.

\begin{thm}\label{thm:MEpp}
    Let $(\GG,\mu\circ\lambda), (\HH,\nu\circ\lambda)$ be measure class preserving measure equivalent countable p.m.p.\ groupoids. If $(\GG,\mu\circ\lambda)$ is properly proximal, then so is $(\HH,\nu\circ\lambda)$.
\end{thm}
\begin{proof}
    Suppose $(\GG,\mu\circ\lambda)$ is properly proximal. Let $(\Omega, m)$ be an ergodic measure class preserving measure equivalence coupling for $(\GG,\mu\circ\lambda), (\HH,\nu\circ\lambda)$, and $t_\GG,t_\HH$ denote the anchor maps on $\Omega$ of the $\GG$-action and the $\HH$-action, respectively. Let $X,Y\subset \Omega$ be the strict fundamental domains of the $\GG$-action and the $\HH$-action, and let $m_X\coloneqq \frac{\restr{m}{X}}{m(X)}, m_Y\coloneqq \frac{\restr{m}{Y}}{m(Y)}$ be the probability measures on $X,Y$ by normalizing the restriction of $m$ on $X,Y$, so that $(t_\GG)_* m_Y \sim \mu$ on $\GG^{(0)}$ and $(t_\HH)_* m_X \sim \nu$ on $\HH^{(0)}$ by \Cref{lem: inv null set}. Also, as in the proof of \Cref{prop: me equiv soe}, there exist complete unit sections $Y_0\subset Y=(\GG\ltimes Y)^{(0)}$ and $X_0\subset X=(\HH\ltimes X)^{(0)}$ such that $((\GG\ltimes Y)^{Y_0}_{Y_0}, \restr{m_Y}{Y_0}) \simeq ((\HH\ltimes X)_{X_0}^{X_0}, \restr{m_X}{X_0})$.

    As noted in \cite[Lemma 4.7, Lemma 4.8]{bddcohomologytransversegroupoid}, $(\GG\ltimes Y,m_Y)$ and $(\HH\ltimes X,m_X)$ are ergodic principal countable p.m.p.\ groupoids. Since $(\GG,\mu\circ\lambda)$ is properly proximal and $(t_\GG)_* m_Y \sim \mu$, by \Cref{prop: transf preserv prop prox} $(\GG\ltimes Y,m_Y)$ is also properly proximal. By Corollary \ref{cor: f.i. equiv rel restr}, upon shrinking $X_0,Y_0$ to smaller subsets we know $((\HH\ltimes X)_{X_0}^{X_0}, \restr{m_X}{X_0})\simeq ((\GG\ltimes Y)_{Y_0}^{Y_0}, \restr{m_Y}{Y_0})$ are also properly proximal. By \Cref{restriction}(ii) we see that $(\HH\ltimes X,m_X)$ is properly proximal, and finally, applying \Cref{prop: transf preserv prop prox} again we conclude that $(\HH,\nu\circ\lambda)$ is properly proximal.
\end{proof}

\section{Examples of properly proximal groupoids} \label{sec:examples} 

\noindent This section provides examples and non-examples of properly proximal groupoids. Regarding examples, we first consider transverse measured groupoids and nondegenerate free products. We then discuss two classes of properly proximal equivalence relations arising from cost theory and orbit equivalence rigidity. The second class provides properly proximal equivalence relations which cannot be generated, modulo null sets, by an essentially free p.m.p.\ action of any countable group. As for non-examples, we show that, as in the group case, inner amenable groupoids are never properly proximal.


\subsection{Transverse measured groupoids} These groupoids were introduced by Hartnick and Sarti in \cite{bddcohomologytransversegroupoid} in order to find a suitable replacement for lattices in the settings of discrete measured groupoids.

Let $G$ be a unimodular lcsc group and $m_G$ its Haar measure. Consider a p.m.p.\ Borel action of $G$ on a standard Borel probability space $(X,\mu)$. We then obtain a Borel action groupoid $G\ltimes X$ which admits an invariant measure $\mu$ for the Haar system defined by $m_G$. For background on Borel groupoids see \cite[Section 2.2]{bddcohomologytransversegroupoid}. A Borel subset $Y\subset X$ is a \textit{cross section} if $GY=X$ and for every $x\in GY$ the \textit{hitting time set} $Y_x=\{g\in G:gx\in Y\}$ is non-empty and locally finite. In this case, the restriction Borel groupoid $\mathcal G\coloneqq (G\ltimes X)|_Y$ has countable fibers, and the invariant probability measure $\mu$ on $X$ induces an invariant $\sigma$-finite Borel measure $\nu$ for $\mathcal G$, known as the \textit{transverse measure} of $(X,\mu,Y)$ (see \cite{ARBC25}). The finiteness of $\nu$ is equivalent to the local integrability of $x\mapsto Y_x$, i.e., $\int_X|Y_x\cap L|\,\mathrm{d}\mu(x)<\infty$ for every compact subset $L\subset G$ (see \cite[Proposition 5.1]{ARBC25}). In this case, $(X,\mu,Y)$ is said to be an \textit{integrable transverse $G$-system} and the discrete measured groupoid $(\GG,\nu)$ is the \textit{transverse measured groupoid} of $(X,\mu,Y)$.

\begin{exmp} If $\Gamma<G$ is a lattice in a lcsc group $G$, then $G$ acts on $X=G/\Gamma$ fixing the unique probability measure $\mu$ on $X$. If $Y=\{\Gamma\}$, it holds that $(X,\mu,Y)$ is an integrable transverse system over $G$, and the associated transverse groupoid is isomorphic to the Borel groupoid defined by $\Gamma$ with transverse measure given by the Dirac measure of total mass $\text{covol}(\Gamma)^{-1}$. 
\end{exmp}

For $G$ a countable discrete group with counting measure and $(X,\mu,Y)$ an ergodic integrable transverse $G$-system, form the transverse measured groupoid $(\mathcal{G},\nu)$. In \cite[Proposition 4.11]{bddcohomologytransversegroupoid} it is shown that the groupoid $G$ and $\mathcal G$ are (measure class preserving) measure equivalent in the sense of Appendix \ref{sec: appendix}. Therefore, as a direct consequence of \Cref{thm:MEpp}, we obtain the following statement.

\begin{corollary}\label{cor:transverse} Let $G$ be a countable discrete group with counting measure $m_G$, and let $(X,\mu,Y)$ be an ergodic integrable transverse $G$-system with transverse measured groupoid $(\GG,\nu)$. Then the group $G$ is properly proximal if and only if $\mathcal G$ is properly proximal.
\end{corollary}

\subsection{Free products} In this section, we show that free products of groupoids provide examples of properly proximal groupoids. We begin by recalling the definition of the free product of groupoids. 

Let $\{(\mathcal{H}_i,\nu_i)\}_{i\in I}$ be a collection of countable p.m.p.\ groupoids over the same unit space $(X,\mu)$. The free product $\mathcal{G}=\ast_{X,i\in I}\mathcal{H}_i$ is generated by $X$ and by the $\mathcal{H}_i$ in such a way that the groupoids $\mathcal{H}_i$ are freely independent: there does not exist a finite sequence of indices $i_1,...,i_n\in I$ with $n\geq 1$ and $i_j\neq i_{j+1}$, along with non-unit elements $h_j\in \mathcal{H}_{i_j}\setminus X$ such that the composable product $h_1\cdots h_n\in X$. We refer to the composable product $h_1\cdots h_n$, where $h_j\in \mathcal{H}_{i_j}\setminus X$ and $i_j\neq i_{j+1}$ for all $j$, as a reduced word. The product of two reduced words $\gamma=\gamma_n\cdots\gamma_1$ and $\rho=\rho_m\cdots\rho_1$ with $r(\rho)=s(\gamma)$ (i.e. $r(\rho_m)=s(\gamma_1)$) is given by the reduction of the concatenation $\gamma\rho=\gamma_n\cdots\gamma_1\rho_m\cdots\rho_1$. Note that $s(\gamma\rho)=s(\rho_1)$ and $r(\gamma\rho)=r(\gamma_n)$. 

We identify $\mathcal{G}$ with the collection of the unit space $(X,\mu)$ together with all reduced words, endowed with the operation of concatenation followed by reduction. We call the resulting $\mathcal G$ the \textit{free product of the groupoids} $(\mathcal{H}_i,\nu_i)$, which is once again a countable p.m.p.\ groupoid over $(X,\mu)$ with measure denoted by $\nu$. 

Consider a free product $\mathcal{G}=\mathcal{G}_1\ast_X\mathcal{G}_2$. Denote by $G_n$ the set of reduced words in $\mathcal G$ with $n$ letters, which by construction it is measurable. Denoting by $\ell(\gamma)$ the length of a reduced word $\gamma\in\mathcal G$, we see that $\{\gamma\in\GG:\ell(\gamma)=n\}=G_n$ is a measurable set for every $n\in\mathbb N$. By convention, $\{\gamma\in\mathcal{G}:l(\gamma)=0\}=X$ and $\{\gamma\in\mathcal{G}:l(\gamma)=1\}=(\mathcal{G}_1\cup\mathcal{G}_2)\setminus X$.  

The reader can find more about free products of groupoids and equivalence relations in the works of \cite{Gab00,Car11,AlvGab,TDW24} and \cite[Section 5]{Kosaki}.

\begin{prop}\label{prop:freeproducts} Let $(\mathcal G_1,\nu_1)$ and $(\mathcal G_2,\nu_2)$ be countable p.m.p.\ groupoids over the same unit space $(X,\mu)$. Assume that $|\mathcal{G}_1\cap r^{-1}(x)|\geq 3$ and $|\mathcal{G}_2\cap r^{-1}(x)|\geq 2$, for almost every $x\in X$. Then $\mathcal G=\mathcal G_1\ast_X\mathcal G_2$ is properly proximal.
\end{prop}

\begin{proof} Let $\gamma\in\mathcal{G}$ be given by $\gamma=\gamma_n\cdots\gamma_1$ where $\gamma_i\in\mathcal{G}_{j_i}\setminus X$ with $j_i\neq j_{i+1}$. Denote by $p_1\in L^{\infty}(\mathcal{G})$ (respectively, $p_2\in L^{\infty}(\mathcal{G})$) the characteristic function on the set of reduced words $\gamma_n\cdots\gamma_1=\gamma\in\mathcal{G}$ whose last letter $\gamma_n$ is in $\mathcal{G}_1\setminus X$ (respectively, $\mathcal{G}_2\setminus X$). We first show $p_1,p_2\in\mathbb S(\mathcal{G})$. 

Fix $i\in\{1,2\}$ and $\theta\in[[\mathcal{G}]]$. Denote by $F_1=\{\gamma\in\mathcal{G}_{r(\theta)}:p_i(\gamma)=1,\:p_i(\gamma\theta)=0\}$, $F_2=\{\gamma\in\mathcal{G}_{r(\theta)}:p_i(\gamma)=0,\:p_i(\gamma\theta)=1\}$ and $F=F_1\cup F_2$. We claim that $F$ is a $c_0$-set. By symmetry and since the union of $c_0$-sets is a $c_0$-set, it suffices to show $F_1$ is a $c_0$-set. For $x\in r(\theta)$, denote by $\eta_x$ the unique element in $\theta$ such that $r(\eta_x)=x$. Suppose that $\gamma\in F_1$ has reduced form $\gamma=\gamma_n\ldots\gamma_1$ and $s(\gamma)=x$. Since $p_i(\gamma)=1$, we have that $\gamma_n\in\mathcal G_i\setminus X$, and from $p_i(\gamma\theta)=0$, we must have that the letter $\gamma_n$ disappears after cancellations from the letters in the reduced form of $\eta_x$. This implies that $\gamma^{-1}$ is an initial segment of the reduced word representing $\eta_x$. In particular, $\ell(\gamma)\leq \ell(\eta_x)$, and for a fixed $x\in r(\theta)$, there are at most $\ell(\eta_x)$ possible elements $\gamma\in F_1$ with $s(\gamma)=x$. This means $|F_1\cap s^{-1}(x)|\leq \ell(\eta_x)$. 

Fix $\epsilon>0$. For $N\geq 1$, define $B_N:=\{x\in r(\theta):\ell(\eta_x)> N\}$. This set is measurable since the length function $\ell$ and the assignment $x\mapsto \eta_x=(r_{\theta}^{-1})(x)$ are measurable. Since $\ell(\eta_x)<\infty$ for almost every $x\in r(\theta)$, there exists $N\geq 1$ large enough for which $\mu(B_N)\leq \epsilon$. Take $F_N:=F_1\cap s^{-1}(r(\theta)\setminus B_N)$, and observe that
\begin{equation*}
    \nu(F_N)=\int_{\mathcal G^{(0)}}|F_N\cap s^{-1}(x)|d\mu(x)=\int_{r(\theta)\setminus B_N}|F_1\cap s^{-1}(x)|d\mu(x)\leq (N+1)\mu(r(\theta))<\infty.
\end{equation*}
Since $F_1\subset F_N\cup s^{-1}(B_N)$ and $\mu(B_N)\leq\epsilon$, it follows that $F_1$ is a $c_0$-set. The same proof gives that $F_2$, and thus $F$, is a $c_0$-set. For $\gamma\not\in F$, $(p_i\chi_{\mathcal G_{r(\theta)}}-R_{\theta}p_i)(\gamma)=0$ which means $p_i\chi_{\mathcal G_{r(\theta)}}-R_{\theta}p_i$ is supported on a $c_0$-set. We have therefore shown that $p_i\in\mathbb S(\mathcal G)$.

Now, since $|\mathcal{G}_1\cap r^{-1}(x)|\geq 3$ and $|\mathcal{G}_2\cap r^{-1}(x)|\geq 2$ for almost every $x\in X$, we can find global bisections $\theta_1,\theta_2\in[\mathcal{G}_1]$ and $\theta_3\in[\mathcal{G}_2]$ satisfying that $\theta_1\cap\theta_2=\emptyset$ and $\theta_i\cap X=\emptyset$, for $i=1,2,3$. It then holds that $L_{\theta_1}p_2+L_{\theta_2}p_2\leq p_1$, $L_{\theta_3}p_1\leq p_2$ and $p_1+p_2=1-\chi_X$. 

Suppose towards contradiction that $\GG$ is not properly proximal and $\varphi\colon\mathbb S(\mathcal{G})\to\mathbb C$ is a left $[[\GG]]$-invariant map. Then $2\varphi(p_2)=\varphi(L_{\theta_1}p_2+L_{\theta_2}p_2)\leq \varphi(p_1)=\varphi(L_{\theta_3}p_1)\leq \varphi(p_2)$, so $\varphi(p_1)=\varphi(p_2)=0$, and thus $\varphi(\chi_X) = \varphi(1-p_1-p_2)= 1$. Since $\GG$ is aperiodic, by Lusin-Novikov Theorem, there exist infinitely many partial Borel isomorphisms $\rho_n \in [[\GG]]$ such that $s(\rho_n) = X$ and the corresponding Borel sections $B_n\subset \GG$ are pairwise disjoint. By left $[[\GG]]$-invariance of $\varphi$ and for $n\neq m$, we have 
\[1=\varphi(1) \geq \varphi(L_{\rho_n}\chi_X + L_{\rho_m}\chi_X) = \varphi(L_{\rho_n}\chi_X) + \varphi(L_{\rho_m}\chi_X) = 2\varphi(\chi_X) = 2,\]
a contradiction. Therefore $\GG$ is properly proximal.
\end{proof}

\subsection{Treeable equivalence relations and irrational compressions}

We next record two classes of properly proximal p.m.p.\ equivalence relations. The first comes from Gaboriau's theory of cost and Hjorth's realization theorem for treeable equivalence relations. The second comes from Furman's orbit equivalence rigidity theorem for the natural action of higher-rank lattices on tori.

Recall that a graphing of a p.m.p.\ equivalence relation $\RR$ on $(X,\mu)$ is a countable family $\Phi=\{\varphi_i\colon A_i \to B_i\}_{i\in I} \subset [[\RR]]$ which generates $\RR$. Its cost is defined by
$\operatorname{cost}(\Phi)=\sum_{i\in I}\mu(A_i)$, and the cost of $\RR$ is $\operatorname{cost}(\RR) \coloneqq \inf\{\operatorname{cost}(\Phi): \Phi\text{ is a graphing of }\RR\}$. A graphing is called a \emph{treeing} if the graph induced by $\Phi$ on almost every $\RR$-class is a tree, and $\RR$ is called treeable if it admits a treeing. Gaboriau proved that every treeing realizes the cost of the equivalence relation it generates \cite{Gab00}. Hjorth subsequently proved that an ergodic treeable p.m.p.\ equivalence relation of cost $n\in \N\cup\{\infty\}$ is generated by an essentially free p.m.p.\ action of $\mathbb F_n$ \cite[Lemmas 4.1, 4.2]{MR2258624}.

\begin{prop} \label{prop: nonamen treeable equiv rel}
Let $\RR$ be an ergodic aperiodic treeable p.m.p.\ equivalence relation. If $\operatorname{cost}(\RR)>1$, then $\RR$ is properly proximal. In particular, every nonamenable ergodic treeable p.m.p.\ equivalence relation is properly proximal.
\end{prop}

\begin{proof}
    Suppose first that $1<\operatorname{cost}(\RR)<\infty$. Choose an integer $n>\operatorname{cost}(\RR)$. Since $\RR$ is aperiodic and ergodic, $(X,\mu)$ is nonatomic, and hence there exists a measurable subset $A\subset X$ such that
    \[\mu(A)=\frac{\operatorname{cost}(\RR)-1}{n-1}.\]
    The set $A$ is a complete unit section because $\RR$ is ergodic. Consider the restriction $(\RR^A_A,\frac{1}{\mu(A)}\mu|_A)$. Gaboriau's compression formula gives $\operatorname{cost}(\RR)-1  =\mu(A)\bigl(\operatorname{cost}(\RR^A_A)-1\bigr)$, and therefore
    \[\operatorname{cost}(\RR^A_A)-1=\frac{\operatorname{cost}(\RR)-1}{\mu(A)}=n-1.\]
    Thus, $\operatorname{cost}(\RR^A_A)=n$. The restriction $\RR^A_A$ remains ergodic and treeable, so Hjorth's theorem provides an essentially free p.m.p.\ action $\F_n\curvearrowright(A,\frac{1}{\mu(A)}\mu|_A)$ whose orbit equivalence relation is $\RR^A_A$. Since the group $\F_n$ is properly proximal, it follows from \Cref{prop: transf preserv prop prox} that $\RR^A_A$ is properly proximal, and \Cref{restriction}(ii) then implies that $\RR$ is properly proximal.

    If $\operatorname{cost}(\RR)=\infty$, Hjorth's theorem directly realizes $\RR$ as the orbit equivalence relation of an essentially free p.m.p.\ action of $\F_\infty$, which is a properly proximal group, and the conclusion again follows from \Cref{prop: transf preserv prop prox}.

    Finally, if $\operatorname{cost}(\RR)=1$, Hjorth's theorem realizes $\RR$ as the orbit equivalence relation of an essentially free action of $\Z$, and hence $\RR$ is amenable. Consequently, every nonamenable ergodic treeable relation has cost strictly greater than one, proving the last assertion.
\end{proof}

Together with the fact, proved in the next subsection, that amenable p.m.p.\ groupoids are not properly proximal, the preceding proposition shows that, among ergodic aperiodic treeable p.m.p.\ equivalence relations, proper proximality is equivalent to nonamenability.

The next example is of a different nature. It shows that a properly proximal equivalence relation need not admit a presentation as the orbit equivalence relation of an essentially free p.m.p.\ action. The construction originates in Furman's negative answer to the measured version of a question of Feldman and Moore.

\begin{exmp}[Irrational compression of a higher-rank lattice relation] \label{exmp: restr SLn action}
    Let $d\geq 3$, let $\Gamma=\operatorname{SL}_d(\Z)$, and consider the natural action
    \[ \Gamma\curvearrowright (\T^d,m_{\T^d}), \qquad \T^d=\R^d/\Z^d.\]
    Let $\RR=\RR(\Gamma\curvearrowright\T^d)$ be its orbit equivalence relation. If $A\subset\T^d$ is a measurable subset such that $0<m_{\T^d}(A)<1$ and $ m_{\T^d}(A)\notin\Q$, then the restricted relation $(\RR^A_A, \frac{1}{m_{\T^d}(A)}m_{\T^d}|_A)$ is properly proximal. However, $\RR^A_A$ cannot be generated, modulo null sets, by an essentially free p.m.p.\ action of any countable group.

    Indeed, $\operatorname{SL}_d(\Z)$ is a lattice in the noncompact semisimple Lie group $\operatorname{SL}_d(\R)$, and is therefore properly proximal by \cite[Proposition 1.6]{BIP21}. The natural action on $\T^d$ is ergodic and essentially free. Ergodicity follows, for example, from the fact that every nonzero orbit of the dual action on $\Z^d$ is infinite. To see essential freeness, fix $g\in\Gamma\setminus\{e\}$. Its fixed-point set is $\operatorname{Fix}(g) =\ker\bigl(\overline{g-I}\colon \T^d\to \T^d\bigr)$, where $\overline{g-I}$ is the torus homomorphism induced by the nonzero integer matrix $g-I$. Hence, $\operatorname{Fix}(g)$ is a proper closed subgroup of $\T^d$, and consequently has Haar measure zero. Since $\Gamma$ is countable, the action is essentially free.

    \Cref{prop: transf preserv prop prox} now shows that $\RR$ is properly proximal. Since $\RR$ is ergodic and $A$ has positive measure, $A$ is a complete unit section. By \Cref{lem: red is me}, $\RR$ and $\RR^A_A$ are measure class preserving measure equivalent, so \Cref{thm:MEpp} implies that $\RR^A_A$ is properly proximal.

    The final assertion is Furman's orbit equivalence rigidity theorem \cite[Theorem D(1)]{MR1740985}. We briefly recall the idea. Suppose that $\RR^A_A$ were generated by an essentially free p.m.p.\ action of a countable group $\Lambda$. This would give a weak orbit equivalence between the $\Lambda$-action and the original $\Gamma$-action with compression constant $1/m_{\T^d}(A)$. The strong orbit equivalence rigidity of the toral $\operatorname{SL}_d(\Z)$-action forces this weak orbit equivalence to arise, up to finite groups and finite-index subgroups, from a virtual isomorphism of the actions. The compression constant of such a virtual isomorphism is necessarily rational, whereas $1/m_{\T^d}(A)$ is irrational. This contradiction shows that no such essentially free action can generate $\RR^A_A$.
\end{exmp}

Notice that the irrationality assumption is not used to establish proper proximality: every positive-measure restriction of $\RR$ is properly proximal. Irrationality is used only in Furman's rigidity argument to exclude an essentially free group-action presentation. Thus the example is not freely generated, although it remains stably orbit equivalent to the original $\operatorname{SL}_d(\Z)$-action.

\subsection{Inner amenable groupoids} In \cite{Kida-TuckerDrob}, a notion of inner amenability for countable p.m.p.\ groupoids was introduced. Here, we establish that, analogous to the group case \cite[Proposition 4.11]{BIP21}, inner amenable groupoids provide canonical examples of groupoids that fail to be properly proximal.

Before defining inner amenability, we introduce some notation. For an element $\theta\in[[\mathcal G]]$ and $\gamma\in \GG^{r(\theta)}_{r(\theta)}$, we set $\gamma^{\theta}=\theta^{-1}\gamma\theta=[r_{\theta}^{-1}(r(\gamma))]^{-1}\gamma[r_{\theta}^{-1}(s(\gamma))]$. For a measurable subset $D\subset\GG$, we define $D^{\theta}=\{\gamma^{\theta}:\gamma\in D\cap\GG_{r(\theta)}^{r(\theta)}\}$ and $D^{-1}=\{\gamma^{-1}\in\mathcal G:\gamma\in D\}$. Moreover, for a function $f:\GG\to\mathbb C$, we define $f^{\theta}\colon\GG\to\mathbb C$ by
\begin{equation*}
    f^{\theta}(\gamma)=\begin{cases}
        f([s_{\theta}^{-1}(r(\gamma))]\gamma[s_{\theta}^{-1}(s(\gamma))]^{-1})&\text{ if }\gamma\in\GG_{s(\theta)}^{s(\theta)},\\
        0&\text{ otherwise.}
    \end{cases}
\end{equation*}
Note that for $\gamma\in\mathcal G^{s(\theta)}_{s(\theta)}$, $f^{\theta}(\gamma)=f(\theta\gamma\theta^{-1})=L_{\theta^{-1}}R_{\theta^{-1}}f(\gamma)$. 

\begin{defn} A countable p.m.p.\ groupoid $(\mathcal{G},\nu)$ is said to be \textit{inner amenable} if there exists a diffuse, conjugation invariant mean on $(\mathcal{G},\nu)$ which is symmetric and balanced. That is, a finitely additive, probability measure $m:\mathcal{G}\to [0,1]$ satisfying $m(D)=0$ if $\nu(D)<\infty$, $m(D^{\theta})=m(D)$, $m(D)=m(D^{-1})$, and $m(\mathcal{G}^A_A)=\mu(A)$, for every Borel subsets $D\subset\mathcal{G}$ and $A\subset \GG^{(0)}$, and every $\theta\in[\mathcal{G}]$. 
\end{defn}

\begin{remark}\label{rem:inneramenability} Note that if $m$ is a measure as in the prior definition and $B\subset \GG^{(0)}$ is measurable, $m(r^{-1}(B))=\mu(B)$. Indeed, since 
\begin{equation*}
    1=m(\GG)=\mu(B)+\mu(\GG^{(0)}\setminus B)+m(\GG_{B}^{\GG^{(0)}\setminus B})+m(\GG_{\GG^{(0)}\setminus B}^{B})=1+2m(\GG_{B}^{\GG^{(0)}\setminus B}),
\end{equation*}
we have that $m(\GG_{B}^{\GG^{(0)}\setminus B})=0$, and therefore, $m(r^{-1}(B))=m(\GG_{B})=\mu(B)+m(\GG_{B}^{\GG^{(0)}\setminus B})=\mu(B)$. Moreover, since $m$ is symmetric, $m(r^{-1}(B))=m(r^{-1}(B)^{-1})=m(s^{-1}(B))$.
\end{remark}

\begin{prop}\label{prop:inneramnotproperprox} Let $(\mathcal{G},\nu)$ be a countable p.m.p.\ groupoid with unit space $(\GG^{(0)},\mu)$. If $\mathcal{G}$ is inner amenable, then it is not properly proximal.
\end{prop}
\begin{proof} Let $m$ be a diffuse, conjugation invariant mean on $(\mathcal{G},\nu)$ which is symmetric and balanced. Notice this implies the existence of a conjugation $[[\mathcal{G}]]$-invariant state on $L^{\infty}(\mathcal{G},\nu)$ that vanishes on $c_0(\mathcal{G})$. Indeed, define $\phi_m(f)=\int_{\mathcal{G}}f(\gamma)\,\mathrm dm(\gamma)$. Observe that for $f\in L^{\infty}(\GG,\nu)$ and $\theta\in[\GG]$, and from the conjugation invariance of $m$, we obtain
\begin{align*}
    \phi_m(f^{\theta})&=\int_{\GG}f^{\theta}(\gamma)\,\mathrm dm(\gamma)=\int_{\GG}f(\gamma^{\theta^{-1}})\,\mathrm dm(\gamma)
    =\phi_m(f).
\end{align*}
To see that $\phi_m$ vanishes on $c_0(\GG)$, let $f\in c_0(\GG)$ and $\eps>0$. Take $B,B'\subset \GG^{(0)}$ with $\mu(B),\mu(B')\leq\eps$ and $F\subset \GG$ with $\nu(F)<\infty$ for which $\|(1-\chi_{\GG^B})(1-\chi_{\GG_{B'}})f(1-\chi_F)\|_{\infty}\leq\eps$. From the diffuseness of $m$, $m(F)=0$, and from the remark above $m(\GG^B)=m(r^{-1}(B))=\mu(B)\leq\eps$ and $m(\GG_{B'})=m(s^{-1}(B'))=\mu(B')\leq\eps$. Therefore, we have that
\begin{align*}
    |\phi_m(f)|&=\left|\int_{\GG\setminus F}f(\gamma)\,\mathrm{d}m(\gamma)\right|\leq \|f\|_{\infty}(m(\GG^B)+m(\GG_{B'}))+\int_{\GG}|((1-\chi_{\GG^B})(1-\chi_{\GG_{B'}})f(1-\chi_F))(\gamma)|\,\mathrm{d}m(\gamma)\\
    &\leq 2\|f\|_{\infty}\cdot\eps+\eps.
\end{align*}
Since $\eps>0$ was arbitrary, $\phi_m(f)=0$. 

Therefore, the map $\phi_m$ factorizes to a left $[\mathcal{G}]$-invariant state on $\mathbb S(\mathcal{G})$, since $\phi_m(L_{\theta}f)=\phi_m(R_{\theta}L_{\theta}f)=\phi_m(f^{\theta^{-1}})=\phi_m(f)$, for every $\theta\in[\mathcal{G}]$ and $f\in\mathbb S(\mathcal{G})$. Now, from \Cref{rem:inneramenability}, for any measurable subset $A\subset\GG^{(0)}$, we get that $\phi_m(\chi_A)=m(r^{-1}(A))=\mu(A)$. Since simple functions are norm dense in $L^{\infty}(\GG^{(0)})$ and $\phi_m$ is a state, we obtain $\phi_m|_{L^{\infty}(\GG^{(0)})}=\int_{\GG^{(0)}}\cdot\:\,\mathrm{d}\mu$ is normal. Altogether, we conclude that $\mathcal{G}$ is not properly proximal. 
\end{proof}

It follows from \cite[Proposition 3.17]{Kida-TuckerDrob} and \Cref{prop:inneramnotproperprox} that an aperiodic amenable countable p.m.p.\ groupoid cannot be properly proximal. In fact, we will show in \Cref{thm:properproxR} that every properly proximal groupoid gives rise to a relatively properly proximal von Neumann algebra, and Definition \ref{defn: prop prox vN alg} and the remark thereafter shows that relatively properly proximal von Neumann algebras cannot be injective. Hence amenable countable p.m.p.\ groupoids are not properly proximal.

\section{von Neumann algebras associated to properly proximal groupoids}\label{sec: invariant vNa}

\subsection{Properly proximal von Neumann algebras}

We recall the notion of properly proximal von Neumann algebras introduced in \cite{ding2023properproximality}, as a generalization of properly proximal countable groups from \cite{BIP21}.

Let $(M,\tau)\subset \mathbb{B}(L^{2}M)$ be a tracial von Neumann algebra. A hereditary C$^{*}$-subalgebra $\mathbb{X}$ of $\mathbb{B}(L^{2}M)$ is called an \textit{$M$-boundary piece} if $\mathbb{M}(\mathbb{X})\cap M$ and $\mathbb{M}(\mathbb{X})\cap M'$ are ultraweakly dense in $M$ and $M'$, respectively, where $\mathbb{M}(\mathbb{X})$ is the multiplier algebra of $\mathbb{X}$ and $M' = M'\cap \mathbb{B}(L^{2}M)$ is the commutant of $M$ in $\mathbb{B}(L^{2}M)$. To avoid pathological examples, we always assume $\mathbb{X} \neq \{0\}$, so then $\mathbb{K}(L^2M) \subset \mathbb{X}$ by the assumption of $\mathbb{X}$.

The space of compact operators $\K(L^2M)$ is always an example of an $M$-boundary piece. Another source of $M$-boundary pieces comes from subalgebras of $M$, which will be of major concern of this paper.

\begin{exmp}
    Suppose $\mathcal{B} = \{(B_{i}, \mathbb E_{i})\}_{i\in I}$ is a collection of unital von Neumann subalgebras $B_{i}$ of $M$ with normal faithful conditional expectations $\mathbb E_{i}\colon M\to B_{i}$. Denote by $e_{B_{i}} \in \mathbb{B}(L^{2}M)$ the orthogonal projection onto the space $L^{2}B_{i} \subset L^{2}M$ corresponding to $\mathbb E_{i}$. The boundary piece generated by the subalgebras $\mathcal B$ in $\mathbb B(L^2M)$, and denoted by $\mathbb{X}_{\mathcal{B}}$, is the C$^*$-algebra generated by $\{xJyJ e_{B_{i}}: x,y\in M, i\in I\}$.
\end{exmp}

Fix an $M$-boundary piece $\X$. Denote by $\K^L_\X(M) \coloneqq \overline{\B(L^2M)\X}^{\|\cdot\|_{\infty,2}} \subset \B(L^2M)$ the closure of $\mathbb{B}(L^{2}M)\mathbb{X}$ in the $\|\cdot\|_{\infty,2}$-topology inside $\mathbb{B}(L^{2}M)$, where $\| T\|_{\infty,2} = \sup_{x\in (M)_1}\| T\widehat{x}\|_2$ for $T\in \B(L^2M)$. It follows that $\mathbb{K}_{\mathbb{X}}^{L}(M)$ is a left ideal of $\mathbb{B}(L^{2}M)$ containing $M$ and $M'$ in its space of right multipliers.  Let
\[\mathbb{K}_{\mathbb{X}}(M) = (\mathbb{K}_{\mathbb{X}}^{L}(M))^{*} \cap \mathbb{K}_{\mathbb{X}}^{L}(M) = (\mathbb{K}_{\mathbb{X}}^{L}(M))^{*}\mathbb{K}_{\mathbb{X}}^{L}(M) \subset \mathbb{B}(L^{2}M)\]
be the hereditary C$^{*}$-subalgebra of $\mathbb{B}(L^{2}M)$ associated with $\mathbb{K}_{\mathbb{X}}(M)$, and note that both $M$ and $M'$ are in its multiplier algebra. We also define
\[\mathbb{K}_{\mathbb{X}}^{\infty,1}(M) \coloneqq \overline{\mathbb{K}_{\mathbb{X}}(M)}^{\|\cdot\|_{\infty,1}}\]
to be the $\|\cdot\|_{\infty,1}$-topology closure of $\K_{\mathbb X}(L^2M)$ in $\mathbb B(L^2M)$, where $\| T\|_{\infty,1} = \sup_{x,y\in (M)_1}|\langle T\widehat{x},\widehat{y}\rangle|$ and $T\in \B(L^2M)$. It follows that $\mathbb{K}_{\mathbb{X}}^{\infty,1}(M) = \overline{\X}^{\|\cdot\|_{\infty,1}}$, so by \cite[Proposition 2.4]{ding2023properproximality} (\cite[Proposition 2.2]{MR1616512}), an operator $T \in \mathbb{B}(L^{2}M)$ is contained in $\mathbb{K}_{\mathbb{X}}^{\infty,1}(M)$ if and only if there exist orthogonal families of projections $\{e_{i}\}_{i\in I}, \{\tilde{e}_{i}\}_{i\in I},\{f_{j}\}_{j\in J},\{\tilde{f}_{j}\}_{j\in J}\subset M$ such that $e_iJ\tilde{e}_iJTf_{j}J\tilde{f}_{j}J \in \mathbb{X}$ for every $i\in I, j\in J$.

For simplicity, when $\X = \K(L^2M)$ we write $\K^{\infty,1}(M)$ instead of $\K_{\K(L^2M)}^{\infty,1}(M)$, and when $\X = \X_{\mathcal{B}}$ for some collection $\mathcal{B} = \{(B_{i}, \mathbb E_{i})\}_{i\in I}$ of unital von Neumann subalgebras of $M$, we write $\K_{\mathcal{B}}^{\infty,1}(M)$ instead of $\K_{\X_{\mathcal{B}}}^{\infty,1}(M)$.

We now recall the definition of small-at-infinity boundary and (relative) proper proximality in \cite{ding2023properproximality}.

\begin{defn}[\protect{\cite{ding2023properproximality}}]
    The \textit{small-at-infinity boundary $\mathbb{S}_{\mathbb{X}}(M)$ of $M$ with respect to the $M$-boundary piece $\mathbb{X}$} is
\[\mathbb{S}_{\mathbb{X}}(M) = \{T\in \mathbb{B}(L^{2}M): [T,x] \in \mathbb{K}_{\mathbb{X}}^{\infty,1}(M) ,\  \text{for all } x \in M'\}.\]
In general, $\mathbb{S}_{\mathbb{X}}(M)$ is only an operator system containing $M$; it does not necessarily form a C$^{*}$-algebra.
\end{defn}

Again for simplicity, we write $\S(M)$ instead of $\S_{\K(L^2M)}(M)$ when $\X = \K(L^2M)$, and write $\S_{\mathcal{B}}(M)$ instead of $\S_{\X_{\mathcal{B}}}(M)$ when $\X = \X_{\mathcal{B}}$ for some collection $\mathcal{B} = \{(B_{i}, \mathbb E_{i})\}_{i\in I}$ of unital von Neumann subalgebras of $M$.

\begin{defn}[\protect{\cite[Theorem 6.2]{ding2023properproximality}}] \label{defn: prop prox vN alg}
     The von Neumann algebra $M$ is \textit{properly proximal relative to the boundary piece $\mathbb{X}$} if there does not exist an $M$-central state on $\mathbb{S}_{\mathbb{X}}(M)$ that is normal on $M$. We say $M$ is \textit{properly proximal} if $M$ is properly proximal relative to $\mathbb{K}(L^2M)$.
\end{defn}

If $\X,\Y$ are $M$-boundary pieces in $\B(L^2M)$ such that $\X\subset\Y$, then $\K^{\infty,1}_\X(M) \subset \K^{\infty,1}_\Y(M)$ and $\S_\X(M) \subset \S_\Y(M)$, so any $M$-central state on $\S_\Y(M)$ restricts to an $M$-central state on $\S_\X(M)$. Therefore, if $M$ is properly proximal relative to $\X$, then $M$ is properly proximal relative to $\Y$. In particular, if $M$ is properly proximal, then $M$ is properly proximal relative to any $M$-boundary $\X$. On the other hand, the largest $M$-boundary piece is $ \X_M = \B(L^2M)$, and $M$ is properly proximal relative to $\B(L^2M)$ if and only if there does not exist $M$-central state on $\B(L^2M)$ that is normal on $M$, or equivalently, $M$ is not injective. Hence if $M$ is properly proximal relative to any boundary piece $\X$, then $M$ is not injective.

Later we will need the following strengthened definition of proper proximality for von Neumann algebras when we study the groupoid von Neumann algebra associated with properly proximal groupoids.

\begin{defn}\label{def: A-properlyproximal rel A}
    Let $M$ be a von Neumann algebra, $\X$ an $M$-boundary piece, and $A\subset M$ a subalgebra with normal faithful conditional expectation $\mathbb E_A\colon M\to A$. We say $M$ is \emph{$A$-properly proximal relative to $\X$} if there does not exist an $M$-central state on $\S_\X(M) \cap (JAJ)' = \S_\X(M)\cap \langle M,e_A\rangle$ that is normal on $M$. Here, $e_A \in \B(L^2 M)$ denotes the Jones projection corresponding to $\mathbb E_A$.
\end{defn}

Since $\S_\X(M)\cap \langle M,e_A\rangle \subset \S_\X(M)$, it follows that $M$ is properly proximal relative to $\X$ whenever $M$ is $A$-properly proximal relative to $\X$.

\subsection{Proper proximality of twisted groupoid von Neumann algebras} From \cite[Theorem 6.4]{ding2023properproximality} we have that a countable discrete group $\Gamma$ is properly proximal if and only if the group von Neumann algebra $L\Gamma$ is properly proximal, and the group $\Gamma$ is properly proximal relative to a collection $\mathcal{G}$ of subgroups if and only if $L\Gamma$ is properly proximal relative to the collection of subalgebras $\{L\Lambda: \Lambda\in \mathcal{G}\}$. In the main theorem of this section, we extend such correspondence to countable p.m.p.\ groupoids, and moreover show it is invariant under twisting by $2$-cocycles. We prove a series of preparation lemmas for the proof of \Cref{thm:properproxR}. In the one below we show the multiplier operators of functions in $c_0\GG$ are relatively compact in $\B(L^2\GG)$.

\begin{lemma}\label{lem:c_0(G)relativelycompact}
    Let $(\GG,\nu)$ be a countable p.m.p.\ groupoid over the probability measure unit space $(\mathcal G^{(0)},\mu)$ and $c\in Z^2(\GG,\T)$. If $f \in c_0(\GG)$, then $M_f \in \K_{L^\infty \mathcal G^{(0)}}^{\infty,1}(L_c\GG)$.
\end{lemma}
\begin{proof}
    Fix $\eps > 0$. There exist measurable subsets $B,B' \subset \GG^{(0)}$ with $\mu(B),\mu(B') < \eps/2$ and a finite measure subset $F\subset \GG$ such that
\begin{equation*}
   \| (1-\chi_{\GG^B})(1-\chi_{\GG_{B'}})f(1-\chi_{F})\|_\infty\leq\eps/2.
\end{equation*}
By Lusin-Novikov Theorem, there exists a countable partition $F = \sqcup_{i=1}^\infty E_i$, where each $E_i = \Graph(\phi_i)$ for some $\phi_i \in [[\GG]]$. Since $\nu(F)<\infty$ there exists $N\in\mathbb{N}$ such that $\nu(\sqcup_{i=N}^\infty E_i)< \eps/2$, so that $\mu(\cup_{i=N}^\infty r(\phi_i))\leq \sum_{i=N}^\infty \mu(r(\phi_i)) = \sum_{i=N}^\infty \nu(\Graph (\phi_i)) \leq \eps/2$, and similarly $\mu(\cup_{i=N}^\infty s(\phi_i)) \leq \eps/2$. Let $C\coloneqq B\cup \cup_{i=N}^\infty r(\phi_i)$ and $C'\coloneqq B'\cup \cup_{i=N}^\infty s(\phi_i)$, so then $\mu(C),\mu(C')< \eps$. Define $p\coloneqq M_{\chi_{\GG^C}}$ and $p'\coloneqq M_{\chi_{\GG_{C'}}}$, so then $p,p'\in L^\infty \GG^{(0)}\subset L_c\GG$ and $\tau(p),\tau(p') < \eps$. For $i<N$, define $f_i$ on $\GG^{(0)}$ by 
\[f_i(x) = \begin{cases}f(s_{\phi_i}^{-1}(x)) &\text{ if } x\in s(\phi_i),\\ 0 &\text{ otherwise.}\end{cases}\] 
Then for $E\coloneqq \sqcup_{i=1}^{N-1} E_i$, we have $f\chi_E = \sum_{i=1}^{N-1} L_{\phi_i}f_i\circ s $ and \[M_{f\chi_E} = \sum_{i=1}^{N-1} (u_{\phi_i}^c)M_{f_i} e_{L^\infty \GG^{(0)}}(u_{\phi_i}^c)^* \in \K_{L^\infty \GG^{(0)}}^{\infty,1}(L_c\GG).\] 
Since $\K_{L^\infty \GG^{(0)}}^{\infty,1}(L_c\GG)$ is both an $L_c\GG$-$L_c\GG$ and a $J(L_c\GG)J$-$J(L_c\GG)J$ bimodule, $M_{(1-\chi_{\GG^C})(1-\chi_{\GG_{C'}})f\chi_{E}} =  (1-p)(1-p')M_{f\chi_E} \in \K_{L^\infty \GG^{(0)}}^{\infty,1}(L_c\GG)$, and 
\begin{align*}
   \| (1-\chi_{\GG^C})(1-\chi_{\GG_{C'}})f(1-\chi_{E}) \|_{\infty} &\leq \| (1-\chi_{\GG^C})(1-\chi_{\GG_{C'}})f(1-\chi_{F}) \|_{\infty} +\eps/2\\
   &\leq \| (1-\chi_{\GG^B})(1-\chi_{\GG_{B'}})f(1-\chi_{F}) \|_{\infty}+\eps/2<\eps. 
\end{align*}
Hence,
\begin{equation*}
    \begin{aligned}
        \| M_f - M_{(1-\chi_{\GG^C})(1-\chi_{\GG_{C'}})f\chi_{E}}\|_{\infty,1} &\leq  \| M_{f-(1-\chi_{\GG^C})(1-\chi_{\GG_{C'}})f}\|_{\infty,1} + \| M_{(1-\chi_{\GG^C})(1-\chi_{\GG_{C'}})f-(1-\chi_{\GG^C})(1-\chi_{\GG_{C'}})f\chi_{E}}\|_{\infty,1}\\
        &\leq \| M_{f\chi_{\GG^C}}\|_{\infty,1} + \| M_{f(1-\chi_{\GG^C})\chi_{\GG_{C'}}}\|_{\infty,1} + \| M_{(1-\chi_{\GG^C})(1-\chi_{\GG_{C'}})f(1-\chi_{E})}\|_{\infty,1}\\
        &=  \| pM_{f}\|_{\infty,1} + \| (1-p)Jp'JM_{f}\|_{\infty,1} + \| (1-p)(1-Jp'J)M_{f(1-\chi_{E})}\|_{\infty,1}\\
        & \leq 2\| M_f\| \eps^{1/2} + \eps.
    \end{aligned}
\end{equation*}
Since $\K_{L^\infty \GG^{(0)}}^{\infty,1}(L_c\GG)$ is closed under the $\|\cdot\|_{\infty,1}$-norm and $\eps>0$ was chosen arbitrarily, we conclude that $M_f\in \K_{L^\infty \GG^{(0)}}^{\infty,1}(L_c\GG)$.
\end{proof}

Let $(\mathcal G,\nu)$ be a countable p.m.p.\ groupoid with unit space $(\mathcal G^{(0)},\mu)$. Denote by $\E\colon (J(L^\infty \GG^{(0)})J)'\to L^\infty(\GG)$ the conditional expectation defined in \eqref{eqn: cond exp} and recalled below:
\begin{equation*}
    \E(T)(\gamma) = \langle T_{s(\gamma)}\delta_\gamma,\delta_\gamma\rangle_{\ell^2(\GG_{s(\gamma)})}, \qquad T = \int_{\GG^{(0)}}^\oplus T_{x}\,\mathrm{d}\mu(x)\in (J(L^\infty \GG^{(0)})J)'.
\end{equation*}
The next series of lemmas aim to show that $\E(\K_{L^\infty \GG^{(0)}}^{\infty,1}(L_c\GG)\cap\langle L_c\mathcal G,e_{L^{\infty}\GG^{(0)}}\rangle) \subset c_0(\GG)$. Unlike in the group case, in general $\E$ may not be continuous from the $\|\cdot\|_{\infty,1}$-topology to the norm topology for groupoids. Indeed, in the case when the unit space $(\GG^{(0)},\mu)$ is diffuse, choose measurable sets $B_n \subset \GG^{(0)}$ such that $\mu(B_n)\to 0$, and let $p_n = \chi_{B_n}\in L^\infty \GG^{(0)}$. For each $n$, consider $p_ne_{L^\infty\GG^{(0)}}$, which is a projection in $\K_{L^\infty \GG^{(0)}}^{\infty,1}(L_c\GG)\cap\langle L_c\mathcal G,e_{L^{\infty}\GG^{(0)}}\rangle$ since $p_n$ commutes with $e_{L^{\infty}\GG^{(0)}}$. On one hand, we have for $x,y\in (M)_1$,
\begin{align*}
    \abs{\langle p_ne_{L^{\infty}\GG^{(0)}} \widehat x, \widehat y\rangle} &= \abs{\langle \widehat{p_n \E_{L^\infty \GG^{(0)}}(x)}, \widehat y\rangle} = \abs{\tau(y^* p_n \E_{L^\infty \GG^{(0)}}(x))} = \abs{\tau(\E_{L^\infty \GG^{(0)}}(y^*)p_n \E_{L^\infty \GG^{(0)}}(x))}\\
    &\leq \tau(p_n) = \mu(B_n),
\end{align*}
so $\| p_ne_{L^\infty \GG^{(0)}}\|_{\infty,1} \leqslant \mu(B_n) \to 0$. On the other hand, we also have for $x\in \GG^{(0)}$ and $\gamma\in \GG_x$,
\[\E(p_ne_{L^\infty \GG^{(0)}})(\gamma) = \langle (p_n e_{L^\infty \GG^{(0)}})_x\delta_\gamma, \delta_\gamma\rangle  = p_n(r(\gamma))\langle \text{proj}_{\C \delta_{1_x}}\delta_\gamma, \delta_\gamma\rangle = \begin{cases} p_n(x) & \text{if }\gamma = 1_x, \\ 0 & \text{otherwise}. \end{cases}\]
Hence $\E(p_n e_{L^\infty \GG^{(0)}}) = \chi_{B_n}$ and $\| p_ne_{L^\infty \GG^{(0)}}\|_{\infty} = 1$ for all $n$.

\vspace{1mm}

\begin{lemma} \label{lem: Fourier coeff}
    Let $x,y\in L_c\GG$. For every $\eps>0$, there exists a finite-measure Borel subset $F\subset \GG$ such that
    \[\left\Vert \E_{L^\infty \GG^{(0)}}(xu_\theta^c y)\right\Vert_2 <\eps\]
    for every local bisection $\theta\in [[\GG]]$ satisfying $\nu(\theta\cap F) = 0$.
\end{lemma}
\begin{proof}
    First suppose $x= u_\alpha^c$ and $y = u_\beta^c$ for some local bisections $\alpha,\beta\in [[\GG]]$. If $\nu(\theta\cap \alpha^{-1}\beta^{-1})=0$, then $\alpha\theta\beta \cap \GG^{(0)}$ is a null set, so there exists a unitary operator $a_{\theta,\alpha,\beta} \in L^\infty\GG^{(0)}$ for which
    \[(xu_\theta^c y)\chi_{\GG^{(0)}} = (u_\alpha^c u_\theta^c u_\beta^c) \chi_{\GG^{(0)}} =a_{\theta,\alpha,\beta} u_{\alpha\theta\beta}^c \chi_{\GG^{(0)}} =0,\]
    and thus $\mathbb E_{L^\infty \GG^{(0)}}(xu_\theta^c y) = 0$. Hence in this case we can choose $F = \alpha^{-1}\beta^{-1}$ which has finite measure since $\alpha,\beta$ are local bisections.
    
    When $x = \sum_{i=1}^n a_i u_{\alpha_i}^c, y =\sum_{j=1}^m b_j u_{\beta_j}^c$ are finite linear combinations of partial isometries arising from local bisections $\alpha_1,\cdots,\alpha_n,\beta_1,\cdots,\beta_k$ in $[[\mathcal G]]$, then we let $F\coloneqq (\cup_{i=1}^n\alpha_i)^{-1}(\cup_{j=1}^m\beta_j)^{-1}$ a finite union of local bisections of finite $\nu$-measure. It follows from linearity and the computations above that whenever $\nu(\theta\cap F) = 0$, $\mathbb E_{L^\infty \GG^{(0)}}(xu_\theta^c y) = 0$.

    For general $x,y\in L_c\GG$ and $\eps>0$, we let $x_0,y_0 \in L_c\GG$ such that $\widehat{x}_0$, $\widehat{y}_0$ are supported on finite unions of local bisections and
    \[\|y\|\cdot \| x - x_0\|_2 + \| x\|\cdot \| y - y_0\|_2 < \eps.\]
    Let $F\coloneqq \supp(\widehat{x}_0)^{-1}\supp(\widehat{y}_0)^{-1}$, which is a set of finite measure. If $\theta\cap F$ is null, then by above we have $\mathbb E_{L^\infty \GG^{(0)}}(x_0 u_\theta^c y_0) = 0$, so
    \begin{equation*}
        \begin{aligned}
            \| \mathbb E_{L^\infty \GG^{(0)}}(xu_\theta^c y)\|_2 &\leq \| \mathbb E_{L^\infty \GG^{(0)}}((x-x_0)u_\theta^c y)\|_2 + \| \mathbb E_{L^\infty \GG^{(0)}}(x_0 u_\theta^c (y-y_0))\|_2\\
            &\leq \| y\|\cdot \| x - x_0\|_2 + \| x\|\cdot \| y - y_0\|_2 < \eps. 
        \end{aligned}
    \end{equation*}
    This proves the lemma.
\end{proof}

\begin{lemma} \label{lem: boundary coeff vanish}
    Let $T \in \X_{L^\infty \GG^{(0)}}$. For every $\eps>0$ there exists a finite-measure subset $F\subset (\GG,\nu)$ such that 
    \[|\langle T \widehat{u_\theta^c},\widehat{u_\theta^c}\rangle|<\eps, \]
    for every local bisection $\theta\in [[\GG]]$ satisfying $\nu(\theta\cap F) = 0$.
\end{lemma}
\begin{proof}
    Recall that $\X_{L^\infty \GG^{(0)}}$ is the hereditary C$^*$-subalgebra of $\B(L^2\GG)$ generated by the operators $xJyJ \mathbb E_{L^\infty \GG^{(0)}}$, with $x,y\in L_c\GG$. Hence $\X_{L^\infty \GG^{(0)}}$ is the norm closure of the linear span of operators of the form
    \begin{equation}\label{eq:dense elements X_LX}
        T=aJbJ\,\mathbb E_{L^\infty \GG^{(0)}} R\: \mathbb E_{L^\infty \GG^{(0)}}\,cJdJ, \:\:\:\text{ with }\:\: a,b,c,d\in L_c\GG,\:\:\ R\in \B(L^2\GG^{(0)}).
    \end{equation}
    Here, $\mathbb E_{L^\infty \GG^{(0)}}R\: \mathbb E_{L^\infty \GG^{(0)}}$ denotes the corresponding operator on $L^2\GG$ with initial and final space contained in $L^2\GG^{(0)}$.

    We first prove the assertion for one such algebraic term. Let $\theta\in [[\GG]]$ be a local bisection and $u_\theta^c\in L_c\mathcal G$ be the corresponding partial isometry. We have
    \[\mathbb E_{L^\infty \GG^{(0)}}JdJc\,\widehat{u_\theta^c}=\widehat{\mathbb E_{L^\infty \GG^{(0)}}(cu_\theta^c d^*)}\:\:\text{ and }\:\:\mathbb E_{L^\infty \GG^{(0)}}Jb^*Ja^*\,\widehat{u_\theta^c}=\widehat{\mathbb E_{L^\infty \GG^{(0)}}(a^*u_\theta^c b)}.\]
    Therefore, $\langle T\widehat{u_\theta^c},\widehat{u_\theta^c}\rangle = \Bigl\langle R\,\widehat{\mathbb E_{L^\infty \GG^{(0)}}(cu_\theta^c d^*)},\widehat{\mathbb E_{L^\infty \GG^{(0)}}(a^*u_\theta^c b)} \Bigr\rangle$, so
    \[|\langle T\widehat{u_\theta^c},\widehat{u_\theta^c}\rangle| \leq
    \| R\|\cdot \|\mathbb E_{L^\infty \GG^{(0)}}(cu_\theta^c d^*)\|_2 \|\mathbb E_{L^\infty \GG^{(0)}}(a^*u_\theta^c b)\|_2 .\]
    Given $\eps>0$, choose $\eps_0>0$ small enough so that $\| R\| \eps_0^2<\eps$. By \Cref{lem: Fourier coeff}, there exist a finite-measure Borel subset $F\subset \GG$ such that $\|\mathbb E_{L^\infty \GG^{(0)}}(cu_\theta^c d^*)\|_2, \|\mathbb E_{L^\infty \GG^{(0)}}(a^*u_\theta^c b)\|_2 < \eps_0$ whenever $\theta\cap F$ is null. For any such $\theta$, $|\langle T\widehat{u_\theta^c}, \widehat{u_\theta^c} \rangle| \leq \| R\| \eps_0^2 <\eps$, and this proves the lemma in the case when $T\in \X_{L^\infty \GG^{(0)}}$ is of the form given in \eqref{eq:dense elements X_LX}. Clearly, the lemma also holds when $T$ is finite combinations of such operators.

    Finally, let $T\in \X_{L^\infty \GG^{(0)}}$ be arbitrary. Choose an algebraic finite sum $T_0$ as above such that $\| T - T_0\| <\eps/2$. Since $\| u_\theta^c \| \leq 1$, we have
    \[|\langle (T-T_0)\widehat{u_\theta^c}, \widehat{u_\theta^c}\rangle|\leq \| T-T_0\| \cdot\|\widehat{u_\theta^c}\|^2 
    <\eps/2. \]
    Applying the algebraic case to $T_0$ with tolerance $\eps/2$, we have a finite-measure Borel set $F\subset \GG$ such that $|\langle T_0\widehat{u_\theta^c}, \widehat{u_\theta^c}\rangle| < \eps/2$ whenever $\theta\cap F$ is null. Hence,
    \[ |\langle T\widehat{u_\theta^c}, \widehat{u_\theta^c}\rangle| \leq |\langle (T-T_0)\widehat{u_\theta^c}, \widehat{u_\theta^c}\rangle| + |\langle T_0\widehat{u_\theta^c}, \widehat{u_\theta^c}\rangle| <\eps \]
    whenever $\theta\cap F$ is null, as wanted.
\end{proof}

\begin{thm} \label{thm: cond exp to c_0}
    Let $(\GG,\nu)$ be a countable p.m.p.\ groupoid with unit space $(\mathcal G^{(0)},\mu)$. Then
    \[\E\big(\K_{L^\infty \GG^{(0)}}^{\infty,1}(L_c\GG) \cap (J(L^\infty \GG^{(0)}J)'\big) \subset c_0(\GG).\]
\end{thm}
\begin{proof}
    Let $T\in \K_{L^\infty \GG^{(0)}}^{\infty,1}(L_c\GG) \cap (J(L^\infty \GG^{(0)}J)'$, and set $f \coloneqq \E(T)\in L^\infty(\GG)$. We prove that $f\in c_0(\GG)$.

    Assume towards a contradiction that $f\not\in c_0(\GG)$. By \Cref{lem: non c_0-func}, after replacing $T$ by $\lambda T$ with $\lambda\in\mathbb T$, there exist $\eps>0$ and a Borel subset $S\subset \GG$, which is not a $c_0$-set, such that $\Real(f(\gamma))\geq \eps$, for almost every $\gamma\in S$. By \Cref{lem: non c_0-set}, there exists $\delta>0$ such that for every finite-measure Borel set $F\subset S$ there is a local bisection $\theta\subset S\setminus F$ such that $\nu(\theta)=\mu(s(\theta))=\mu(r(\theta)) \geq \delta$. 
    
    Since $\K_{L^\infty \GG^{(0)}}^{\infty,1}(L_c\GG)=\overline{\X}_{L^\infty \GG^{(0)}}^{\|\cdot\|_{\infty,1}}$, there exists $T_0\in \X_{L^\infty \GG^{(0)}}$ such that $\| T-T_0\|_{\infty,1}<\frac{\eps\delta}{4}$. By \Cref{lem: boundary coeff vanish} applied to $T_0$ with tolerance $\eps\delta/4$, there exists a finite-measure Borel set $F_0\subset \GG$ such that $|\langle T_0 \widehat{u_\theta^c}, \widehat{u_\theta^c}\rangle| < \frac{\eps\delta}{4}$ whenever $\theta\in[[\mathcal G]]$ is a partial bisection with $\nu(\theta\cap F_0)=0$. Choose a local bisection $\theta\subset S\setminus F_0$ with $\nu(\theta)\geq \delta$. Since $T\in (J(L^\infty \GG^{(0)})J)'$ is decomposable, from \eqref{eqn: inner prod} we have
    \[\langle T\widehat{u_\theta^c},\widehat{u_\theta^c}\rangle = \int_\theta \E(T)(\gamma) \,\mathrm{d}\nu(\gamma) = \int_\theta f(\gamma) \,\mathrm{d}\nu(\gamma). \]
    Since $\theta\subset S$, it follows that $|\langle T\widehat{u_\theta^c}, \widehat{u_\theta^c}\rangle|\geq \Real(\langle T\widehat{u_\theta^c},\widehat{u_\theta^c}\rangle) \geq \eps\cdot \nu(\theta) \geq \eps\delta$. On the other hand, $u_\theta^c\in L_c\GG$ is a contraction, so by the definition of the $\|\cdot\|_{\infty,1}$-norm, $|\langle (T - T_0)\widehat{u_\theta^c}, \widehat{u_\theta^c}\rangle| 
    \leq \| T-T_0\|_{\infty,1}<\eps\delta/4$. Also, by the choice of $F_0$ and the fact that $\theta\cap F_0$ is null, we have $|\langle T_0 \widehat{u_\theta^c}, \widehat{u_\theta^c}\rangle| < \eps\delta/4$. Therefore,
    \[\eps\delta \leq |\langle T\widehat{u_\theta^c}, \widehat{u_\theta^c}\rangle| \leq |\langle (T - T_0)\widehat{u_\theta^c}, \widehat{u_\theta^c}\rangle| + |\langle T_0\widehat{u_\theta^c}, \widehat{u_\theta^c}\rangle| < \frac{\eps\delta}{4} + \frac{\eps\delta}{4} = \frac{\eps\delta}{2},\]
    a contradiction. Therefore, we must have that $f=\E(T)\in c_0(\GG)$.
\end{proof}

\begin{thm}\label{thm:properproxR}
    Let $(\GG,\nu)$ be a countable p.m.p.\ groupoid with unit space $(\mathcal G^{(0)},\mu)$. The following statements are equivalent:
    \begin{enumerate}
        \item[(i)] $(\GG,\nu)$ is properly proximal.
        \item[(ii)] For any $2$-cocycle $c\in Z^2(\GG,\T)$, $L_c\GG$ is $L^\infty \GG^{(0)}$-properly proximal relative to $L^\infty \GG^{(0)}$.
        \item[(iii)] There is a $2$-cocycle $c\in Z^2(\GG,\T)$ such that $L_c\GG$ is $L^\infty \GG^{(0)}$-properly proximal relative to $L^\infty \GG^{(0)}$.
    \end{enumerate}

\end{thm}

\begin{proof} 
The implication (ii)$\:\Rightarrow$(iii) is immediate.

(i)$\:\Rightarrow$(ii). Suppose $L_c\GG$ is not properly proximal for some $2$-cocycle $c\in Z^2(\GG,\T)$, and let $\psi$ be a $L_c\GG$-central state on $\S(L_c\GG)\cap (J(L^\infty \GG^{(0)})J)'$ that is normal on $L_c\GG$. Let $\iota\colon L^\infty\GG\to \B(L^2\GG)$ be the canonical embedding given by $\iota(f) = M_f$ for $f\in L^\infty\GG$. By \Cref{lem:c_0(G)relativelycompact}, $\iota(c_0\GG)\subset \K_{L^\infty \GG^{(0)}}^{\infty,1}(L_c\GG)$. For any $f\in\mathbb S(\GG)$, $\theta\in [[\GG]]$ and $b\in L^{\infty}\GG^{(0)}$, we have $[M_f,v_{\theta}]=(M_{f\chi_{\GG_{r(\theta)}}-R_{\theta}f})v_{\theta}\in \K_{L^{\infty}\GG^{(0)}}^{\infty,1}(L_c\GG)$ and $[M_f,JbJ]=0$, so $M_f\in \mathbb S_{L^{\infty}\GG^{(0)}}(L\GG)$ by \cite[Lemma 6.1]{ding2023properproximality}. Clearly $M_f\in (J(L^\infty \GG^{(0)})J)'$, so we have $\iota(\S(\GG))\subset \S_{L^\infty \GG^{(0)}}(L_c\GG)\cap (J(L^\infty \GG^{(0)})J)'$. Since $\psi$ is $L_c\GG$-central, for $f\in \S(\GG) \subset L^\infty(\GG)$, we have
\[\psi(M_{L_\theta f})= \psi(\Ad(u^c_\theta)(M_{f})) = \psi(M_{f}(u^c_\theta)^* u^c_\theta) =\psi(M_{f}\chi_{s(\theta)})= \psi(M_{f\chi_{\GG^{s(\theta)}}}).\]  
Therefore, we conclude that $\psi\circ\iota$ is a left $[[\GG]]$-invariant state on $\S(\GG)$ that is normal on $L^{\infty}\GG^{(0)}$, and thus, $\GG$ is not properly proximal.

(iii)$\:\Rightarrow$(i). Let $c\in Z^2(\GG,\T)$ be a $2$-cocycle. Let $\E\colon (J(L^\infty \GG^{(0)})J)'\to L^\infty\GG$ denote the canonical faithful normal conditional expectation constructed in (\ref{eqn: cond exp}) via fiberwise diagonal map, and restrict it to $\S_{L^{\infty}\GG^{(0)}}(\GG)\cap (J(L^\infty \GG^{(0)})J)'$. Since $\E$ is right $[\GG]$-equivariant, by  \Cref{thm: cond exp to c_0}, we have for any $T\in \S_{L^\infty \GG^{(0)}}(L_c\GG)\cap (J(L^\infty \GG^{(0)})J)'$ and $\theta\in [\GG]$,
\[\E(T) - R_\theta(\E(T)) = \E(T-\Ad(v_\theta^c)(T)) = \E([T,v_\theta^c](v_\theta^c)^*)\in \E(\K^{\infty,1}_{L^\infty \GG^{(0)}}(L_c\GG) \cap (J(L^\infty \GG^{(0)})J)') \subset c_0(\GG),\]
by definition of $T\in\mathbb S_{L^{\infty}\GG^{(0)}}(L_c\GG)$ and since $\K^{\infty,1}_{L^\infty \GG^{(0)}}(L_c\GG)$ is a $JL_c\GG J$-bimodule. Hence we have $\E(\S(L_c\GG) \cap (J(L^\infty \GG^{(0)}) J)')\subset \S(\GG)$.

If $\GG$ is not properly proximal there exists a state $\varphi\colon \mathbb S(\GG)\to\mathbb C$ that is $[[\GG]]$-invariant and normal on $L^{\infty}\GG^{(0)}$. Consider the state $\psi\coloneqq \varphi\circ \E \colon \S(L_c\GG) \cap (J(L^\infty \GG^{(0)}) J)' \to \C$. By the left equivariance of $\E$ and the $[\GG]$-invariance of $\varphi$ we have for every $\theta\in [\GG]$,
\[\psi(u^c_\theta T) = \varphi(\E(\Ad(u^c_\theta)(Tu^c_\theta) )) = \varphi(L_\theta(\E(Tu^c_\theta))) = \varphi(\E(Tu^c_\theta)) = \psi(Tu^c_\theta),\] 
so $\psi$ is $L_c\GG$-central. Also, by \eqref{eqn: restr of cond exp} we have $\E|_{L_c\GG} = \E_{L^\infty\GG^{(0)}}$, where $\E_{L^\infty\GG^{(0)}}\colon L_c\GG\to L^\infty\GG^{(0)}$ is the canonical normal conditional expectation, so 
\[\psi|_{L_c\GG} = \varphi|_{L^\infty \GG^{(0)}}\circ \E_{L^\infty \GG^{(0)}}\]
is normal on $L_c\GG$. This proves $L_c\GG$ is not $L^\infty \GG^{(0)}$-properly proximal relative to $L^\infty \GG^{(0)}$.
\end{proof}

Particularly, when $\GG= \Gamma$ is a countable discrete group, it follows that $\Gamma$ is properly proximal if and only if there exists, and therefore for every, $2$-cocycle $c\in Z^2(\Gamma,\T)$ such that the twisted group von Neumann algebra $L_c\Gamma$ is properly proximal.

When $L^\infty \GG^{(0)}\subset L\GG$ is a mixing subalgebra, we can upgrade relative proper proximality to actual proper proximality by \cite[Theorem 1.1]{DS24structure} (see also \cite[Theorem 5.14]{toyosawayang}). 

\begin{corollary}
    Let $(\GG,\nu)$ be a properly proximal countable p.m.p.\ groupoid over the probability measure unit space $(\GG^{(0)},\mu)$ and $c\in Z^2(\GG,\T)$. If $L^\infty \GG^{(0)}$ is a mixing subalgebra of $L_c\GG$, then there exists a central projection $z \in \ZZ(L_c\GG)$ such that $zL_c\GG$ is properly proximal and $(1-z)L_c\GG$ is injective.

    In particular, if $L_c\GG$ is moreover a factor, then $L_c\GG$ is properly proximal.
\end{corollary}

For the conditions under which the twisted von Neumann algebra $L_c\GG$ is a factor, we refer the reader to \cite[Corollary 5.6]{groupoidfactor}. 

As an application of \Cref{thm:properproxR}, we obtain a sort of converse to \Cref{prop:productsG1G2} in the setting of icc groupoids. Following the definition from \cite[Definition 4.2]{groupoidfactor}, a countable p.m.p.\ groupoid $(\GG,\nu)$ is \textit{icc} (infinite conjugacy class) if it satisfies the following condition: whenever $B\subset [\mathcal G]$ is a nontrivial bisection with finite measure conjugacy class (i.e. the set $\cup_{\gamma\in\mathcal G}\gamma B\gamma^{-1}$ has finite $\nu$-measure), then $B\subset\mathcal G^{(0)}$. Before the proof, we need the following result. 

\begin{prop}\label{prop:dichotomypropprox}
Let $(\GG,\mu\circ\lambda)$ be a countable p.m.p.\ groupoid with unit space $(G^{(0)},\mu)$ and let $c\in Z^{2}(G,\mathbb T)$. Then $L_c\GG$ is not $L^\infty\GG^{(0)}$-properly proximal relative to $L^\infty\GG^{(0)}$ if and only if at least one of the following holds:
\begin{enumerate}
    \item there exists a nonzero central projection $p\in\mathcal  Z(L_c\GG)$ such that $L^\infty\GG^{(0)} p\subset L_c\GG p$ is coamenable;
    \item there exists an $L_c\GG$-central state
    $\varphi\colon \S_{L^\infty\GG^{(0)}}(L_c\GG)\cap\langle L_c\GG,e_{L^\infty\GG^{(0)}}\rangle\to \C$ such that $\varphi|_{L_c\GG}$ is normal and
    \[\varphi\bigl(\K_{L^\infty\GG^{(0)}}^{\infty,1}(L_c\GG)\cap\langle L_c\GG,e_{L^\infty\GG^{(0)}}\rangle\bigr)=0.\]
\end{enumerate}
\end{prop}

\begin{proof}
For simplicity, write $M = L_c\GG$, $A = L^\infty\GG^{(0)}$, and still let $\E \colon \langle M,e_A\rangle=(JAJ)'\to L^\infty\GG$ denote the canonical faithful normal conditional expectation given by the fiberwise diagonal map as defined in \eqref{eqn: cond exp}.

Recall that we showed in the proof of \Cref{thm:properproxR} that the multiplication representation
$\iota \colon L^\infty(\GG)\to \B(L^2(\GG))$ given by $\iota(f)=M_f$ satisfies $\iota(\S(\GG)) \subset \S_A(M)\cap \langle M,e_A \rangle$. Moreover, we also have $\E(\K_A^{\infty,1}(M)\cap \langle M,e_A\rangle) \subset c_0(\GG)$ and $\E(\S_A(M)\cap \langle M,e_A \rangle) \subset \S(\GG)$.

Suppose first that $M$ is not $A$-properly proximal relative to $A$. Then there exists an $M$-central state $\Phi\colon \S_A(M)\cap \langle M,e_A\rangle \to \C$ whose restriction to $M$ is normal. Define a state $\phi\colon \S(\GG)\to \C$ by $\phi(f)=\Phi(M_f)$ for $f\in \S(\GG)$. The restriction $\phi|_A$ is normal. Moreover, $\phi$ is left $[\GG]$-invariant. Indeed, for $\theta\in[\GG]$,
\[\phi(L_\theta f)=\Phi(M_{L_\theta f}) = \Phi\bigl(u_\theta^cM_f(u_\theta^c)^*\bigr) = \Phi\bigl(M_f(u_\theta^c)^*u_\theta^c\bigr) = \Phi(M_f)=\phi(f).\]
We distinguish two cases. Assume first that $\phi|_{c_0(\GG)}=0$. Define $\varphi \coloneqq \phi\circ \E\colon \S_A(M)\cap \langle M,e_A\rangle \to \C$, which is a well-defined state and $\varphi|_{\K_A^{\infty,1}(M)\cap\langle M,e_A\rangle}=0$. We claim that $\varphi|_M$ is normal and $\varphi$ is $M$-central. By \eqref{eqn: restr of cond exp}, the restriction of the 
expectation $\E$ to $M$ is given by $\E(x)= \E_A(x)\circ r$ for $x\in M$, where $\E_A \colon M\to A$ is the canonical conditional expectation. Therefore, $\varphi(x)=\phi(\E_A(x))$ for $x\in M$, so $\varphi|_M$ is normal. Now, for $a\in A = L^\infty\GG^{(0)}$ and $T\in\S_A(M)\cap \langle M,e_A\rangle $, the $L^\infty\GG$-bimodularity of $\E$ gives
\[\varphi(aT)=\phi(a\E(T))=\phi(\E(T)a)=\varphi(Ta).\]
Next, for $\theta\in[\GG]$, using the left equivariance of $\E$ and the left $[\GG]$-invariance of $\phi$, we obtain
\[\varphi(u_\theta^cT) =\phi(\E(u_\theta^cT)) =\phi\bigl(\E(\Ad(u_\theta^c)(Tu_\theta^c))\bigr) =
\phi\bigl(L_\theta(\E(Tu_\theta^c))\bigr) =\phi(\E(Tu_\theta^c))=\varphi(Tu_\theta^c).\]
Since $A$ together with the bisection unitaries $\{u_\theta^c:\theta\in [\GG]\}$ generate $M$, the preceding identities imply that $\varphi$ is $M$-central. Altogether, these show condition (2) holds. 

Assume now that $\phi|_{c_0\GG}\neq 0$. Let $\eta \coloneqq \phi|_{c_0\GG}$ be the restriction state. Since $c_0(\GG)$ is an invariant ideal in $L^\infty \GG$, the positive functional $\eta$ admits a canonical positive extension to the multiplier algebra of $c_0\GG$. Restricting this extension to $L^\infty \GG$ gives a positive functional $m \colon L^\infty \GG  \to \C$. Concretely, if $(h_i)_i\subset c_0(\GG)_+$ is a contractive approximate unit, then
\[m(f)= \lim_{i\to\mathcal U}\phi(h_i^{1/2}fh_i^{1/2}), \qquad f\in L^\infty\GG,\]
for some free ultrafilter $\mathcal U$. This limit is independent of the chosen approximate unit. In particular, we have $m(1)=\|\eta\|>0$. Since $\eta$ is left $[\GG]$-invariant and each $L_\theta$ preserves $c_0(\GG)$, uniqueness of the multiplier extension gives $m\circ L_\theta=m$, for $\theta\in[\GG]$. Also, for every $a\in A_+$,
$0\leq m(a)\leq\phi(a)$. Thus $m|_A$ is dominated by the normal positive functional $\phi|_A$, and hence $m|_A$ is normal.

Define $\omega\coloneqq m\circ \E \colon \langle M,e_A\rangle \to \C$. Then $\omega$ is a nonzero positive functional since $\omega(1)=m(1)>0$, and its restriction to $M$ is normal since $\omega(x)=m(\mathbb E_A(x))$ for $x\in M$. As above, for $a\in A$, $T\in\langle M,e_A\rangle $, and $\theta\in[\GG]$, we have $\omega(aT)=\omega(Ta)$ and
\[\omega(u_\theta^cT) = m(\E(u_\theta^cT)) = m\bigl(L_\theta(\E(Tu_\theta^c))\bigr) = m(\E(Tu_\theta^c)) = \omega(Tu_\theta^c).\]
Since $A$ and the bisection unitaries generate $M$, a bounded strong-$*$ approximation argument shows that $\omega$ is $M$-central.

Since $\omega|_M$ is normal, there is a unique element $z\in L^1(M,\tau)_+$ such that $\omega(x)=\tau(zx)$, for every $x\in M$. For $x,y\in M$, the $M$-centrality of $\omega$ gives $\tau(zxy)=\omega(xy)=\omega(yx)=\tau(zyx)$, and by traciality, $\tau((zx-xz)y)=0$, for $y\in M$. This means $z$ is affiliated with $\mathcal Z(M)$. Since $\omega\neq0$, we have $z\neq 0$. Choose $\eps>0$ such that $p=\chi_{[\eps,\infty)}(z)$ is a nonzero projection in $\mathcal Z(M)$. On $Mp$, the operator $zp$ is bounded below by $\eps p$, so $a \coloneqq (zp)^{-1/2}\in\mathcal  Z(Mp)$ is a bounded positive element. Define
\[\Omega(T)=\frac1{\tau(p)}\omega(aTa), \qquad T\in p\langle M,e_A\rangle p.\]
The functional $\Omega$ is an $Mp$-central state with $\Omega(p)=\tau(za^2)/\tau(p)=1$. Moreover, for every $x\in Mp$, $\Omega(x)= \tau(zaxa)/\tau(p) = \tau(x)/\tau(p)$, showing $\Omega|_{Mp}$ is normal. By the basic-construction characterization of coamenability, this shows that $Ap\subset Mp$ is coamenable. Thus condition (1) holds.

Conversely, condition (2) directly implies that $M$ is not $A$-properly proximal relative to $A$. Suppose condition (1) holds. By coamenability, there exists an $Mp$-central state $\Omega \colon p\langle M,e_A\rangle p\to \C$ such that $\Omega(x)=\tau(x)/{\tau(p)}$ for $x\in Mp$. Define a state $\Phi\colon \S_A(M)\cap \langle M,e_A\rangle \to \C$ by $\Phi(T)=\Omega(pTp)$. Since $p\in\mathcal Z(M)$, the state $\Phi$ is $M$-central, and $\Phi(x)=\tau(px)/\tau(p)$ is normal whenever $x\in M$. By definition, $M$ is not $A$-properly proximal relative to $A$, completing the proof.
\end{proof}

In \cite[Proposition 8.2]{ding2023properproximality}, the authors show that a von Neumann algebra $M$ is not properly proximal relative to a boundary piece $\mathbb X$ if and only if $Mp$ is amenable for some nonzero central projection $p$, or the state arising from the lack of proper proximality on $\mathbb S_{\mathbb X}(M)$ vanishes on $\mathbb K_{\mathbb X}^{\infty,1}(M)$. Therefore, the prior result establishes an analogous characterization when $M$ is not $A$-properly proximal relative to a von Neumann subalgebra $A$.

\begin{prop}\label{prop:productsvNalg} Let $(\mathcal G_1,\nu_1)$ and $(\mathcal G_2,\nu_2)$ be countable p.m.p.\ groupoids. Denote by $\mathcal G=\mathcal G_1\times\mathcal G_2$ the direct product of $\mathcal G_1$ and $\mathcal{G}_2$ with measure $\nu$. Assume that $\mathcal G_i$ is icc, ergodic and nonamenable for $i=1,2$. If $\mathcal G$ is properly proximal, then $\mathcal G_1$ and $\mathcal G_2$ are properly proximal. 
\end{prop}

\begin{proof} Suppose, without loss of generality, that $\mathcal G_1$ is not properly proximal. Denote by $X_i=\mathcal G_i^{(0)}$, $M_i=L(\mathcal G_i)$, $A_i=L^{\infty}(X_i)$, for $i=1,2$, and let $M=L(\mathcal G)=M_1\overline{\otimes}M_2$ and $A=L^{\infty}(X_1\times X_2)=A_1\overline{\otimes}A_2$. By \Cref{thm:properproxR} applied to the trivial cocycle, $M_1$ is not $A_1$-properly proximal relative to $A_1$. If there were a nonzero central projection $p\in\mathcal Z(M_1)$ for which $A_1p\subset M_1p$ were coamenable, \cite[Corollary B]{groupoidfactor} would give $p=1$. This would imply $\mathcal G_1$ is amenable \cite[Proposition 2.4]{OzawaPopa07}, contrary to our assumption. Therefore, by \Cref{prop:dichotomypropprox}, there exists an $M_1$-central state $\psi:\mathbb S_{A_1}(M_1)\cap \langle M_1,e_{A_1}\rangle\to\mathbb C$ such that $\psi|_{M_1}$ is normal and $\psi$ vanishes on $\mathbb K^{\infty,1}_{A_1}(M_1)\cap\langle M_1,e_{A_1}\rangle$. 

For the rest of the proof we follow \cite[Theorem 8.8]{ding2023properproximality}. Consider the u.c.p. map $\text{Ad}(P_1):\mathbb B(L^2(M_1\overline{\otimes}M_2))\to\mathbb B(L^2(M_1))$ given by $\text{Ad}(P_1)(T)=P_1TP_1$, where $P_1:L^2(M_1\overline{\otimes}M_2)\to L^2(M_1)$ is the orthogonal projection. By \cite[Lemma 8.7]{ding2023properproximality}, $\text{Ad}(P_1)$ maps $\mathbb S_A(M)$ into $\mathbb S_{A_1}(M_1)$ and maps $\mathbb K_A^{\infty,1}(M)$ into $\mathbb K^{\infty,1}_{A_1}(M_1)$, and furthermore, it is not hard to check that it maps $\langle M,e_A\rangle$ into $\langle M_1,e_{A_1}\rangle$. Define $\varphi=\psi\circ\text{Ad}(P_1)$ on $\mathbb S_A(M)\cap\langle M,e_A\rangle$. It follows that $\varphi|_M=\psi|_{M_1}\circ\mathbb E_{M_1}$ is normal, and since $\text{Ad}(P_1)$ is $M_1$-bimodular and $\psi$ is $M_1$-central, we obtain $\varphi$ is also $M_1$-central. For $u\in\mathcal U(M_2)$ and $T\in \mathbb S_A(M)\cap \langle M,e_A\rangle$, using that $P_1u(JuJ)=P_1$, we have
\begin{align*}
    \varphi(uT)&=\psi(P_1uTP_1)=\psi(P_1u(JuJ)T(JuJ)^*P_1)=\psi(P_1T(JuJ)^*P_1)\\
    &=\psi(P_1T(JuJ)^*u(JuJ)P_1)=\psi(P_1TuP_1)=\varphi(Tu),
\end{align*}
where the second equality holds since $\psi$ vanishes on $\mathbb K_{A_1}^{\infty,1}(M_1)\cap\langle M_1,e_{A_1}\rangle$. Hence, $M_2$ is also in the centralizer of $\varphi$, giving that $\varphi$ is $M$-central. We have, therefore, constructed an $M$-central state on $\mathbb S_A(M)\cap\langle M,e_A\rangle$ whose restriction to $M$ is normal. By definition, $M$ is not $A$-properly proximal relative to $A$. Applying \Cref{thm:properproxR} again, we get $\mathcal G$ is not properly proximal, as wanted. 
\end{proof}


\subsection{Weak compactness and rigidity} Recall the notion of weak compactness introduced by Ozawa and Popa in \cite{OzawaPopa07}. A trace preserving action $\sigma:\Gamma\curvearrowright (Q,\tau)$ on a tracial von Neumann algebra is \textit{weakly compact} if there exists a state $\varphi$ on $\mathbb B(L^2Q)$ such that $\varphi|_Q=\tau$ and $\varphi\circ\text{Ad}(u)=\varphi$, for every $u\in\mathcal U(Q)\cup\sigma(\Gamma)$. A regular inclusion of tracial von Neumann algebras $Q\subset M$ is \textit{weakly compact} if the action $\NN_M(Q)\curvearrowright Q$ is weakly compact, where $\NN_M(Q)=\{u\in\mathcal{U}(M):uQu^*=Q\}$ denotes the normalizer of $Q$ in $M$.

We shall use the following corresponding notion for countable p.m.p.\ groupoid over the unit space $\mathcal G^{(0)}$.

\begin{defn}\label{def:weakcompact1} We say a countable p.m.p.\ groupoid $\mathcal{G}$ with unit space $(\mathcal G^{(0)},\mu)$ is \textit{weakly compact} if the action $[\GG]\curvearrowright L^{\infty}\GG^{(0)}$ is weakly compact, i.e., $L^{\infty}(\mathcal G^{(0)})\subset L\mathcal G$ is weakly compact.

Equivalently, this means that there exists a net of unit vectors $\zeta_i\in L^{2}(\GG^{(0)})\otimes \overline{L^2(\GG^{(0)})}$ satisfying:
\begin{enumerate}
    \item[(i)] $\|(a\otimes 1)\zeta_i-(1\otimes\overline{a})\zeta_i\|_2\to 0$, for all $a\in L^{\infty}(\GG^{(0)},\mu)$,
    \item[(ii)] $\|\zeta_i\circ(\theta\times \theta)-\zeta_i\|_2\to 0$, for all $\theta\in [\mathcal{G}]$, and
    \item[(iii)] $\lim_i\langle (a\otimes 1)\zeta_i,\zeta_i\rangle=\int_{\GG^{(0)}}a(x)d\mu(x)$, for all $a\in L^{\infty}(\GG^{(0)})$. 
\end{enumerate}
\end{defn}

\begin{corollary}\label{cor:weakcompact&properprox}
    Let $(\GG,\nu)$ be a countable p.m.p.\ groupoid over unit space $(\GG^{(0)},\mu)$. Assume $L\GG$ is properly proximal relative to $L^{\infty}\GG^{(0)}$. If $P\subset L\GG$ is a weakly compact regular von Neumann subalgebra, then $P\preceq_M L^\infty \GG^{(0)}$.

    In particular, if $(\RR,\nu)$ is a properly proximal countable measured equivalence relation on $(X,\mu)$, then $L\RR$ admits a weakly compact Cartan subalgebra $A$ if and only if $\RR$ is weakly compact and, in this case, $A$ is unitarily conjugate to $L^\infty X$.
\end{corollary}

\begin{proof} The first part follows from \cite[Theorem 6.11]{ding2023properproximality} since $L^{\infty}(\mathcal G^{(0)})$ is a regular subalgebra of $L\GG$. The in particular part follows from the definition of weak compactness of groupoids, \Cref{thm:properproxR}, \cite[Theorem 6.11]{ding2023properproximality} and \cite{popabettinumbers01}.
\end{proof}

\begin{appendices}
\section{Measure equivalence} \label{sec: appendix}

\noindent The notion of measure equivalence of countable discrete groups was introduced by Gromov as a measure-theoretic analogue of quasi-isometry of groups, and was later extended to locally compact group settings in \cite{BFS13}. Measure equivalence was first introduced to discrete measured groupoid setting in \cite{Bernoulligroupoids} as a generalization of stable orbit equivalence, and later also re-introduced in \cite{bddcohomologytransversegroupoid} from a different perspective. We show in this section that these two different definitions for measure equivalence are actually equivalent, and also give the necessary results needed in \Cref{thm:MEpp} to show that properly proximality is preserved under measure equivalence. We first recall the definition that appeared in \cite{bddcohomologytransversegroupoid}. 
    
\begin{defn}[\protect{\cite[Definition 4.2]{bddcohomologytransversegroupoid}}] \label{defn:meas equiv}
    Let $\GG,\HH$ be countable measured groupoids. A \emph{(measure equivalence) $(\GG,\HH)$-coupling} for $\GG$ and $\HH$ is a Lebesgue space $(\Omega,m)$ with a free measure-preserving $\GG\times \HH$-action admitting finite measure strict fundamental domains $X,Y\subset \Omega$ for the $\GG$-action and the $\HH$-action, respectively. We say that $\GG$ and $\HH$ are \emph{measure equivalent}, and write $\GG \sim_{\ME}\HH$, if there exists a $(\GG,\HH)$-coupling. 
\end{defn}

As noted in \cite[Section 4.1]{bddcohomologytransversegroupoid}, one can always take the $(\GG,\HH)$-coupling $(\Omega, m)$ to be \emph{ergodic}, meaning that the transformation groupoid $(\GG\times \HH)\ltimes \Omega$ is ergodic. Also, since the $\GG$-action commutes with the $\HH$-action on $\Omega$, it descends to an action $\GG \acts \HH\backslash\Omega$. On the other hand, for every composable pair $(g,y)\in \GG\times_{\GG^{(0)}}Y$, by freeness there exists a unique element $\alpha(g,y)\in \HH$ such that $(g,\alpha(g,y))y\in Y$, and it induces an action $\GG\acts Y$ by $g\cdot y \coloneqq (g,\alpha(g,y))y$. The composition $Y\hookrightarrow \Omega \twoheadrightarrow \HH\backslash\Omega$ defines a Borel isomorphisms which intertwines the $\GG$ actions on $Y$ and $\HH\backslash\Omega$. Moreover, the measure $m$ on $\Omega$ is invariant under the $\GG\times \HH$-action, so $\restr{m}{Y}$ is invariant under the $\GG$-action, and thus $\GG\ltimes Y$ is a countable p.m.p.\ groupoid when the unit space is equipped with the restricted measure $\restr{m}{Y}$ upon normalization. By symmetry, same arguments also hold when one considers $\HH$-action on the strict fundamental domain $X$ of the $\GG$-action.

We note that passing to a measure-equivalence coupling generally loses the measure-theoretic data attached to the original groupoids because such couplings do not, by themselves, encode the original measures with which the groupoids are equipped. On the other hand, properly proximal groupoids $\GG$ are defined in terms of the space of essentially bounded functions $L^\infty(\GG,\mu\circ\lambda)$, which depend on the measure class of $\mu$ on $\GG$. We therefore propose the following definition imposing restrictions on the measures of the measure equivalence couplings.

\begin{defn}
    Let $(\GG,\mu\circ\lambda), (\HH,\nu\circ\lambda)$ be countable p.m.p.\ groupoids. We say a measure equivalent $(\GG,\HH)$-coupling $(\Omega, m)$ is \emph{measure class preserving} if the anchor maps $t_\GG, t_\HH$ corresponding to the $\GG$-action and $\HH$-action push the measure $m$ to measures on $\GG^{(0)}$ and $\HH^{(0)}$ in the same measure class of $\mu$ and $\nu$, respectively. Namely, $(t_{\GG})_*m \sim \mu$ on $\GG^{(0)}$ and $(t_{\HH})_*m \sim \nu$ on $\HH^{(0)}$. 
    
    If there exists a measure class preserving measure equivalence coupling for $(\GG,\mu\circ\lambda)$ and $(\HH,\nu\circ\lambda)$, then we say $(\GG,\mu\circ\lambda)$ and $(\HH,\nu\circ\lambda)$ are \emph{measure class preserving measure equivalent} and write $(\GG,\mu\circ\lambda) \sim_{\mcpME} (\HH,\nu\circ\lambda)$.  
\end{defn}

It is shown in \cite[Proposition 4.4]{bddcohomologytransversegroupoid} that measure equivalence is indeed an equivalence relation for discrete Borel groupoids. It turns out that measure class preserving measure equivalence is an equivalence relation for countable p.m.p.\ groupoids. To prove this, we first record a lemma.

\begin{lemma} \label{lem: inv null set}
    Let $\GG\curvearrowright(\Omega,m)$ be a free measure-preserving action of a countable Borel groupoid $\GG$, and $X\subset \Omega$ be a finite-measure strict fundamental domain for this action. If $b\colon \Omega\to S$ is a $\GG$-invariant measurable map, i.e., $b(g\omega)=b(\omega)$ for any composable pair $(g,\omega)\in \GG\times_{\GG^{(0)}}\Omega$, then $b_*m \sim (\restr{b}{X})_*(\restr{m}{X})$, i.e., the measures $b_*m$ and $(\restr{b}{X})_*(\restr{m}{X})$ on $S$ have the same null sets. 

    In particular, when $(\Omega,m)$ is a measure equivalence $(\GG,\HH)$-coupling for two countable p.m.p.\ groupoids $(\GG,\mu\circ\lambda)$ and $(\HH,\nu\circ\lambda)$ and $b = t_{\HH}$ is the anchor map for the $\HH$-action, we have $(t_{\HH})_*m \sim \nu$ on $\HH^{(0)}$ if and only if $(\restr{t_{\HH}}{X})_*(\restr{m}{X}) \sim \nu$ on $\HH^{(0)}$.
\end{lemma}
\begin{proof}
    Let $A\subset S$ be measurable. Since $b$ is $\GG$-invariant, the set $b^{-1}(A)$ is $\GG$-invariant. We claim that $m(b^{-1}(A)) = 0$ if and only if $m(b^{-1}(A)\cap X) = 0$.

    Clearly, if $m(b^{-1}(A))=0$, then $m(b^{-1}(A)\cap X)=0$. Conversely, suppose that $m(b^{-1}(A)\cap X)=0$. Since $X$ is a strict fundamental domain, we have $\GG X = \Omega$. Fix a basis $\mathcal{B}= \{B_i\}_i\subset[[\mathcal G]]$ for the groupoid $\mathcal{G}$, and let $X_i \coloneqq s(B_i)X \subset X$ for each $i\in\mathbb{N}$. Left multiplication by $B_i$ induces a measure-preserving partial transformation $\varphi_i\colon X_i \to \varphi_i(X_i) = B_i X_i \subset \Omega$ such that $\Omega = \GG X = \bigcup_{i\geq 1}B_i X_i = \bigcup_{i\geq 1} \varphi_i(X_i)$.

    Since $b^{-1}(A)$ is $\GG$-invariant, for every $i$ we have
    \[\varphi_i^{-1}\bigl(b^{-1}(A)\cap \varphi_i(X_i)\bigr)=B_i^{-1}\bigl(b^{-1}(A)\cap B_iX_i\bigr) = b^{-1}(A)\cap X_i.\]
    The maps $\varphi_i$ are measure-preserving, so
    \[m\bigl(b^{-1}(A)\cap \varphi_i(X_i)\bigr) = m\bigl(b^{-1}(A)\cap X_i\bigr) \leq m\bigl(b^{-1}(A)\cap X\bigr) =0,\]
    and taking the countable union over $i$, we obtain $m(b^{-1}(A)) = 0$. This proves our claim.

    It follows from our claim that $(b_*m)(A)=0$ if and only if $(\restr{b}{X})_*(\restr{m}{X})(A) = 0$. Since this holds for every measurable $A\subset S$, the measures $b_*m$ and $(\restr{b}{Y})_*(\restr{m}{X})$ have the same null sets.

    For the second statement, we observe that the anchor map $b = t_{\HH}$ for the $\HH$-action is $\GG$-invariant since the $\HH$-action commutes with the $\GG$-action. Hence it follows by applying first statement to $b = t_{\HH}$.
\end{proof}


\begin{prop} \label{prop: mcpme is equiv rel}
    Measure class preserving measure equivalence is an equivalence relation on the class of countable p.m.p.\ groupoids.    
\end{prop}
\begin{proof}
    It is easy to check that the measure equivalence couplings for reflexivity and symmetry constructed in \cite[Proposition 4.4]{bddcohomologytransversegroupoid} are also measure class preserving, so here we only check for transitivity. Suppose $(\GG,\mu_{\GG}\circ\lambda) \sim_{\mcpME} (\HH,\mu_{\HH}\circ\lambda)$ via a measure class preserving measure equivalence coupling $(\Omega, m)$ with anchor maps $a_\GG\colon \Omega\to \GG^{(0)},a_\HH\colon \Omega\to \HH^{(0)}$, and $(\HH,\mu_{\HH}\circ\lambda) \sim_{\mcpME} (\KK,\mu_{\KK}\circ\lambda)$ via a measure class preserving measure equivalence coupling $(\Sigma, n)$ with anchor maps $b_\HH\colon \Sigma\to \HH^{(0)},b_\KK\colon \Sigma\to \KK^{(0)}$. By assumption, we have $(a_\GG)_* m \sim \mu_{\GG},(a_\HH)_* m \sim \mu_{\HH},(b_\HH)_* n \sim \mu_{\HH}$, and $(b_\KK)_* n \sim \mu_{\KK}$. 

    From the composition construction in \cite[Proposition 4.4]{bddcohomologytransversegroupoid}, we consider the fiber product 
    \[\Omega\,{}_{a_\HH}\!\times_{b_\HH}\Sigma = \{(u,v)\in\Omega\times\Sigma: a_\HH(u)=b_\HH(v)\} \subset \Omega\times \Sigma\]
    equipped with the relative product measure $\mu$ defined as follows. Recall that by assumption we have $(a_\HH)_*m\sim\mu_\HH$ and $(b_\HH)_*n\sim\mu_\HH$, so we consider the Radon-Nikodym derivatives
    \[f\coloneqq \frac{\mathrm{d}(a_\HH)_*m}{\mathrm{d}\mu_\HH}, \qquad g\coloneqq \frac{\mathrm{d}(b_\HH)_*n}{\mathrm{d}\mu_\HH},\]
    which satisfy that $0<f(x),g(x)<\infty$ for $\mu_\HH$-almost every $x\in \HH^{(0)}$. We disintegrate $m$ and $n$ as
    \[ m=\int_{\HH^{(0)}} m^x\,\mathrm{d}(a_\HH)_*m(x) = \int_{\HH^{(0)}} f(x)m^x\,\mathrm{d}\mu_H(x), \qquad
    n=\int_{\HH^{(0)}} n^x\,\mathrm{d}(b_\HH)_*n(x) = \int_{\HH^{(0)}} g(x)n^x\,\mathrm{d}\mu_H(x),\]
    where $m^x$ is supported on $a_\HH^{-1}(x)$, $n^x$ is supported on $b_\HH^{-1}(x)$, and $m^x,n^x$ are probability measures for almost every $x$. On the fiber product $\Omega\,{}_{a_\HH}\!\times_{b_\HH}\Sigma$, the relative product measure $\mu$ is given by
    \[\int_{\HH^{(0)}} f(x)g(x)\,(m^x\otimes n^x)\,\mathrm{d}\mu_\HH(x).\]
    More concretely, for every Borel subset $C\subset\Omega\,{}_{a_{\HH}}\!\times_{b_{\HH}}\Sigma$, let $C_x \coloneqq C\cap(a_{\HH}^{-1}(x)\times b_{\HH}^{-1}(x))$ for $x\in\HH^{(0)}$, and define 
    \[\mu(C) \coloneqq \int_{\HH^{(0)}} f(x)g(x)(m^x\otimes n^x)(C_x)\,\mathrm{d}\mu_{\HH}(x).\]
    For example, if $A\subset\Omega$ and $B\subset\Sigma$ are Borel, then
    \[\mu\Bigl((A\times B)\cap\bigl(\Omega\,{}_{a_{\HH}}\!\times_{b_{\HH}}\Sigma\bigr)\Bigr) = \int_{\HH^{(0)}} f(x)g(x)m^x(A)n^x(B) \,\mathrm{d}\mu_{\HH}(x).\]
    The map $x\mapsto(m^x\otimes n^x)(C_x)$ is measurable as it is immediate for the preceding rectangular sets, and the general case follows from a monotone class argument. Since $0<f(x)g(x)<\infty$ almost everywhere, $\mu$ a $\sigma$-finite measure on the fiber product. Equivalently, for every nonnegative Borel function $F$ on the fiber product,
    \[ \int F(u,v)\,\mathrm{d}\mu(u,v)= \int_{\HH^{(0)}} f(x)g(x)\cdot \left( \int_{a_{\HH}^{-1}(x)} \int_{b_{\HH}^{-1}(x)} F(u,v)\,\mathrm{d}n^x(v)\,\mathrm{d}m^x(u) \right)\mathrm{d}\mu_{\HH}(x).\]

    The groupoid $\HH$ acts diagonally on $\Omega\,{}_{a_H}\!\times_{b_H}\Sigma$, and the fused coupling is $(\Omega\,{}_{a_\HH}\!\times_{b_\HH}\Sigma)/\HH$ equipped with the quotient measure $\overline{\mu}$ of $\mu$. By the proof of \cite[Proposition 4.4]{bddcohomologytransversegroupoid}, $((\Omega\,{}_{a_\HH}\!\times_{b_\HH}\Sigma)/\HH, \overline{\mu})$ is a $(\GG,\KK)$ coupling. The anchor maps for the $\GG$-action and the $\KK$-action on the fused coupling are given by $c_\GG([u,v])=a_\GG(u)$ and $c_\KK([u,v])=b_\KK(v)$, respectively. These maps are well-defined because the diagonal $\HH$-action does not change the anchor maps $a_\GG$ and $b_\KK$.

    It remains to check the measure class preserving condition $(c_\GG)_*\overline{\mu} \sim \mu_\GG$ and $(c_\KK)_*\overline{\mu} \sim \mu_\KK$. We prove the first statement as the second is identical. 
    
    Let $E\subset \GG^{(0)}$ be measurable, and define $S_E\coloneqq \{(u,v)\in \Omega\,{}_{a_\HH}\!\times_{b_\HH}\Sigma: a_\GG(u)\in E\}$. The set $S_E$ is $\HH$-invariant, and its quotient is precisely $c_\GG^{-1}(E)$. Hence, 
    $(c_\GG)_*\overline{\mu}(E) =  \overline{\mu}(c_\GG^{-1}(E))=0$ if and only if $\mu(S_E) = 0$. Now compute 
    \begin{equation*}
        \begin{aligned}
            \mu(S_E) &= \int_{\HH^{(0)}} f(x)g(x)\,(m^x\otimes n^x) \bigl(\{(\omega,\sigma):a_\GG(\omega)\in E\}\bigr)\,\mathrm{d}\mu_\HH(x)\\
            &=\int_{\HH^{(0)}}f(x)g(x)\,m^x(a_\GG^{-1}(E))\,n^x(b_\HH^{-1}(x))\,\mathrm{d}\mu_\HH(x)= \int_{\HH^{(0)}}f(x)g(x)\,m^x(a_\GG^{-1}(E))\,\mathrm{d}\mu_\HH(x),
        \end{aligned}
    \end{equation*}
    where the last equality follows from the fact that $n^x$ is a probability measure supported on $b_\HH^{-1}(x)$ for $\mu_{\HH}$-almost every $x\in \HH^{(0)}$. Because $f(x)g(x)>0$ for $\mu_{\HH}$-almost every $x$, $\mu(S_E)$ vanishes if and only if $m^x(a_\GG^{-1}(E))=0$ for $\mu_\HH$-almost every $x$. On the other hand,
    \[m(a_\GG^{-1}(E))=\int_{\HH^{(0)}}f(x)\,m^x(a_\GG^{-1}(E))\,\mathrm{d}\mu_\HH(x),\]
    and since $f(x)>0$ almost everywhere, we have $m(a_\GG^{-1}(E))=0$ if and only if $m^x(a_\GG^{-1}(E))=0$ for $\mu_\HH$-almost every $x$. Therefore, it follows that $\mu(S_E)=0$ if and only if $(a_\GG)_*m(E) = m(a_\GG^{-1}(E))=0$. Combining the preceding equivalence with the equivalence given by $\mu(S_E) = 0$ if and only if $(c_\GG)_*\overline{\mu}(E) =0$, we conclude that $(c_\GG)_*\overline{\mu} \sim (a_\GG)_*m \sim \mu_\GG$, as desired. This finishes the proof for transitivity.
\end{proof}

There are many natural examples of measure class preserving measure equivalence couplings, as we will see in the sequel. At this point we would like to point out that the classical measure equivalence couplings for countable groups regarded as one-object groupoids are always measure class preserving, because unit spaces are points, and every nonzero coupling measure pushes forward to the Dirac measure class.

We next recall the definition of measure equivalence that appeared in \cite{Bernoulligroupoids}. To distinguish it from the above, we would call it stable orbit equivalence for a moment.
\begin{defn}[\protect{\cite[Definition 2.30]{Bernoulligroupoids}}]
    We say that countable measured groupoids $\GG$ and $\HH$ are \emph{stably orbit equivalent} if there exist countable p.m.p.\ groupoids $(\widetilde{\GG},\widetilde{\mu}\circ\lambda)$ and $(\widetilde{\HH},\widetilde{\nu}\circ\lambda)$ with locally bijective Borel groupoid homomorphisms $\widetilde{\GG} \xrightarrow{p}\GG, \widetilde{\HH} \xrightarrow{q}\HH$, and Borel subsets $A\subset \widetilde{\GG}^{(0)}, B\subset \widetilde{\HH}^{(0)}$, respectively, such that $\widetilde{\GG}_A^A\simeq\widetilde{\HH}_B^B$. 

    When $(\GG,\mu\circ\lambda)$ and $(\HH,\nu\circ\lambda)$ are countable p.m.p.\ groupoids, we furthermore say they are \emph{measure class preserving stably orbit equivalent} if the above maps  $p\colon\widetilde{\GG}\to \GG$ and $q\colon \widetilde{\HH}\to \HH$ moreover satisfy $p_*\widetilde{\mu} \sim \mu$ on $\GG^{(0)}$ and $q_*\widetilde{\nu} \sim \nu$ on $\HH^{(0)}$, and the subsets $A\subset \widetilde{\GG}^{(0)}, B\subset \widetilde{\HH}^{(0)}$ are complete unit sections satisfying $(\widetilde{\GG}_A^A,(\widetilde{\mu}\circ\lambda)_A)\simeq (\widetilde{\HH}_B^B,(\widetilde{\nu}\circ\lambda)_B)$.
\end{defn}
Note that our definition is more general than the original given in \cite[Definition 2.30]{Bernoulligroupoids}, where they require the locally bijective Borel groupoid homomorphisms $p,q$ to be measure preserving. Nevertheless, it still follows from \cite[Proposition 2.7]{Bernoulligroupoids} that $\GG$ and $\HH$ are stably orbit equivalent if and only if there exist Borel actions $\GG\acts Y$ and $\HH \acts Z$ such that the transformation groupoids $\GG\ltimes Y$ and $\HH\ltimes Z$ admit complete unit sections $A\subset (\GG\ltimes Y)^{(0)}=Y,\: B\subset (\HH\ltimes Z)^{(0)}=Z$ satisfying $(\GG\ltimes Y)^A_A \simeq (\HH\ltimes Z)^B_B$. In the case when $\GG$ and $\HH$ are countable discrete groups and the actions are free, this is precisely the usual notion of stable orbit equivalence.

As in the group case \cite{Furman99,Gab05}, measure equivalence and stably orbit equivalence are actually equivalent notions for countable measured groupoids. Hence we record following lemmas and propositions that unify the two equivalences and we do not distinguish them in the sequel.
\begin{lemma} \label{lem: ext is me}
    Let $(\widetilde{\GG},\widetilde{\mu}\circ\lambda)$ be a countable p.m.p.\ groupoid and $(\GG,\mu\circ\lambda)$ a countable measured groupoid. If there exists a locally bijective groupoid homomorphism $p\colon \widetilde{\GG}\to \GG$, then $\widetilde{\GG} \sim_{\ME} \GG$.

    Moreover, if $(\GG,\mu\circ\lambda)$ is a countable p.m.p.\ groupoid and $p_*\widetilde{\mu} \sim \mu$ on $\GG^{(0)}$, then $(\widetilde{\GG},\widetilde{\mu}\circ\lambda) \sim_{\mcpME} (\GG,\mu\circ\lambda)$.
\end{lemma}
\begin{proof}
    We claim that $(\Omega,m) = (\widetilde{\GG},\widetilde{\mu}\circ\lambda)$ is a $(\widetilde{\GG},\GG)$-coupling. There is a natural (left) action of $\widetilde{\GG}$ on $\Omega = \widetilde{\GG}$. We now define a commuting (right) action of $\GG$ on $\Omega$ by lifted right multiplication. Let $\omega\in \widetilde{\GG}$ and let $g\in \GG$ be composable on the right with $p(\omega)$, in the sense that $r(g)=p(\widetilde s(\omega))$. By local bijectivity of $p$, there is a unique arrow $\widetilde g\in \widetilde{\GG}$ such that $\widetilde{r}(\widetilde g)=\widetilde s(\omega)$, $p(\widetilde g)=g$, and define the right action by $\omega\cdot g \coloneqq  \omega \widetilde g$. Equivalently, by passing to inverses, this right action may be viewed as a left $G$-action. The left $\widetilde{\GG}$-action and the lifted right $\GG$-action commute, since multiplication in $\widetilde{\GG}$ is associative.

    Both actions are free. Indeed, if $\omega\cdot g=\omega$, then, by definition, $\omega\widetilde g=\omega$. Multiplying on the left by $\omega^{-1}$, we obtain $\widetilde g=1_{\widetilde s(\omega)}$. Hence $g=p(\widetilde g)=1_{p(\widetilde s(\omega))}$, so the lifted $\GG$-action is free. The left $\widetilde{\GG}$-action is free by the same argument.

    Clearly, the unit space $\widetilde{\GG}^{(0)}\subset \widetilde{\GG} = \Omega$ is a finite-measure strict fundamental domain for the $\widetilde{\GG}$-action. It is also a strict finite-measure fundamental domain for the lifted (right) $\GG$-action. Indeed, for $\omega\in \widetilde{\GG}$, take $g=p(\omega)^{-1}$. The unique lift of $g$ with range $\widetilde s(\omega)$ is $\omega^{-1}$, and therefore $\omega\cdot p(\omega)^{-1} = \omega\omega^{-1}=1_{\widetilde r(\omega)}$. The uniqueness of the representation follows from the freeness of $\GG$-action. Thus $\Omega = \widetilde{\GG}$ is a $(\widetilde{\GG},\GG)$-coupling, so $\widetilde{\GG}\sim_{\ME}\GG$.

    The moreover statement follows trivially by taking the same coupling $(\Omega,m) = (\widetilde{\GG},\widetilde{\mu}\circ\lambda)$ and anchor maps $t_{\widetilde\GG}=\widetilde r$ and $t_{\GG}=p\circ \widetilde s$, and by the additional assumption.
\end{proof}
    
\begin{lemma} \label{lem: red is me}
    Let $(\GG,\mu\circ\lambda)$ be a countable p.m.p.\ groupoid and $A \subset \GG^{(0)}$. Then $\GG \sim_{\ME} \GG_A^A$.

    Moreover, if $A \subset \GG^{(0)}$ is a complete unit section, then $(\GG,\mu\circ\lambda) \sim_{\mcpME} (\GG_A^A,\mu_A\circ\lambda)$.
\end{lemma}
\begin{proof}
    We claim that $\Omega\coloneqq \GG_A = \{\omega\in \GG:s(\omega)\in A\}$ equipped with the restriction measure from $\GG$ is a $(\GG,\GG_A^{A})$-coupling. The groupoid $\GG$ acts on $\Omega$ by left multiplication. The reduced groupoid $\GG_A^A$ acts on $\Omega$ by right multiplication given by $\omega\cdot a \coloneqq \omega a$ for $a\in \GG^A_A$. These two actions commute by associativity, and both actions are free.

    The subset $A \subset \GG_A$, identified with the unit arrows over $A$, is a finite-measure strict fundamental domain for the left $\GG$-action since every $\omega\in \GG_A$ is left-equivalent to $1_{s(\omega)}\in A$ under the $\GG$-action.

    For the right $\GG_A^A$-action, observe that the right $\GG_A^A$-orbits in $\GG_A$ are precisely the fibers of the range map $r\colon \GG_A\to\GG^{(0)}$. Since the fibers of $r$ are countable, the Lusin-Novikov uniformization theorem gives a measurable section $X\subset \GG_A$ of $r$, defined modulo null sets. This section meets each right $\GG_A^A$-orbit exactly once, so $X$ is a strict fundamental domain of the right $\GG_A^A$-action. Moreover, $X$ has finite measure, since it contains one arrow over almost every unit of the probability space $\GG^{(0)}$. This proves $\GG_A$ is a $(\GG,\GG_A^A)$-coupling and hence $\GG \sim_{\ME} \GG_A^A$.

    For the moreover statement, we also take the same coupling $\Omega= \GG_A$. We observe that when $A\subset \GG^{(0)}$ is a complete unit section, modulo a null set, every unit of $\GG^{(0)}$ is connected by an arrow of $\GG$ to a unit in $A$, so the range map $r\colon \GG_A\to \GG^{(0)}$ is essentially surjective. 
    Hence there exists a measurable section map $\sigma\colon \GG^{(0)}\to \Omega$ such that $r(\sigma(x))=x$ and $s(\sigma(x)) \in A$ for $x\in \GG^{(0)}$. Note that he image $\sigma(\GG^{(0)})$ meets every right $\GG_A^A$-orbit exactly once since all elements in a right orbit have the same image under $r$. Hence $\sigma(\GG^{(0)}$ is a finite-measure strict fundamental domain for the right $\GG_A^A$-action is naturally identified with $\GG^{(0)}$ under $\sigma$.
    It is then clear by \Cref{lem: inv null set} that $\Omega$ is a measure class preserving coupling for $(\GG,\mu\circ\lambda)$ and $(\GG_A^A,\mu_A\circ\lambda)$.
\end{proof}

We end this section by proving that (measure class preserving) measure equivalence and stably orbit equivalence correspond to the same notion for countable (p.m.p.) measured groupoids. 

\begin{thm} \label{prop: me equiv soe}
    Two countable measured groupoids $(\GG,\mu\circ\lambda)$ and $(\HH,\nu\circ\lambda)$ are measure equivalent if and only if they are stably orbit equivalent.

    If $(\GG,\mu\circ\lambda)$ and $(\HH,\nu\circ\lambda)$ are countable p.m.p.\ groupoids, then they are measure class preserving measure equivalent if and only if they are measure class preserving equivalent stably orbit equivalent.
\end{thm}
\begin{proof}
    Suppose $(\GG,\mu\circ\lambda)$ and $(\HH,\nu\circ\lambda)$ are measure equivalent. Let $(\Omega,m)$ be an ergodic measure equivalence coupling, and $X,Y\subset \Omega$ be the strict fundamental domains of the $\GG$-action and the $\HH$-action, respectively. Let $m_X\coloneqq \frac{\restr{m}{X}}{m(X)}, m_Y\coloneqq \frac{\restr{m}{Y}}{m(Y)}$ be the probability measures on $X,Y$ by normalizing the restriction of $m$ on $X,Y$, respectively. By the remark after Definition \ref{defn:meas equiv}, $\GG$ and $\HH$ admit p.m.p.\ actions on $(Y,m_Y)$ and $(X,m_X)$, respectively. 
    
    Since the anchor maps always extend to locally bijective Borel groupoid homomorphisms from the transformation groupoids onto the acting groupoids, to show that $\GG$ is stably orbit equivalent to $\HH$, it suffices to show there exist subsets $A\subset \GG\ltimes Y$ and $B\subset \HH\ltimes X$ such that $(\GG\ltimes Y)_A^A \simeq (\HH\ltimes X)^B_{B}$. For this, we let $U\coloneqq X\cup Y\subset \Omega$, and note that $m(U)\leq m(X)+m(Y) <\infty$. Consider the restricted transformation groupoid $\NN\coloneqq ((\GG\times \HH)\ltimes\Omega)_U^{U}$ equipped with the normalized restriction measure $m_U\coloneqq \frac{\restr{m}{U}}{m(U)}$, so then $\NN$ is a countable p.m.p.\ groupoid. In fact, $X$ is a complete unit section of $\NN$ because $X$ is a strict fundamental domain of the $\GG$-action, and similarly $Y$ is also a complete unit section of $\NN$. Note that $\NN_Y^Y \simeq \GG\ltimes Y$. Indeed, every arrow of $\NN_Y^Y$ is represented by $((g,h),y)$ with $y\in Y$, $g\in \GG, h\in \HH$, and $(g,h)y\in Y$, and for fixed $g$ and $y$, the strict fundamental domain condition gives a unique $h = \alpha(g,y)$. This means that the map $((g,h),y)\mapsto (g,y)$ identifies $(\NN_Y^Y , \restr{m_U}{Y})$ with the transformation groupoid $(\GG\ltimes Y,m_Y)$, up to normalizing the measures on the unit spaces. Similarly, we have $(\NN_X^X, \restr{m_U}{X})\simeq (\HH\ltimes X,m_X)$. 
    
    The assumption that $X,Y$ are complete unit sections of $\NN$ is equivalent, modulo a $m_U$-null set, to $X,Y$ having positive measure within almost every ergodic component of $\NN$, so we may find complete unit sections $X_0\subset X, Y_0\subset Y$ such that $X_0$ and $Y_0$ have the same positive measure within almost every ergodic component of $\NN$, and hence $\NN_{X_0}^{X_0} \simeq \NN_{Y_0}^{Y_0}$ by \cite[Proposition 2.18]{Bernoulligroupoids}. Thus
    \[(\GG\ltimes Y)_{Y_0}^{Y_0} \simeq (\NN_Y^Y)_{Y_0}^{Y_0} = \NN_{Y_0}^{Y_0} \simeq \NN_{X_0}^{X_0} = (\NN_X^X)_{X_0}^{X_0} \simeq (\HH\ltimes X)_{X_0}^{X_0},\]
    which proves that $\GG$ is stably orbit equivalent to $\HH$.

    Conversely, suppose $\GG$ and $\HH$ are stably orbit equivalent. There exist countable p.m.p.\ groupoids $\widetilde{\GG}$ and $\widetilde{\HH}$ with locally bijective Borel groupoid homomorphisms $\widetilde{\GG} \xrightarrow{p}\GG, \widetilde{\HH} \xrightarrow{q}\HH$, and subsets $A\subset \widetilde{\GG}^{(0)},\: B\subset \widetilde{\HH}^{(0)}$, respectively, such that $\widetilde{\GG}_A^A \simeq \widetilde{\HH}^B_{B}$. By \cite[Proposition 4.4]{bddcohomologytransversegroupoid}, measure equivalence is transitive, so by the previous two lemmas we have
    \begin{equation} \label{eqn: chain of me}
        \GG \,\sim_{\ME}\, \widetilde{\GG} \,\sim_{\ME} \,\widetilde{\GG}_A^A \,\simeq\, \widetilde{\HH}_B^B \,\sim_{\ME} \,\widetilde{\HH} \,\sim_{\ME} \,\HH,
    \end{equation}
    so $\GG$ is measure equivalent to $\HH$.

    Now we check for the measure class preserving case. The forward direction follows by the exactly same argument as above. Indeed, it follows since the fundamental domains $X,Y$ are complete unit sections of the transformation groupoids $\HH\ltimes X$ and $\GG\ltimes Y$, respectively, and by \Cref{lem: inv null set}, when $(\Omega,m)$ is a measure class preserving coupling, we know the pushforward of the restricted measures on $X,Y$ by the anchor maps are in the same measure class of $\nu$ on $\HH^{(0)}$ and $\mu$ on $\GG^{(0)}$, respectively. The converse also follows by the same proof as above, as one can use the moreover statements in the previous two lemmas and transitivity of measure class preserving measure equivalence in \Cref{prop: mcpme is equiv rel} to check that $\sim_{\ME}$ in (\ref{eqn: chain of me}) can be replaced by $\sim_{\mcpME}$ in this case.
\end{proof}

\end{appendices}

\bibliographystyle{amsalpha}
\bibliography{bib}
\end{document}